\documentclass[12pt]{amsart}
\usepackage{url}

\usepackage{natbib}

\usepackage{amsmath,amsxtra,amssymb,latexsym,epsfig,amscd,amsthm,fancybox,epsfig}
\usepackage[mathscr]{eucal}
\usepackage{graphicx}
\usepackage{subcaption}
\usepackage{wrapfig}
\usepackage{multicol,xcolor}
\usepackage{epsfig} 
\usepackage{epstopdf}
\usepackage{cases}
\usepackage{color}
\usepackage{hyperref}
\usepackage{bigints}
\usepackage{hyperref}
\usepackage{subcaption}
\hypersetup{
    colorlinks=true,
    linkcolor=blue,
    filecolor=magenta,
    urlcolor=cyan,
    citecolor=blue,
}
\newtheorem{thm}{Theorem}[section]
\newtheorem {asp}{Assumption}[section]
\newtheorem{lm}{Lemma}[section]
\newtheorem{rmk}{Remark}[section]

\newtheorem{deff}{Definition}[section]
\newtheorem{clm}{Claim}[section]

\newtheorem{prop}{Proposition}[section]
\theoremstyle{definition}

\theoremstyle{remark}

\numberwithin{equation}{section}

\DeclareMathOperator{\suppo}{supp}
\DeclareMathOperator{\Conv}{Conv}
\newcommand{\eps}{\varepsilon}

\newcommand{\M}{\mathcal{M}}
\newcommand{\F}{\mathcal{F}}

\newcommand{\E}{\mathbb{E}}

\newcommand{\BX}{\mathbf{X}}
\newcommand{\bx}{\mathbf{x}}

\newcommand{\bc}{\mathbf{c}}

\newcommand{\bp}{\mathbf{p}}
\newcommand{\bq}{\mathbf{q}}
\newcommand{\bz}{\mathbf{z}}

\newcommand{\N}{\mathbb{N}}
\newcommand{\CN}{\mathcal{N}}
\newcommand{\Se}{\mathcal{S}}
\newcommand{\PP}{\mathbb{P}}

\newcommand{\K}{\mathcal{K}}

\newcommand{\R}{\mathbb{R}}

\newcommand{\Lom}{\mathcal{L}}
\newcommand{\U}{\mathcal{U}}

\newcommand{\wtd}{\widetilde}
\numberwithin{equation}{section}

\newcommand{\bed}{\begin{displaymath}}
\newcommand{\eed}{\end{displaymath}}
\newcommand{\bea}{\bed\begin{array}{rl}}
\newcommand{\eea}{\end{array}\eed}

\newcommand{\barray}{\begin{array}{ll}}
\newcommand{\earray}{\end{array}}

\def\disp{\displaystyle}

\newcommand{\1}{\boldsymbol{1}}
\newcommand{\0}{\boldsymbol{0}}
\newcommand{\bdelta}{\boldsymbol{\delta}}

\newcommand{\dist}{\mathrm{dist}}

\def\bar{\overline}
\def\hat{\widehat}
\def\a.s{\text{\;a.s.\;}}

\title[Population dynamics]{Population dynamics under random switching}

\author[A. Hening]{Alexandru Hening }
\address{Department of Mathematics\\
Texas A\&M University\\
Mailstop 3368\\
College Station, TX 77843-3368\\
United States
}
\email{ahening@tamu.edu}

 \author[S. Sabharwal]{Siddharth Sabharwal}
 \address{Department of Mathematics\\
Texas A\&M University\\
Mailstop 3368\\
College Station, TX 77843-3368\\
United States
}
\email{siddhutifr93@tamu.edu }


\begin{document}
\begin{abstract}
Populations interact non-linearly and are influenced by environmental fluctuations. In order to have realistic mathematical models, one needs to take into account that the environmental fluctuations are inherently stochastic. Often, environmental stochasticity is modeled by systems of stochastic differential equations. However, this type of stochasticity is not always the best suited for ecological modeling. Instead, biological systems can be modeled using piecewise deterministic Markov processes (PDMP). For a PDMP the process follows the flow of a system of ordinary differential equations for a random time, after which the environment switches to a different state, where the dynamics is given by a different system of differential equations. Then this is repeated. The current paper is devoted to the study of the dynamics of $n$ populations described by $n$-dimensional Kolmogorov PDMP. We provide sharp conditions for persistence and extinction, based on the invasion rates (Lyapunov exponents) of the ergodic probability measures supported on the boundary of the positive orthant $\partial \R_+^{n,\circ}$. In order to showcase the applicability of our results, we apply the theory in some interesting ecological examples.

\end{abstract}
\maketitle
\tableofcontents
\section{Introduction}

It is almost never the case that a species is isolated and does not get to interact with other species in its ecosystem. As such, one of the fundamental concerns of ecology is trying to see when interacting species coexist or, if coexistence of all species is not possible, determining exactly which species persist and which go extinct.

One cannot ignore the effects of the environment on species interactions. Given that the environmental conditions fluctuate, usually in a random and unpredictable fashion, it is key to have robust mathematical models that can include both random environmental fluctuations and the nonlinear intraspecies and interspecies interactions.
 
There are examples when biotic effects can result in species going extinct and the effects of random environmental fluctuations lead to a reversal of extinction into coexistence. In other situations, deterministic systems in which all species persist lead to extinctions when one takes into account environmental stochasticity. Recently, there have been important results and breakthroughs by modeling the dynamics using discrete or continuous time Markov processes and analyzing their asymptotic behavior (\cite{ C00, ERSS13, EHS15, LES03, SLS09, SBA11, BEM07, BS09, BHS08, CM10, CCA09, FS24, B23}).

One way of analyzing the coexistence of species is by computing the average per-capita growth rate of a population when rare. If this growth rate is positive the respective population tends to increase when rare, while if it is negative the population tends to decrease and goes extinct. If the ecosystem consists of only two populations then coexistence is ensured if each population can invade when it is rare and the other population is stationary (\cite{T77, CE89, EHS15}).

There is a well established general theory for coexistence for deterministic models (\cite{H81, H84, HJ89}).  Starting with the important work \cite{H81} it has been shown that a sufficient condition for persistence is the existence of a fixed set of weights associated with the interacting populations such that this weighted combination of the population's invasion rates is positive for any invariant measure supported by the boundary.

Recently there has been a flurry of papers studying the effects of environmental stochasticity on continuous-time SDE models. In \cite{BHS08} the authors found that if a deterministic continuous-time model satisfies the above persistence criterion then under some weak assumptions the corresponding stochastic differential equation with a small diffusion term has a positive stationary distribution concentrated on the positive global attractor of the deterministic system. For general stochastic differential equations with arbitrary levels of noise sufficient conditions for persistence and extinction can be found in \cite{SBA11, HN18, HNC21, HNS22, B23, FS24}.

The aim of this paper is to prove similar results for piecewise deterministic Markov processes (PDMP). Some partial results have been obtained for simple models with two or three species (\cite{BL16, HS19, HN20, HNHW23}). However, in these cases there are only two or three ergodic invariant probability measures on the boundary and as such the analysis is simpler. General persistence results for compact state spaces can be found in \cite{B23}. In this paper we prove persistence and extinction results for compact and non-compact state spaces and therefore provide a significant advancement in the field of stochastic population dynamics.

Piecewise deterministic Markov processes have been used recently to prove some very interesting facts about biological populations. For example in \cite{BL16, HN20} the authors look at a two dimensional Lotka-Volterra system in a fluctuating environment. They show that the random switching between two environments that are both favorable to the same species can lead to the extinction of this favored species or to the coexistence of the two competing species. Our analysis generalizes the framework of \cite{BL16} to higher-dimensional nonlinear systems that do not necessarily live in a compact set.

The paper is organized as follows. In section \ref{s:pdmp} we present our framework, the problems we study, the different assumptions and the main results.  In Section \ref{s:app} we showcase some interesting examples. 
In Section \ref{s:prep} we prove some preliminary results regarding the process. For example, we show that the process is Markov-Feller, that, that there is a well-defined strong solution $(\BX(t),r(t))$ for all $t>0$, that this solution is pathwise unique and that there is tightness for certain families of (random) measures. Section \ref{s:perm} is devoted to the study of conditions under which $(\BX(t),r(t))$ exhibits persistence. In particular, we prove when the process converges to its unique invariant probability measure on $\R_+^{n,\circ}\times \CN$ where $\R_+^{n,\circ}:=(0,\infty)^n$ and $\CN=\{1,\dots,n_0\}$. In Theorem \ref{thm3.1} we show that, under certain assumptions, the transition probability of $(\BX(t),r(t))$ converges in total variation to the a unique invariant probability measure of $(\BX(t),r(t))$ on $\R_+^{n,\circ}\times\CN$ exponentially fast. In Section \ref{s:extin} we analyze when one or multiple species go extinct. First, we show in Theorem \ref{thm4.1} that if there exists an invariant probability measure living on the boundary that is an \textit{attractor}, and the boundary is accessible, then the process converges almost surely to the boundary in a weak sense. If we do not know that the boundary is accessible, we prove in Theorem \ref{t:exxx} that the probability of converging exponentially fast to an attracting subspace of the boundary can be made arbitrarily close to one by starting the process close enough to that subspace. Under a few extra assumptions we show in Theorem \ref{thm4.2} that for every accessible attractor $\mu$ on the boundary the process converges with strictly positive probability to $\mu$.
Proofs of certain lemmas from Section \ref{s:prep} and Section \ref{s:extin} appear in Appendix \ref{a:inv measures} and Appendix \ref{extinction lemmas}.

We want to mention that this work is indebted to \cite{B23} which has general results for stochastic persistence, without treating the extinction cases, and looks in depth at PDMP in compact state spaces. During the writing of this paper we also became aware of the nice extinction results from \cite{FS24}, which make use of a different approach.

\section{Piecewise Deterministic Markov Processes}\label{s:pdmp}

In order to take into account environmental fluctuations and their effect on the persistence or extinction of species, one approach is to study the uniform persistence for non-autonomous differential equations \citep{T00, ST11, MSZ04}. A different approach is to consider systems that have random environmental perturbations. One way to do this is by studying stochastic differential equations \citep{SBA11, HN18, HNC21, HNS22}. Another possibility is to look at stochastic equations driven by a Markov chain. These are piecewise deterministic Markov processes of the form
\begin{equation}\label{e1-pdm}
\frac{dX_i}{dt}(t)=X_i(t)f_i(\BX(t),r(t)), i=1,\dots,n
\end{equation}
where $r(t)$ is a cadlag process taking values in a finite state space
$\CN=\{1,\dots,n_0\}.$ Here $\BX(t):=(X_1(t),\dots,X_n(t))$ are the species densities at time $t\geq 0$ and $f_i(\bx,k)$ quantifies the fitness of species $i$ when the densities are $\bx=(x_1,\dots,x_n)$ and the environmental state is $k$.

The switching intensity of $r(t)$ can depend on the state of $\BX(t)$ and is given by

\begin{equation}\label{e:tran}\begin{array}{ll}
&\disp \PP\{r(t+\Delta)=j~|~r(t)=i, \BX(s),r(s), s\leq t\}=q_{ij}(\BX(t))\Delta+o(\Delta) \text{ if } i\ne j \
\hbox{ and }\\
&\disp \PP\{r(t+\Delta)=i~|~r(t)=i, \BX(s),r(s), s\leq t\}=(1+q_{ii}(\BX(t))\Delta+o(\Delta).\end{array}\end{equation}
The above equation says that the probability that the process $r(t)$ jumps from $r(t)=i$ to the state $j\neq i$ in the small time $\Delta$ is approximately equal to $q_{ij}(\BX(t))\Delta$.

Here $q_{ii}(\bx):=-\sum_{j\ne i}q_{ij}(\bx)$ quantifies the probability of staying in the same state $i$. We will assume that the function $q_{ij}(\bx,k)$ is a bounded continuous function for each $i,j\in\CN$
and $Q(\bx)=(q_{ij}(\bx))_{n_0\times n_0}$ is irreducible for all $\bx\in\R^n_+$.  

It is well-known \citep{davis1984piecewise, BBMZ15, B23} that a process $(\BX(t),r(t))$ satisfying \eqref{e1-pdm} and \eqref{e:tran}
is a Markov process with generator acting on functions $G:\R_+^n\times\CN\mapsto\R$ that are continuously differentiable in $\bx$ for each $k\in\CN$ as
\begin{equation}\label{e:gen}
\Lom G(\bx, k)=\sum_{i=1}^n x_if_i(\bx,k)\frac{\partial G}{\partial x_i}(\bx,k)+\sum_{l\in\CN}q_{kl}(\bx)G(\bx,l).
\end{equation}

Important theoretical work analyzing PDMP has been done recently. Under various irreducibility and H{\"o}rmander type conditions \cite{BH12, BBMZ15} have found principles under which there exists a unique invariant probability measure for certain PDMP. The convergence in total variation, at an exponential rate, to this unique invariant measure has been shown when the state space is compact \citep{BBMZ15} and in the one-dimensional case there are results about the smoothness of the density of the invariant measure \citep{bakhtin2018smooth}. If there is no irreducibility, one can use an associated deterministic control system to find the support of invariant probability measures \citep{benaim2017supports, B23}. The asymptotic behavior of a PDMP his not determined just by the vector fields $f_i(\bx,k)$, but can drastically change according to the switching rates $q_{ij}(\bx)$. There are many interesting examples of how switching leads to rich, and sometimes counterintuitive, dynamical behavior \citep{BBMZ15, BL16, HS19, HN20, benaim2019random}.

We are interested in proving persistence and extinction results for the system \eqref{e1-pdm}. In order to accomplish this we modify and extend the methods from \cite{HN18, HNC21}, which treated  stochastic differential equations, while also making use of results about PDMP \cite{BBMZ15} and stochastic persistence \cite{B23}. We have not used the methods from \cite{FS24}, an approach that could also yield extinction results for PDMP.

Let 
\[
\Se_0 := \{(\bx,k)\in \R_+^n\times \CN~:~\min_i x_i = 0\}
\]
be the extinction set and for $\eta>0$ define the $\eta$-neighborhood of the extinction set
\[
\Se_\eta := \{(\bx,k)\in \R_+^n\times \CN~:~\min_i x_i \leq \eta\}.
\]
We will denote by $\Se_\eta^c$ the complement of $\Se_\eta$ in $\R_+^n$. We also define $\partial \R_+^n:= \{~\min_i x_i = 0\}$ so that $\Se_0 = \partial \R_+^n \times \CN$.

We next explain what we mean by persistence and extinction in our setting. The first definition of persistence, due to \cite{chesson1982stabilizing}, says that the probability that species are uniformly bounded away from extinction, $X_i\geq \delta$, as time increases, can be arbitrarily close to 1 in the sense that no matter how small one picks $\eps>0$, the probability can be made greater or equal to $1-\eps$. Moreover, the bound $\delta = \delta(\eps)>0$ does not depend on the initial conditions.
\begin{deff}
The process $\BX$ is stochastically persistent in probability if for any $\eps>0$, there exists $\delta>0$ such that
\begin{equation}
\liminf\limits_{t\to\infty} \PP_{\bx,k}\left\{X_i(t)\geq\delta, i=1,\dots,n\right\}>1-\eps, \;\mathbf{x}\in\R^{n,\circ}_+.
\end{equation}
\end{deff}
For the next definition of persistence we will use the randomized occupation measures 
$$\wtd \Pi_t(\cdot,k):=\dfrac1t\int_0^t\1_{\{\BX(s)\in\cdot, r(s)=k\}}ds,\,\,k\in\CN,t>0.$$
If $A$ is a measurable set then $\wtd \Pi_t(A,k)$ measures the fraction of the time $\BX(t)$ spends in the set $A$ and in environment $k$ up to time $t$. 
\begin{deff}
The process $(\BX(t),r(t))$ is almost surely stochastically persistent if for all $\eps>0$ there exists $\delta>0$ such that with probability one
\[
\liminf_{t\to\infty} \wtd \Pi_t(\Se_\delta^c)>1-\eps.
\]
\end{deff}
This definition says that the fraction of time spent by all species densities above the threshold $\delta$ converges to $1$ as $\delta\to 0$. We next look at extinction. First, note that extinction can only happen asymptotically as $t\to\infty$ in our setting as for any $\bx=(x_1,\dots,x_n)$ with $x_i>0, i=1,\dots,n$ we will have 
$$\PP_{\bx,k}\left\{X_i(t)>0, t\geq 0\right\}=1.$$
See Lemma \ref{lm2.0} for a proof of this fact.
\begin{deff}
If $(\BX(0), r(0))=(\bx,k)\in \R^{n,\circ}_+\times\CN $ we say the population $X_i$ goes extinct with probability $p_{\bx,k}>0$ if
\[
\PP_{\bx,k}\left\{\lim_{t\to\infty}X_i(t)=0\right\}=p_{\bx,k}.
\]
We say the population $X_i$ goes extinct if for all $\bx\in\R^{n,\circ}_+, k\in\CN$
\[
\PP_{\bx,k}\left\{\lim_{t\to\infty}X_i(t)=0\right\}=1.
\]
\end{deff}

We will make the following assumptions throughout the paper.

\begin{asp}\label{a1-pdm}
The transition rate $q_{ij}(\bx)$ is a bounded continuous function for each $i,j\in\CN$ and $Q(\bx)=(q_{ij}(\bx))_{n_0\times n_0}$ is irreducible for all $\bx\in\R^n_+$. 
	The functions $f_i(\cdot):\R^n_+\to\R, i=1,\dots,n$ are locally Lipschitz and there is a function $F(\bx,k):\R^n\times\CN\mapsto\R$ that is continuously differentiable in $\bx$ for each $k\in\CN$. Moreover, there exist constants $c,\delta_0, \gamma_b>0$ such that for every $k\in\CN$
	\begin{equation}\label{e:F_growth}
	    F(\bx,k)\geq c(1+\|\bx\|),~ \bx\in \R_+^{n,\circ}, 
	\end{equation}
and    
	\begin{equation}\label{e:sup}
	\limsup_{\|\bx\|\to\infty}\left[\frac{\Lom F(\bx,k)}{F(\bx,k)}+\gamma_b+ \gamma_b\left(\max_{i}\{|f_i(\bx,k)|\}\right)\right]< 0.
	\end{equation}
	
\end{asp}

\begin{rmk}
Note that \eqref{e:sup} is equivalent to the weaker condition
    \begin{equation}\label{e:sup2}
	\limsup_{\|\bx\|\to\infty}\left[\frac{\Lom F(\bx,k)}{F(\bx,k)}+ \delta_0\left(\max_{i}\{|f_i(\bx,k)|\}\right)\right]< 0.
	\end{equation}

\end{rmk}
\begin{rmk}
We note that by the continuity of $F$ we can extend \eqref{e:F_growth} to hold in any subspace.  If we take, for example, $\bx=(\bar \bx, \bx_n)$, let $x_n>0$ and $\bar \bx \in \R^{n-1,\circ}_+$
then since we have for all $\bx_n>0$
 $$\frac{F(\bar\bx, \bx_n,k)}{(1+\|(\bar \bx, \bx_n)\|)}\geq c$$
by letting $\bx_n\to 0$ and using the continuity of the left-hand side we will get that \eqref{e:F_growth} holds on $\R^{n-1,\circ}\times \{0\}$.
\end{rmk}

\begin{rmk}
     The growth assumption \eqref{e:F_growth} can be relaxed to being able to find some $0<\bar \delta<1$ such that $$F(\bx,k)\geq c(1+\|\bx\|)^{\bar\delta},~ \bx\in \R_+^{n,\circ}.$$
\end{rmk}
\begin{rmk}
	If we assume that there is
	a vector $\bc\in\R^{n,\circ}_+$ such that $\Lom (1+\bc^T\bx)=\sum_ic_ix_if(x_i,k)<0$ when $\bx$ is sufficiently large for any $k\in\CN$,
	then the solution will eventually enter a compact set and never leave it.
	In that case we would only work on a compact set and things would be easier. We note that this situation is satisfied in many biologically relevant examples. 
    
Assumption \ref{a1-pdm} is such that for some states, the solution may go to infinity - the solution $\BX$ is tight without requiring $\BX$ to live in a compact set. See the examples from Subsections \ref{s:1d_explosion} and \ref{s:2d_explosion}.
\end{rmk}
We also need to enforce the following assumption.
\begin{asp}\label{a:bound}
Assume that there exists $M_F>0$ such that for all $\bx\in \R^n_+$ and all $k,l\in \CN$ 
\[
\frac{F(\bx,k)}{F(\bx, l)}\leq M_F.
\]
\end{asp}
\begin{rmk}
Note that the above also implies there is a strictly positive lower bound, as clearly we get that  $\bx\in \R^n_+$ and all $k,l\in \CN$ 
\[
0<\frac{1}{M_F}\leq\frac{F(\bx,k)}{F(\bx, l)}\leq M_F.
\]
\end{rmk}
\begin{rmk}
Assumption \ref{a:bound} ensures that we can use the strong law of large numbers for certain martingales in order to show that almost surely
\[
\lim_{t\to\infty}\frac{M(t)}{t} = 0.
\]
In specific problems and models weaker assumptions could be found. It is useful to keep in mind the results of \cite{liptser1980strong} when trying to prove strong laws of large numbers for local martingales.
\end{rmk}
Since piecewise deterministic Markov processes can be quite degenerate, it is not always clear that every point from the state space or from the boundary is accessible nor that invariant probability measures always exist. We define certain concepts related to which states are accessible or reachable by the process, as well as to the existence of unique invariant probability measures in certain subspaces. 

For any fixed $k\in\CN$ let $\Phi_t^k(\cdot)$ be the flow associated with
the ODE
\begin{equation}\label{e:ODE}
    \frac{dX_i}{dt}(t)=X_i(t)f(X_i(t),k),i=1,\dots,n.
\end{equation}
That is, $\Phi_t^k(\bx,k)$
is the solution at time $t$
of the ODE \eqref{e:ODE}
with initial value $\BX(0)=\bx$. Define the orbit set
$$\gamma_+(\bx,k)=\left\{\phi_{t_n}^{k_m}\circ\cdots\circ\phi_{t_1}^{k_1}(\bx,k):, m\in \N, t_l\geq 0, k_l\in\CN: l=1,\dots,m\right\}
$$
and its closure by $\bar \gamma_+(\bx,k)$. The points from $\gamma^+(\bx,k)$ are those which are reachable from a concatenation of paths that are solutions to ODEs of the form \eqref{e:ODE}.
We denote the intersection of the closures of all the orbit sets by 
\begin{equation}\label{e:Gamma}
\Gamma=\bigcap_{\bx\in\R^{n,\circ}_+} \bar{\gamma_+(\bx,k)}.
\end{equation}
These are the points which are reachable, or infinitely close to being reachable, using concatenations of paths from fixed environments, from all starting points $\bx\in \R^{n,\circ}$.
Let $\M$ denote the set of ergodic invariant probability measures of $(\BX(t),r(t))$ whose support is contained in
$\partial\R^n_+\times\CN$. The convex combinations of the ergodic measures will be invariant probability measures living on $\partial\R^n_+\times\CN$ - this set will be denoted by $\Conv(\M)$. Note that by the results from \cite{B23} the number of ergodic measures, $|\M|$, is always finite.

For any species $i$  define $\Se^i:=\{\bx\in \R_+^n~:~ x_i>0\}$ to be the subset of the state space for which species $i$ has a strictly positive density. Let $\mu$ be an ergodic measure. Since $\Se^i$ is an invariant set, that is, if $\BX(0)\in \Se^i$ then almost surely $\BX(t)\in \Se^i, t\geq 0$, it follows that $\mu(\Se^i)\in \{0,1\}.$ As a result one can define the species support of $\mu$ by
\begin{equation}\label{e:species_support}
I_\mu : =\{\{i_1,\dots,i_u\} ~:~ \mu(\Se^{i_l})=1, l=1,\dots,u \}.
\end{equation}

\begin{deff}
For any invariant probability measure $\mu$ we define the invasion rate of the $i$th species, also called the external Lyapunov exponent, by
\begin{equation}\label{e:invasion_rate}
\lambda_i(\mu):= \sum_{k\in\CN}\int_{\partial\R^n_+}f_i(\bx, k)\mu(d\bx,k)   
\end{equation}
\end{deff}

\begin{rmk}
The intuition behind an invasion rate is the following. By It\^o's formula we have
\[
\frac{\ln X_i(t)}{t} = \frac{\ln X_i(0)}{t} + \frac{1}{t}\int_0^t f_i(\BX(s),r(s))\,ds.
\]
If $\BX(t)$ is close to the support of an ergodic invariant measure $\mu$ for a long time, then by ergodicity and the Feller property the quantity
\[
\frac{1}{t}\int_0^t f_i(\BX(s),r(s))\,ds
\]
can be approximated with the average with respect to $\mu$
\[
\lambda_i(\mu):= \sum_{k\in\CN}\int_{\partial\R^n_+}f_i(\bx, k)\mu(d\bx,k)   
\]
while for $t\to\infty$
\[
\frac{\ln X_i(0)}{t} \to 0.
\]
One can also show that $\lambda_i(\mu)$ will give us the long-term log-growth rate of $X_i(t)$ if $\BX(t)$ is close to the support of $\mu$. This shows why $\lambda_i(\mu)$ can be called the Lyapunov exponents of $\mu$. 
\end{rmk}

We note in particular that by Lemma \ref{lm2.3} below, the invasion rates are well-defined. The following assumption guarantees the persistence of all the species. 
\begin{asp}\label{a2-pdm}
	For any $\mu\in\Conv(\M)$,
	$$\max_{i=1,\dots,n}\left\{\sum_{k\in\CN}\int_{\partial\R^n_+}f_i(\bx, k)\mu(d\bx,k)\right\}>0.$$
\end{asp}
\begin{rmk}\label{r:ergodic_minmax}
Loosely speaking this condition ensures that the boundary $\partial \R_+^n$ will repel the process, or that all the invariant probability measures living on the boundary will be repelers. Biologically this means that for any invariant measure living on the boundary there is at least one species $i$ which has a positive invasion rate. 

The minmax principle shows \citep{SBA11,B23} that Assumption \ref{a2-pdm} is equivalent to the existence of positive weights $\mathbf p = (p_1,\dots,p_n)>0$ such that
\begin{equation}\label{e.pp}
\min\limits_{\mu\in\M}\left\{\sum_{i=1}^n p_i\lambda_i(\mu)\right\}>0.
\end{equation}
This is very useful because Assumption \ref{a2-pdm} has to be checked for infinitely many measures, since usually $|\Conv(\M)|=\infty$ while \eqref{e.pp} only has to be checked for finitely many ergodic measures since $|\M|<\infty.$
\end{rmk}

\begin{deff}
\label{def:doeblin}
We say that $\bx^* \in\R_+^{n,\circ}$  is a Doeblin point if there exist a neighborhood $U$ of $\bx^* $, a non zero  measure $\xi$ on $\R_+^{n,\circ}\times \CN$,  and $T > 0$  such that for all $(\bx,k) \in U$ we have
\begin{equation}
\label{eq:Doeblin}
P_{T}(\bx,k,\cdot)  \geq \xi(\cdot),
\end{equation}
where $P_{T}(\bx,k,\cdot)$ is the transition probability function of the process $(\BX(t),r(t))$. 
\end{deff}
If a Doeblin point exists one can show \citep{BBMZ15, B23} that there can be at most one invariant probability measure on $\R_+^{n,\circ}\times\CN$.

Let $G^i(\bx):=(x_1f_1(\bx,i),\dots,x_nf_n(\bx,i))$ and $\mathcal{F}_0(\bx) = \{G^i(\bx)-G^j(\bx), i,j\in \CN\}$ and define recursively $ \mathcal{F}_k(\bx) = \mathcal{F}_{k-1}(\bx)\cup \{[G^i, V](\bx), V\in  \mathcal{F}_{k-1}(\bx), j\in \CN\}$ where $[X,Y]$ stands for the Lie bracket of two smooth vector vector fields $V, W: \R_+^n\to\R_+^n$ defined for $\bx\in \R_+^n$ by
\[
[V,W](\bx):= DW(\bx) V(\bx) - DV(\bx)W(\bx)
\]
where $DV(\bx)$ is the Jacobian of $V$ at the point $\bx$.
We let $ \mathcal{F}_k (\bx)$ be the vector field in $\R_+^n$ spanned by $\{V(\bx), V\in  \mathcal{F}_k\}$. 

 We say that \eqref{e1-pdm} satisfies the \textbf{strong bracket condition} or the \textbf{strong H{\"o}rmander condition} if there is $\bx_0\in \R_+^{n,\circ}$ such that 
\begin{equation}\label{e:Horm}
\text{span} (\F_k(\bx_0)) = \R^n
\end{equation}
for some $k\in \N$.

If we want to make sure there is a stronger type of coexistence, in the sense that the system converges to a unique invariant probability measure where all species are present, we need the following technical assumption.
\begin{asp}\label{a3-pdm}
	The set $\Gamma$ from \eqref{e:Gamma} is nonempty and
	there exists an $\bx_0\in\Gamma$ that satisfies the strong bracket condition \eqref{e:Horm}.
\end{asp}
We can now state our main persistence result.

\begin{thm}\label{thm5.1}
Suppose that Assumptions \ref{a1-pdm}, \ref{a:bound} and \ref{a2-pdm} hold.
There exist $\theta>0$, $T^*>0$ $n^*\in \N$ and $\kappa=\kappa(\theta,T^*)\in(0,1)$, $\tilde K=\tilde K(\theta,T^*)>0$   such that
\begin{equation}\label{e:lya2}
\E_{\bx,k} V^\theta(\BX(n^*T^*), r(n^*T^*))\leq \kappa V^\theta(\bx,k)+\tilde K\,\text{ for all }\, (\bx,k)\in\R^{n,\circ}_+\times\CN.
\end{equation}
As a result,
$(\BX(t),r(t))$ is stochastically persistent in probability and almost surely stochastically persistent. 

In addition, if Assumption \ref{a3-pdm} is satisfied, then there is a unique probability measure $\pi^*$ on $\R^{n,\circ}_+\times\CN$ and the convergence of the transition probability of $(\BX(t),r(t))$ in total variation to $\pi^*$ on $\R^{n,\circ}_+$ is
exponentially fast, that is, there exists $\gamma_D>0$ such that
\begin{equation}
	\lim\limits_{t\to\infty} e^{\gamma_D t}\|P_{(\BX,r)}(t, \mathbf{x},k, \cdot)-\pi^*(\cdot)\|_{\text{TV}}=0, \;\mathbf{x}\in\R^{n,\circ}_+, k\in\CN.
	\end{equation}
For any initial value $\mathbf{x}\in\R^{n,\circ}_+, k\in\CN$ and any $\pi^*$-integrable function $f$ we have
\begin{equation}\label{slln2}
\PP_{\bx,k}\left\{\lim\limits_{T\to\infty}\dfrac1T\int_0^Tf\left(\BX(t),r(t)\right)dt=\sum_{k\in\CN}\int_{\R_+^{n,\circ}}f(\mathbf{u},k)\pi^*(d\mathbf{u},k)\right\}=1.
\end{equation}

\end{thm}

We next give some more definitions which are required for the extinction results. 
For a subset $I$ of $\{1,\dots,n\}$, denote 
$I^c:=\{1,\dots,n\}\setminus I$ and define
$$
\R_+^{I}:=\left\{\bx\in\R_+: x_i=0\text{ if }i\in I^c \right\},
$$
$$
\R_+^{I,\circ}:=\left\{\bx\in\R^n_+:x_i=0\text{ if }i\in I^c \text{ and }  x_i>0 ~\text{ if }i\in I\right\},
$$
and
$$\partial \R_+^I:=\left\{\bx=(x_1,\dots,x_n)\in\R_+^n:  x_i=0\text{ if }i\in I^c\text{ and } x_i=0\text{ for at least one }i\in I\right\}.$$
Note that for the case $I=\emptyset$,
$\R_+^I=\R_+^{I,\circ}=\{\0\}.$
Denote by
$\M^I, \M^{I,\circ}, \partial M^I$  the sets of ergodic measures on $\R_+^I,\R_+^{I,\circ}$
and $\partial \R_+^I$ respectively.

Consider $\pi\in\M\setminus\{\bdelta^*\}$.
We note that by the definition of the species support $I_\pi$, given by \eqref{e:species_support}, this is the minimal subset of $\{1,\dots,n\}$,  such that
$\pi\left(\R_+^{I_\pi,\circ}\right)=1$. We can then define $\R_+^\pi:=\R_+^{I_\pi}$ to be the subspace which supports the ergodic measure $\pi$.

Next, we look at some conditions that are needed for the extinction results.


\begin{asp}\label{a4-pdm}
	There exists a subset $I\subset\{1,\dots,n\}$ such that
	\begin{equation}\label{ae3.1}
	\max_{i\in I_\pi^c,\; \pi\in\M^{I,\circ}}\{\lambda_i(\pi)\}<0.
	\end{equation}
	If $I\neq \emptyset$, we assume further that
	\begin{equation}\label{ae3.2}
	\max_{i\in I}\{\lambda_i(\nu)\}>0,
	\end{equation}
	for any $\nu\in \Conv(\partial \M^I)$.
\end{asp}
\begin{deff}\label{d:attractors} Let $\M_1$ be the set of ergodic invariant probability measure on $\partial\R^n_+\times\CN$
satisfying Assumption \ref{a4-pdm}. The ergodic measures which do not satisfy the condition of being `transversal attractors' will be denoted by $\M_2:=\M\setminus \M_1$.
\end{deff}

\begin{thm}\label{thm5.2}
 Suppose Assumptions \ref{a1-pdm}, \ref{a:bound} and  \ref{a4-pdm} hold and $\R_+^I$ is accessible. Then for any $\delta<\delta_0$
and any $\bx\in\R^{n,\circ}_+$ we have
\begin{equation}\label{e.extinction}
\lim_{t\to\infty}\E_{\bx,k} \bigwedge_{i=1}^n X_i^{\delta}(t)=0,
\end{equation}
where $\bigwedge_{i=1}^n a_i=\min_{i=1,\dots,n}\{a_i\}.$
\end{thm}
Sometimes it is hard to prove that the boundary is accessible. The following theorem shows that even in that case, if we start close enough to the boundary we will converge exponentially fast to it with high probability. 
\begin{thm}\label{t:exxx_i}
Suppose Assumptions \ref{a1-pdm} and \ref{a:bound} hold. If $I$ satisfies Assumption \ref{a4-pdm} there exists
$\beta_I>0$ such that, for any a compact set $\K^I\subset
\Se^{I}_+$,
we have
$$
\lim_{\dist(\bz,\K^I)\to0, \bz\in
S^\circ}\PP_\bz\left\{\lim_{t\to\infty}\dfrac{\ln X_i(t)}t\leq
-\beta_I, i\in I^c\right\}=1
$$
\end{thm}

\begin{asp}\label{a5-pdm}
Let $S$ be the family of subsets $I$ satisfying the Assumption \ref{a4-pdm}. We assume either that $S^c:=2^{\{1,\dots,n\}}\setminus S$ is empty, where $2^{\{1,\dots,n\}}$ denotes the family of all subsets of $\{1,\dots,n\}$
or that
$$\max_{i=1,\dots,n}\left\{\lambda_i(\nu)\right\}>0\text{ for any }\nu\in\Conv\left(\displaystyle\bigcup_{J\notin S}\M^{J,\circ}\right).$$
\end{asp}
Let $\U(\omega)$ be the weak$^*$-limit set of the
randomized occupation measure
$$\wtd \Pi_t(\cdot,k):=\dfrac1t\int_0^t\1_{\{\BX(s)\in\cdot, r(s)=k\}}ds,\,\,k\in\CN,t>0.$$
We want to give more exact extinction results, so that we can say exactly which species persist and which go extinct. The following result says that if a subspace $\R_+^I$ has only ergodic attractors $\mu$ in its interior $\R_+^{I,\circ}$ and is accessible, then, with strictly positive probability, the process converges to a convex combination of the attracting ergodic measures. Moreover, the sum of all these probabilities has to be equal to one. 
\begin{thm}\label{thm5.3}
	Suppose that Assumptions \ref{a1-pdm}, \ref{a:bound}, \ref{a4-pdm} and \ref{a5-pdm} are satisfied and $\M^1\neq \emptyset$. Suppose furthermore that $\bigcup_{I\in S} S^I_+$ is accessible.
Then for any $\bx\in\R^{n,\circ}_+$
\begin{equation}\label{e0-thm4.2}
\sum_{I\in S} P_{\bx,k}^I=1
\end{equation}
where for $\bx\in\R^{n,\circ}_+, k\in \CN, I \in S$

$$P_{\bx,k}^I:=\PP_{\bx,k}\left\{\emptyset\neq\U(\omega)\subset\Conv\left(\M^{I,+}\right)
~\text{and}~\lim_{t\to\infty}\frac{\ln
X_j(t)}{t}\in\left\{\lambda_j(\mu):\mu\in\Conv\left(\M^{I,+}\right)\right\},
j\notin I\right\}.$$
Moreover, if $\R_+^I $ is accessible from $(\bx,k)$ then $P_{\bx,k}^I>0$.
\end{thm}
If one ensures that subspaces have at most one interior attractor, we get the following extinction result as an immediate corollary.
 \begin{thm}\label{t:ex_one_intro}
Suppose that Assumptions \ref{a1-pdm}, \ref{a:bound}, \ref{a4-pdm} and \ref{a5-pdm} are satisfied and $\M^1\neq \emptyset$. Suppose furthermore that $\bigcup_{I \in S} S^I_+$ is accessible and that every subset $I\in S$ is such that $\M^{I,\circ}=\{\mu_I\}$. Then for all $(\bx,k)\in \R_+^{n,\circ}\times\CN$ we have 
$$
\sum_{I\in S} P_{\bx,k}^I=1
$$
where
$$P_{\bx,k}^I:=\PP_{\bx,k}\left\{\U(\omega)=\{\mu_I\}\,\text{ and }\,\lim_{t\to\infty}\dfrac{\ln X_i(t)}t=\lambda_i(\mu_I)<0, i\in I^c\right\}.$$ 
Moreover, if $\R_+^I $ is accessible from $(\bx,k)$ then $P_{\bx,k}^I>0$.

\end{thm}

\section{Applications}\label{s:app}
In this section we will showcase our general results by applying them in some interesting ecological models.

\subsection{Competitive Lotka-Volterra models}
A common ecological two-species model is one in which two species compete for resources according to Lotka-Volterra functional responses \citep{wangersky1978lotka}. In our setting, this model is given by
\begin{equation}\label{e3-ex1}
\begin{split}
\frac{dX_1}{dt}(t)&=X_1(t)[a_1(r(t))-b_1(r(t))X_1(t)-c_1(r(t))X_2(t)]\\
\frac{dX_2}{dt}(t)&=X_2(t)[a_2(r(t))-b_2(r(t))X_2(t)-c_2(r(t))X_1(t)].
\end{split} 
\end{equation}
For simplicity we assume that the switching process $r(t)$ takes values in $\CN=\{1,2\}$, the switching rates are constant, and $r(t)$ is irreducible. Here $a_i(k)$ is the per-capita growth rate of species $i$ in environment $k$, $b_i(k)$ the intraspecific competition rate of species $i$ in environment $k$ and $c_i(k)$ the interspecific competition rate of species $j$ on species $i$ in environment $k$.
Suppose that $a_i(k)>0, b_i(k)\geq 0, c_i(k)\geq 0$ for $i,k=1,2$ and that the switching intensities of $r(t)$, satisfy $q_{12}(\bx)>0$ and $q_{21}(\bx)>0$ for all
$\bx\in\R^2_+$.

One can easily check that $F(x,k) = \alpha(k)(1+x_1+x_2)$, for $\alpha(k)>0$ for $k\in \CN$ will satisfy Assumptions \ref{a1-pdm} and \ref{a:bound} so that our results can be applied.

Since at $\bx=(0, 0)$, we have
$$\lambda_1(\delta_0)= \sum_{k=1}^2\int_{(0,0)} \left(a_1(k)-b_1(k)x_1-c_1(k)x_2\right)\,\delta_0(d\bx)=\sum_{k=1}^2 a_1(k)>0$$
and
$$\lambda_2(\delta_0)= \sum_{k=1}^2\int_{(0,0)} \left(a_2(k)-b_2(k)x_1-c_2(k)x_2\right)\,\delta_0(d\bx)=\sum_{k=1}^2 a_2(k)>0$$
one can show by \cite{BL16}, or by Theorem \ref{thm5.1}, that there are invariant probability measures $\mu_i, i=1,2$ whose supports are contained in $\R^\circ_{i+}\times\CN$, $i=1,2$ where $\R^\circ_{i+}= \{\bx\in \R^2_+ ~|~ x_i>0, x_j=0, j\neq i\}$.
Let
$$\lambda_2(\mu_1)=\sum_{k=1}^2\int_{\R^{\circ}_{1+}}[a_2(k)-c_2(k)x_1]\mu_1(d\bx,k)$$
and
$$\lambda_1(\mu_2)=\sum_{k=1}^2\int_{\R^{\circ}_{2+}}[a_1(k)-c_1(k)x_2]\mu_2(d\bx,k).$$

Note that these invasion rates can be computed explicitly but involve  special functions \citep{BL16}. We get the following classification of the long-term behavior of the process.
\begin{itemize}
\item Suppose $\lambda_1(\mu_2)<0, \lambda_2(\mu_1)<0$. Then, using Theorem \ref{t:ex_one_intro}, we have $p^{\bx,k}_1+p^{\bx,k}_2=1$ where
$$p^{\bx,k}_i=\PP_{\bx,k}\left\{X_i\to \mu_i, \lim_{t\to\infty}\dfrac{\ln X_j(t)}t=\lambda_j(\mu_i), j\in\{1,2\}\setminus\{i\}\right\}.$$
If $\bar\gamma_+(\bx)\cap \R^{\circ}_{i+}\ne\emptyset$ then $p^{\bx,k}_i>0$.
Note that in some specific cases we can know when
$\bar\gamma_+(\bx)\cap \R^{\circ}_{i+}\ne\emptyset$
(see \cite{BL16, DDY14}).
\item Suppose $\lambda_j(\mu_i)<0$ for one $i\in\{1,2\}$ and $\lambda_i(\mu_j)>0: j\in\{1,2\}\setminus\{i\}$. Then,  using Theorem \ref{t:ex_one_intro}, we see that $X_j(t)$ converges to $0$ almost surely with the exponential rate $\lambda_j(\mu_i)$ for any initial condition $(\bx,k)\in\R^{2,\circ}_+\times\{1,2\}$. The occupation measure converges almost surely for all initial conditions $(\bx,k)\in\R^{2,\circ}_+\times\{1,2\}$ to $\mu_i$.
\item Suppose $\lambda_1(\mu_2)>0, \lambda_2(\mu_1)>0$. Assume that at least one of the following two conditions holds,
\begin{itemize}
\item[i)] $\dfrac{b_1(1)}{b_2(1)}\ne \dfrac{b_1(2)}{b_2(2)}$
\item[ii)] $\dfrac{c_1(1)}{c_2(1)}\ne \dfrac{c_1(2)}{c_2(2)}.$
\end{itemize}\
It is shown in \cite[Theorem 4.1]{BL16} that $\Gamma\ne\emptyset$ and there is a point $m_0\in\Gamma$ satisfying the strong bracket condition.
By Theorem \ref{thm5.2}, the transition probability of the process $(\BX(t),r(t))$ on $\R^{2,\circ}_+\times\CN$
converges exponentially fast to a unique invariant probability measure on $\R^{2,\circ}_+\times\CN$ in total variation.
\end{itemize}
The above conclusions coincide with the main results of \cite{BL16}. 
\begin{rmk}
We can also recover and generalize the results from the more general setting of \cite{DDY14} where the authors consider
\begin{equation}\label{e:LV_comp}
\begin{split}
\frac{dX_1}{dt}(t)&=X_1(t) f_1(X_1(t),X_2(t),r(t))\\
\frac{dX_2}{dt}(t)&=X_2(t) g(X_1(t),X_2(t),r(t))
\end{split}
\end{equation}
under the assumptions that the process remains in a compact set and that for all $x_1, x_2>0, k\in \CN$ one has $\frac{\partial f}{\partial x_1}(x_1, 0,k)<0, \frac{\partial g}{\partial x_1}(0, x_2,k)<0, f(0,0,k)>0, g(0,0,k)>0, \limsup_{x_1\to \infty} f(x_1,0,k)<0, \limsup_{x_2\to \infty} g(0,x_2,k)<0$.
\end{rmk}

\subsection{Predator-prey models}
If instead we look at predator-prey systems of Lotka Volterra type \citep{lotka2002contribution, volterra1928variations, wangersky1978lotka}, we get the system of the form
\begin{equation}\label{e:pp}
\begin{split}
\frac{dX_1}{dt}(t)&=X_1(t)[a_1(r(t))-b_1(r(t))X_1(t)-c_1(r(t))X_2(t)]\\
\frac{dX_2}{dt}(t)&=X_2(t)[-a_2(r(t))-b_2(r(t))X_2(t)+c_2(r(t))X_1(t)]
\end{split}
\end{equation}

Here we assume that $r(t)\in\CN=\{1,2\}$, and $a_i(k),b_i(k),c_i(k)>0$ for $i=1,2$ and $k\in\CN$. The meaning of the coefficients is as follows: $a_1(k)$ is the per-capita growth rate of the prey, $a_2(k)$ is the per-capita death rate of the predator, $b_1(k), b_2(k)$ are the intraspecific competition rates of prey and predators, $c_1(k)$ is the attack rate of the predator on the prey and $c_2(k)$ is the rate at which the predator makes use of the prey to sustain itself. 
\begin{rmk}
If we have $\frac{a_1(k)}{b_1(k)}\leq\frac{a_2(k)}{c_2(k)}$ for at least one $k\in\CN$, then accessibility of $(0,\infty)\times\{0\}$ follows trivially, since at least one of the vector fields can access the set $(0,\infty)\times\CN$.
\end{rmk}

\begin{rmk}\label{pp_strong}
Denote the vector fields for $k\in\{1,2\}$ by  
\[
V(k)=\begin{bmatrix}
x_1A_1^{k} \\
x_2A_2^{k} 
\end{bmatrix}=\begin{bmatrix}
x_1\left(a_1(k)-b_1(k)x_1-c_1(k)x_2\right) \\
x_2\left(-a_2(k)-b_2(k)x_2+c_2(k)x_1\right).
\end{bmatrix}
\]
We have 
\[
V(1)-V(2)=\begin{bmatrix}
x_1(A_1^{1}-A_1^2) \\
x_2(A_2^{1}-A_2^2). 
\end{bmatrix}
\]
Now computing the Lie Bracket we get 
\[
[V(1),V(1)-V(2)]=\begin{bmatrix}
x_1^2\left(A_1^1b_1(2)-A_1^2b_1(1)\right) +x_1x_2\left(A_2^1c_1(2)-A_2^2c_1(1)\right) \\
x_2^2\left(A_2^1b_2(2)-A_2^2b_2(1)\right)+x_1x_2\left(A_1^2c_1(2)-A_1(1)c_2(2) 
 \right). 
\end{bmatrix}
\] To verify that the vector fields $V(1)-V(2)$ and $[V(1),V(1)-V(2)]$ are linearly independent it is enough to verify that the determinant of the matrix 
\[
M = \begin{bmatrix}
V(1)-V(2) & [V(1),V(1)-V(2)] \\ 
\end{bmatrix}
\]
is nonzero. Note that
\[
det(M)=Dx_1^2x_2+ Ex_1x_2^2
\]
where
\begin{equation}
\begin{split}
D=&-c_2(2)\left[\left(A_1^1-\frac{A_1^2}{2}\right)^2 + \left(\left(\frac{b_1(1)}{c_2(2)}\right)(A_2^2-A_2^1)A_1^2-\left(\frac{A_1^2}{2}\right)^2 \right)   \right]\\
&-c_2(1)\left[\left(A^2_1-\frac{A_1^1}{2}    \right)^2+ \left(\left(\frac{b_1(2)}{c_2(1)}\right)(A_2^1-A_2^2)A_1^1-\left(\frac{A_1^1}{2}\right)^2      \right) \right ]\\
E=&-c_1(2)\left[\left(A_2^1-\frac{A_2^2}{2}   \right)^2 + \left( \left(\frac{b_2(1)}{c_1(2)}\right)(A_1^1-A_1^2)A_2^2-\left(\frac{A_2^2}{2}  \right)^2   \right)         \right]\\
&-c_1(1)\left[\left(A_2^2-\frac{A_2^1}{2} \right)^2 + \left(\left(\frac{b_2(2)}{c_1(1)}\right)(A_1^2-A_1^1)A_2^1-\left(\frac{A_2^1}{2}\right)^2  \right)       \right].
\end{split}    
\end{equation}
So unless $D=E=0$, we have $det(M)\neq 0$, at all points except possibly the straight line $Dx_1+Ex_2=0$. 
    
\end{rmk}

\begin{rmk}\label{pp_fp}
If the fixed points of vector fields $V_1$, and $V_2$ are $\bx_1=(x_{11},x_{12})$ and $\bx_2=(x_{21},x_{22})$. Then if we have for some $i\in\{1,2\}$ 
\[
Dx_{i1} + Ex_{i2}\neq 0
\]
and $(x_{i1},x_{i2})$ is a globally attracting fixed point. Then there exists a fixed point which is globally attracting and the strong bracket condition holds. 
\end{rmk}

\begin{rmk}\label{pp_ext}
Since the model describes predator prey dynamics, if there exists an invariant probability measure $\mu'$ on $\partial \R^2_+$, then $supp(\mu')\subset  (0,\infty)\times \{0\}$.  Then Proposition 2.1 from \cite{BL16} applies, so if $\frac{a_1(1)}{b_1(1)}\neq \frac{a_1(2)}{b_1(2)}$, then  
\begin{equation}\label{e:h1}
\mu'(dx,1)=h_1(x) dx = C(p_1,p_2,\gamma_1,\gamma_2)\frac{p_1|x-p_2|^{\gamma_2}|p_1-x|^{\gamma_1-1}}{x^{1+\gamma_1+\gamma_2}} dx
\end{equation}
\begin{equation}\label{e:h2}
\mu'(dx,2)=h_2(x) dx = C(p_1,p_2,\gamma_1,\gamma_2)\frac{p_2|x-p_2|^{\gamma_2-1}|p_1-x|^{\gamma_1}}{x^{1+\gamma_1+\gamma_2}} dx
\end{equation}
where $p_1=\frac{a_1(1)}{b_1(1)}$, $p_2=\frac{a_1(2)}{b_1(2)}$, $\gamma_1=\frac{q_{12}}{a_1(1)}$ and $\gamma_2=\frac{q_{21}}{a_1(2)}$, and $supp(\mu')\in [p_1,p_2]$ if $p_1<p_2$.

\end{rmk}

Since the switching process $r(t)$ is irreducible and aperiodic, there is a unique invariant distribution $\nu=(\nu_1,\nu_2)=\left(\frac{q_{21}}{q_{12}+q_{21}},\frac{q_{12}}{q_{12}+q_{21}}\right)$ on $\CN$. Hence we have an invariant measure supported on $\{(0,0)\}\times\CN$ $\mu_0=\delta_0\times \nu$. 
\[
\lambda_1(\mu_0)=\sum_{k=1}^2\int \left(a_1(k)-b_1(k)x_1-c_1(k)x_2\right) \mu_0(dx,k) = a_1(1)\left(\frac{q_{21}}{q_{12}+q_{21}}\right) + a_1(2)\left(\frac{q_{12}}{q_{12}+q_{21}}\right)>0
\]
\[
\lambda_2(\mu_0)=\sum_{k=1}^2\int \left(-a_2(k)-b_2(k)x_1-c_2(k)x_2\right) \mu_0(dx,k) = -a_2(1)\left(\frac{q_{21}}{q_{12}+q_{21}}\right) - a_2(2)\left(\frac{q_{12}}{q_{12}+q_{21}}\right)<0
\]
If we restrict our model to $\{0\}\times\R^{\circ}_+$. Then $\lambda_2(\mu_0)<0$, then by Theorem \ref{thm5.3}, $X_2(t)\to 0$ a.s, as we expect since the predator dies out without the prey. There is no invariant measure supported on $\{0\}\times \R_+^{\circ}$.\\
Now if we restrict our model to $\R^{\circ}_+\times\{0\}\times\CN$. Since $\lambda_1(\mu_0)>0$, then by Theorem \ref{thm5.1} there exists a unique invariant measure $\mu_1$ supported in $\R^{\circ}_+\times\{0\}\times\CN$. If $\frac{a_1(1)}{b_1(1)}=\frac{a_1(2)}{b_1(2)}=p$, then $\mu_1=\delta_p\otimes \nu$. Hence we can compute $\lambda_2(\mu_1)$ as
 \[
    \lambda_2(\mu_1)= (c_2(1)p-a_2(1))\left(\frac{q_{21}}{q_{12}+q_{21}}\right) + (c_2(2)p-a_2(2))\left(\frac{q_{12}}{q_{21}+q_{12}}\right).
\]
If $\frac{a_1(1)}{b_1(1)}<\frac{a_1(2)}{b_1(2)}$, then from  
\cite{bakhtin2015regularity}, $\mu_1$ is absolutely continuous and given by a smooth density supported in $\left(\frac{a_1(1)}{b_1(1)},\frac{a_1(2)}{b_1(2)}\right)$, and vice-versa when $\frac{a_1(1)}{b_1(1)}>\frac{a_1(2)}{b_1(2)}$. We can find the density of $\mu_1$ with respect to Lebesgue measure using \cite{BL16} as in Remark \ref{pp_ext}, and use it to compute $\lambda_2(\mu_1)$ as

\begin{equation}
\begin{split}
\lambda_2(\mu_1)=\int_{\frac{a_1(1)}{b_1(1)}}^{\frac{a_1(2)}{b_1(2)}} (c_2(1)x-a_2(1)) h_1(x) dx + \int_{\frac{a_1(1)}{b_1(1)}}^{\frac{a_1(2)}{b_1(2)}} (c_2(2)x-a_2(2))h_2(x) dx    
\end{split}    
\end{equation}
Based on the sign of $\lambda_2(\mu_1)$ we have the following classification of the dynamics.

\begin{itemize}

\item If $\lambda_2(\mu_1)>0$ then Assumption \ref{a2-pdm} is satisfied, and if in addition the conditions of Remark \ref{pp_strong} and Remark \ref{pp_fp} are satisfied then Assumption \ref{a3-pdm} also holds. This allows us to apply Theorem \ref{thm5.1} and get that there is a unique invariant probability measure $\pi^*$ in $\R^{2,\circ}_+$ such that for any initial conditions $X_1(0), X_2(0), r(0))=(x_1, x_2, k)\in \R_+^{2,\circ}\times \CN$ the transition probability of $(X_1(t),X_2(t),r(t))$, converges to $\pi^*$ in total variation exponentially fast. 

\item If we have instead $\lambda_2(\mu_1)<0$ then Assumptions \ref{a4-pdm} and \ref{a5-pdm} are satisfied. Hence by Theorem \ref{thm5.2}, we see that for any initial conditions $X_1(0), X_2(0), r(0))=(x_1, x_2, k)\in \R_+^{2,\circ}\times \CN$ one has $X_2(t)\to 0$ and $X_1(t)\to \mu_1$ almost surely. 

\end{itemize}

\begin{rmk}
The paper of \cite{NN14} studies two-species predator-prey models under switching dynamics in greater generality. They are able to show that the predator goes extinct if $\lambda_2(\mu_1)<0$, without directly considering the accessiblity of the set $(0,\infty)\times \{0\}$. 

\end{rmk}

\subsection{Single Species Dynamics- stability versus explosion.} \label{s:1d_explosion} 
In the following example, we have an isolated species that does not experience any competition in the environment $k=1$ and grows (or decays) linearly while it can possibly experience intra-competition in the environment $k=2$. 
Suppose the dynamics is given by 
\begin{equation}
\frac{dX}{dt}(t)=X(t)f(X(t),r(t))
\end{equation}
with
$$f(x,k)=\begin{cases}
\displaystyle  a_1(1) &\text{if }       k=1\\[.1truein]
\displaystyle  a_1(2)-b_1(2)x  &\text{if } k=2.
\end{cases}$$
We define the switching rate matrix $Q(x)=(q_{ij}(x))_{2\times 2}$, where $q_{12}(x),q_{21}(x)>0$, $a_1(1),b_1(2)>0$, and $a_1(2)\neq 0$. 
Additionally, assume that there are constants $C,c >0$ such that for all $x>0$ we have
$$ 0<c\leq\min\{q_{12}(x),q_{21}(x)\},$$ 
and 
$$ \max\{q_{12}(x),q_{21}(x)\}\leq C<\infty.$$
\begin{rmk}\label{non_comp_1D}
The dynamics of this model is not restricted to any compact subset of $\R_+$. Since the process $r(t)\in \CN$ is an irreducible Markov chain then if $r(0)\neq 1$, then $r(T)=1$ for some $T>0$, and since the waiting time in the state $r(t)=1$ is exponentially distributed with mean at least $\frac{1}{C}$, and since $X(t)$ grows exponentially when $r(t)=1$, the process escapes any compact set with positive probability.     
\end{rmk}

Define the function 

$$F(x,k)=\begin{cases}
\displaystyle  \alpha(1+\sqrt{x}) &\text{if }       k=1\\[.1truein]
\displaystyle  \alpha\beta(1+\sqrt{x})  &\text{if } k=2
\end{cases} $$
for constants $\alpha>0$, and $\beta\in (0,1)$. If we have the condition $\displaystyle c-\frac{a_1(1)}{2}>0$, then we have for $k=1$ that

\begin{equation}
\begin{split}
\limsup_{x\to\infty}&\left[\frac{\Lom F(x,1)}{F(x,1)}+ \delta_0\left(\max_{i}\{|f_i(\bx,1)|\}\right)\right]\\
=&\limsup_{x\to\infty}\frac{\left(xa_1(1)\left(\frac{\alpha}{2\sqrt{x}}\right)+q_{12}(x)\alpha(\beta-1)(1+\sqrt{x})+ \delta_0a_1(1)\alpha(1+\sqrt{x})\right)}{\alpha(1+\sqrt{x})}\\
\leq&\limsup\frac{\left( \left(\frac{a_1(1)}{2}+\delta_0a_1(1)+c(\beta-1)\right)\alpha\sqrt{x}- c\alpha(1-\beta) + \delta_0a_1(1)\alpha \right)}{\alpha(1+\sqrt{x})}.   
\end{split}
\end{equation}
Since $\displaystyle c-\frac{a_1(1)}{2}>0$, for $\beta\in (0,1)$, and $\delta_0>0$,  small enough, we get that 
\begin{equation}
\begin{split}
\left[\frac{\Lom F(x,1)}{F(x,1)}+ \delta_0\left(\max_{i}\{|f_i(\bx,1)|\}\right)\right] <\left(\frac{a_1(1)}{2(1-\beta)}-c + 2\delta_0\left(\frac{a_1(1)}{1-\beta}\right)         \right)<0.
\end{split}  
\end{equation}
For $k=2$, we have in similar fashion

\begin{equation}
\begin{split}
&\limsup_{x\to\infty}\left[\frac{\Lom F(x,2)}{F(x,2)}+ \delta_0\left(\max_{i}\{|f_i(x,2)|\}\right)\right]\\
&\leq\limsup_{x\to\infty} \frac{\left(\frac{\alpha a_1(2)\beta}{2}\sqrt{x}-\frac{\alpha\beta b_1(2)}{2}x^{3/2}+C(1-\beta)\alpha(1+\sqrt{x}) + \delta_0(b_1(2)x-1)\alpha\beta(1+\sqrt{x})\right)}{\alpha\beta(1+\sqrt{x})}\\
&\leq\limsup_{x\to\infty}\frac{\left( \alpha\beta b_1(2)\left(\delta_0-\frac{1}{2}\right)x^{3/2}+ \alpha\beta\frac{ a_1(2)}{2}\sqrt{x}+C\alpha(1-\beta)(1+\sqrt{x}) +\delta_0b_1(2)\alpha\beta x -\delta_0\alpha\beta(1+\sqrt{x}) \right)}{\alpha\beta(1+\sqrt{x})}
\end{split}    
\end{equation}
As a result, if $\delta_0$ is small enough, it follows that
\begin{equation}
\limsup_{x\to\infty} \left[\frac{\Lom F(x,2)}{F(x,2)}+ \delta_0\left(\max_{i}\{|f_i(x,2)|\}\right)\right]<0. 
\end{equation}

As a result we can see that Assumption \ref{a1-pdm} is satisfied, since \eqref{e:sup2} holds for $\delta_0$ small enough. The function $F$ is defined in such a way that Assumption \ref{a:bound} is satisfied trivially. 

\begin{rmk}

We can also consider the fully linear system with 
$$f(x,k)=\begin{cases}
\displaystyle  a_1(1) &\text{if }       k=1\\[.1truein]
\displaystyle  a_1(2)  &\text{if } k=2.
\end{cases}$$
where $a_1(1)>0, a_1(2)<0$. This means that the species grows linearly in environment $k=1$ and dies out (also linearly) in environment $k=1$. One can show that Assumption \ref{a1-pdm} follows if we have $\displaystyle c-\frac{a_1(1)}{2}>0$.
\end{rmk}

\begin{rmk}\label{r:support}
For this model it is easy to see that there are only two possible scenarios for the accessibility of $\{0\}$, one when $a_1(2)>0$, and the other when $a_1(2)<0$. If $a_1(2)>0$, then $\{0\}$ is not accessible and $\Gamma = \left[\frac{a_1(2)}{b_1(2)},\infty\right)$. The accessibility of $\{0\}$ follows trivially if $a_1(2)<0$, since then for any initial condition $X_2(0)>0$ we have $X_2(t)\to 0$ as $t\to\infty$ in the fixed environment $k=2$ and in this case $\Gamma=[0,\infty)$.    
\end{rmk}

\begin{lm}\label{1d_density}
Assume the switching rates are constant. If $\lambda(\delta_0\times\nu)>0$, then there exists a unique invariant measure $\mu$ supported on $\Gamma^\circ\subset (0,\infty)$, such that $\mu$ is absolutely continuous, with densities given by
\begin{equation}
\begin{split}
\mu(dx,1)=& h_1(x) dx\\
\mu(dx,2)=& h_2(x) dx.
\end{split}
\end{equation}
Here $h_1(x),h_2(x)$ are smooth functions supported on $\Gamma$. Moreover:
\begin{enumerate}
    \item If $a_1(2)>0$ then 
\begin{equation}
\begin{split}
h_1(x)=& C_1 \frac{(b_1(2)x-a_1(2))^{\frac{q_{21}}{a_1(2)}}}{x^{1+ \frac{q_{12}}{a_1(1)}+\frac{q_{21}}{a_1(2)} }} \\
h_2(x)=&C_1 \frac{a_1(1)(b_1(2)x-a_1(2))^{\frac{q_{21}}{a_1(2)}-1}}{x^{1+ \frac{q_{12}}{a_1(1)}+\frac{q_{21}}{a_1(2)} }} 
\end{split}    
\end{equation}
for $x> \frac{a_1(2)}{b_1(2)}$ and zero otherwise. The constant $C_1$ is determined by
\[
\int_{\frac{a_1(2)}{b_1(2)}}^\infty (h_1(x)+ h_2(x))\,dx=1
\]
\item If $a_1(2)<0$ then
\begin{equation}
\begin{split}
h_1(x)=&C_1\frac{1}{x^{1+\frac{q_{21}}{|a_1(2)|}-\frac{q_{12}}{a_1(1)}}(b_1(2)x+a_1(2))^{\frac{q_{21}}{a_1(2)}}}\\
h_2(x)=&C_1\frac{a_1(1)}{x^{1+\frac{q_{21}}{|a_1(2)|}-\frac{q_{12}}{a_1(1)}}(b_1(2)x+a_1(2))^{1+\frac{q_{21}}{a_1(2)}}}
\end{split}    
\end{equation}
for $x>0$ and zero otherwise. The constant $C_1$ is determined by
\[
\int_{0}^\infty (h_1(x)+ h_2(x))\,dx=1.
\]

\end{enumerate}
\end{lm}

\begin{rmk}
Note that if $a_1(2)>0$ then it is always true that $\lambda(\delta_0\times\nu)>0$.    
\end{rmk}

\begin{proof}
The strong bracket condition holds trivially, so Assumption \ref{a3-pdm} holds, and we have by Theorem \ref{thm5.1} that there exists a unique invariant measure $\mu\in (0,\infty)\times \CN$  Then by Proposition 3.17 from \cite{BBMZ15} the support of the measure should be $\Gamma^\circ$ and by Theorem 1 from \cite{BH12} $\mu$ has a density with respect to Lebesgue measure. The result from \cite{bakhtin2015regularity} shows the smoothness of the density.  

Since $\mu$ is an invariant distribution, we know that for any $\phi\in C_c^{\infty}(\Gamma^{\circ})$ the following holds

$$\sum\int P_t \phi(x,k) \mu(dx,k) = \sum\int \phi(x,k)  \mu(dx,k)   $$
Hence $\displaystyle  \sum\int \Lom \phi(x,k) \mu(dx,k)=0$ for any $\phi\in C_c^{\infty}(\Gamma^{\circ})$. As a result, we have 

\begin{equation}
\begin{split}
0=&\int \left(xf(x,2)\frac{d\Psi(x,2)}{dx} + q_{21}\left(\Psi(x,1)-\Psi(x,2) \right)\right)h_2(x) dx \\
+&\int \left( xf(x,1)\frac{d\Psi(x,1)}{dx}+ q_{12}\left(\Psi(x,2)-\Psi(x,1)  \right)\right) h_1(x) dx.  
\end{split}    
\end{equation}
Next, taking $\Psi(x,2)=g(x)+c$ for $g\in C_c^{\infty}(0,\infty)$, $c\in \R$ and $\Psi(x,1)=0$, etc., we get the pair of equations
\begin{equation}\label{coupled_2}
\begin{split}
0=&\left((a_1(2)x-b_1(2)x^2)h_2(x)\right)' + q_{21}h_2(x)-q_{12}h_1(x)\\
0=&(a_1(1)xh_1(x))'+ q_{12}h_1(x)-q_{21}h_2(x)
\end{split}    
\end{equation}
along with the condition 
\begin{equation}\label{int_cond}
\int_{supp(\mu)}q_{12}h_1(x)-q_{21}h_2(x) dx =0.  
\end{equation}
Next, we differentiate two cases, according to the sign of $a_{1}(2).$

\textbf{Case I:} $a_1(2)>0$. 

From \eqref{coupled_2}, we can see that 
\[
\left(b_1(2)x^2-a_1(2)x)h_2(x)\right)' = (a_1(1)x h_1(x))'
\]
Then there exists $C^*\in \R$ such that

\begin{equation}\label{un_couple}
\left(b_1(2)x^2-a_1(2)x\right)h_2(x) -a_1(1)xh_1(x) =C^*   
\end{equation}

Combining \eqref{un_couple} and \eqref{coupled_2} we get a pair of linear equations. 

\begin{equation}
\begin{split}
&h_1'(x)+\left(\frac{q_{12}+a_1(1)}{a_1(1)x}-\frac{q_{21}}{b_1(2)x^2-a_1(2)x}\right)h_1(x)=\frac{q_{21}C^*}{a_1(1)x^2(b_1(2)x-a_1(2))}\\
&h_2'(x)+\left(\frac{2b_1(2)x-(q_{21}+a_1(2))}{b_1(2)x^2-a_1(2)x}+\frac{q_{12}}{a_1(1)x}\right)h_2(x)=\frac{q_{12}C^*}{a_1(1)x^2(b_1(2)x-a_1(2))}
\end{split}    
\end{equation}
To get the solutions, we have the following integrating factors

\begin{equation}\label{int_factor}
\begin{split}
I_1(x) =& x^{\frac{q_{12}+a_1(1)}{a_1(1)} + \frac{q_{21}}{a_1(2)}} \cdot (b_1(2)x - a_1(2))^{-\frac{q_{21}}{a_1(2)}}\\
I_2(x) =& x^{\frac{q_{12}}{a_1(1)} + \frac{q_{21} + a_1(2)}{a_1(2)}}
\cdot (b_1(2)x - a_1(2))^{1-\frac{q_{21}}{a_1(2)}}.
\end{split}
\end{equation}
The general solutions will be given by
\begin{equation}\label{gen_sol}
\begin{split}
h_1(x)=&\frac{1}{I_1(x)}\left(\int \frac{q_{21}C^*I_1(x)}{a_1(1)x^2(b_1(2)x-a_1(2))} dx   + C_1     \right)\\
h_2(x)=&\frac{1}{I_2(x)}\left(\int \frac{q_{12}C^*I_2(x)}{a_1(1)x^2(b_1(2)x-a_1(2))} dx   + C_2     \right).
\end{split}    
\end{equation}
Now taking $C^*=0$, we can see that, 
\[
h_1(x)= \frac{C_1}{x^{\gamma}(b_1(2)x-a_1(2))^{-\frac{q_{21}}{a_1(2)}}}
\]
\[
h_2(x)=\frac{C_2}{x^{\gamma}(b_1(2)x-a_1(2))^{1-\frac{q_{21}}{a_1(2)}}}
\]
where $\gamma =1 + \frac{q_{12}}{a_1(1)}+\frac{q_{21}}{a_1(2)}=1 + \frac{\lambda(\delta_0\times\nu)(q_{12}+q_{21})}{a_1(1)a_1(2)} $. Since $\lambda(\delta_0\times\nu)>0$, we know that the solutions are integrable functions. 

Now observe that for $C_2=C_1a_1(1)$ equation \eqref{un_couple} is satisfied. Now, we fix $C_1>0$, such that we have 
\[
\int_{\frac{a_1(2)}{b_1(2)}}^{\infty}h_1(x) dx  + \int_{\frac{a_1(2)}{b_1(2)}}^{\infty} h_2(x) dx =1.
\]
To verify this is a valid solution, we check that \eqref{int_cond} holds. 
The integrals become
\begin{equation}
\begin{split}
\int_{\frac{a_1(2)}{b_1(2)}}^{\infty}h_1(x) dx=&C_1 \frac{(b_1(2))^{\beta}}{(a_1(2))^{\beta-\alpha}}B(\alpha+1, \beta-\alpha)\\
\int_{\frac{a_1(2)}{b_1(2)}}^{\infty} h_2(x) dx=&C_1\frac{(b_1(2))^{\beta}a_1(1)}{(a_1(2))^{\beta-\alpha+1}}B(\alpha,\beta-\alpha+1)
\end{split}    
\end{equation}
showing that \eqref{int_cond} is verified. The density functions are therefore given by
\begin{equation}
\begin{split}
h_1(x)=& C_1 \frac{(b_1(2)x-a_1(2))^{\frac{q_{21}}{a_1(2)}}}{x^{1+ \frac{q_{12}}{a_1(1)}+\frac{q_{21}}{a_1(2)} }} \\
h_2(x)=&C_1 \frac{a_1(1)(b_1(2)x-a_1(2))^{\frac{q_{21}}{a_1(2)}-1}}{x^{1+ \frac{q_{12}}{a_1(1)}+\frac{q_{21}}{a_1(2)} }}. 
\end{split}    
\end{equation}

\textbf{Case II:} $a_1(2)<0$

For this case we assume that
\begin{equation}\label{persis_cond}
\lambda(\delta_0\times\nu)(q_{12}+q_{21})=a_1(1)q_{21}-|a_1(2)|q_{12}>0.     
\end{equation}

The differential equations we get in this case, following the same ideas as in Case 1, are given by
\begin{equation}
\begin{split}
&h_1'(x)+\left(\frac{q_{12}+a_1(1)}{a_1(1)x}-\frac{q_{21}}{b_1(2)x^2+|a_1(2)|x}\right)h_1(x)=\frac{q_{21}C^*}{a_1(1)x^2(b_1(2)x+|a_1(2)|)}\\
&h_2'(x)+\left(\frac{2b_1(2)x-(q_{21}-|a_1(2)|)}{b_1(2)x^2+|a_1(2)|x}+\frac{q_{12}}{a_1(1)x}\right)h_2(x)=\frac{q_{12}C^*}{a_1(1)x^2(b_1(2)x+|a_1(2))|}.
\end{split}    
\end{equation}

To get the solutions, we have the following integrating factors

\begin{equation}\label{int_factor2}
\begin{split}
I_1(x) =& x^{\frac{q_{12}+a_1(1)}{a_1(1)} - \frac{q_{21}}{|a_1(2)|}} \cdot (b_1(2)x +  |a_1(2)|)^{\frac{q_{21}}{|a_1(2)|}}\\
I_2(x) =& x^{\frac{q_{12}}{|a_1(1)|} + \frac{-q_{21} + |a_1(2)|}{|a_1(2)|}}
\cdot (b_1(2)x + |a_1(2)|)^{1+\frac{q_{21}}{|a_1(2)|}}.
\end{split}
\end{equation}
The general solutions become 
\begin{equation}\label{gen_sol2}
\begin{split}
h_1(x)=&\frac{1}{I_1(x)}\left(\int \frac{q_{21}C^*I_1(x)}{a_1(1)x^2(b_1(2)x+|a_1(2)|)} dx   + C_1     \right)\\
h_2(x)=&\frac{1}{I_2(x)}\left(\int \frac{q_{12}C^*I_2(x)}{a_1(1)x^2(b_1(2)x+|a_1(2)|)} dx   + C_2     \right).
\end{split}    
\end{equation}
Taking $C^*=0$, we have that 
\[
h_1(x) = \frac{C_1}{x^{\tau}(b_1(2)x+|a_1(2)|)^{\frac{q_{21}}{|a_1(2)|}}}
\]
\[
h_2(x)= \frac{C_2}{x^{\tau}(b_1(2)x+|a_1(2)|)^{1+\frac{q_{21}}{|a_1(2)|}}   } 
\]
where $\tau= 1-\left(\frac{q_{21}}{|a_1(2)|}-\frac{q{12}}{a_1(1)} \right)=1-\frac{\lambda(\delta_0)(q_{12}+q_{21})}{a_1(1)|a_1(2)|}$.
We see that \eqref{un_couple} holds if we take $C_2=a_1(1)C_1$. 
Also since we have $\lambda(\delta_0\times\nu)>0$, we get that $\displaystyle \frac{\lambda(\delta_0\times\nu)(q_{12}+q_{21})}{a_1(1)q_{12}}=q_{21}-|a_1(2)|>0$, which implies $\frac{q_{21}}{|a_1(2)|}>1$. As a result, due to \eqref{persis_cond}, we get that $h_1,h_2$ are integrable on $(0,\infty)$. 

We find the constant $C_1>0$ by setting
\[
\int_0^{\infty} h_1(x) + \int_0^{\infty}h_2(x) dx =1.
\]
Now since the functions $h_1,h_2$ have been defined precisely, we need to check \eqref{int_cond} to verify that they are indeed the desired density functions.

\begin{equation}
\begin{split}
\int_0^{\infty} h_1(x)dx=&C_1\frac{(b_1(2))^{\tau-1}}{(|a_1(2)|)^{\frac{q_{21}}{|a_1(2)|}+\tau-1}}B\left(\frac{q_{21}}{|a_1(2)|}-\frac{q_{12}}{a_1(1)}, \frac{q_{12}}{a_1(1)}\right)\\
\int_0^{\infty}h_2(x) dx=&C_1\frac{(b_1(2))^{\tau-1}a_1(1)}{(|a_1(2)|)^{\frac{q_{21}}{|a_1(2)|}+\tau}}B\left(\frac{q_{21}}{|a_1(2)|}-\frac{q_{12}}{a_1(1)},\frac{q_{12}}{a_1(1)}+1   \right).  
\end{split}    
\end{equation}
Using the property of Beta integrals, we can see that \eqref{int_cond} is indeed verified. The density functions are give  by

\begin{equation}
\begin{split}
h_1(x)=&C_1\frac{1}{x^{1+\frac{q_{21}}{|a_1(2)|}-\frac{q_{12}}{a_1(1)}}(b_1(2)x+|a_1(2)|)^{\frac{q_{21}}{|a_1(2)|}}}, x>0\\
h_2(x)=&C_1\frac{a_1(1)}{x^{1+\frac{q_{21}}{|a_1(2)|}-\frac{q_{12}}{a_1(1)}}(b_1(2)x+|a_1(2)|)^{1+\frac{q_{21}}{|a_1(2)|}}}, x>0.
\end{split}    
\end{equation}
\end{proof}

Assume the switching rates are independent of the position $\bx$ and $r(t)$ is irreducible. Let $\nu=(\nu_1,\nu_2)$ be the unique invariant measure of $r(t)$ on $\CN$. This can be computed explicitly as 

$$ (\nu_1,\nu_2)=\left(\frac{q_{21}}{q_{12}+q_{21}},\frac{q_{12}}{q_{12}+q_{21}}  \right).$$ 

As a result we can compute the invasion rate into $\delta_0\times\nu$ as
$$\lambda(\delta_0\times\nu)=\nu_1a_1(1)+\nu_2a_1(2)= \frac{q_{21}}{q_{12}+q_{21}}a_1(1) + \frac{q_{12}}{q_{12}+q_{21}} a_1(2).$$ 

\begin{figure}
    \centering
    \includegraphics[width=0.75\linewidth]{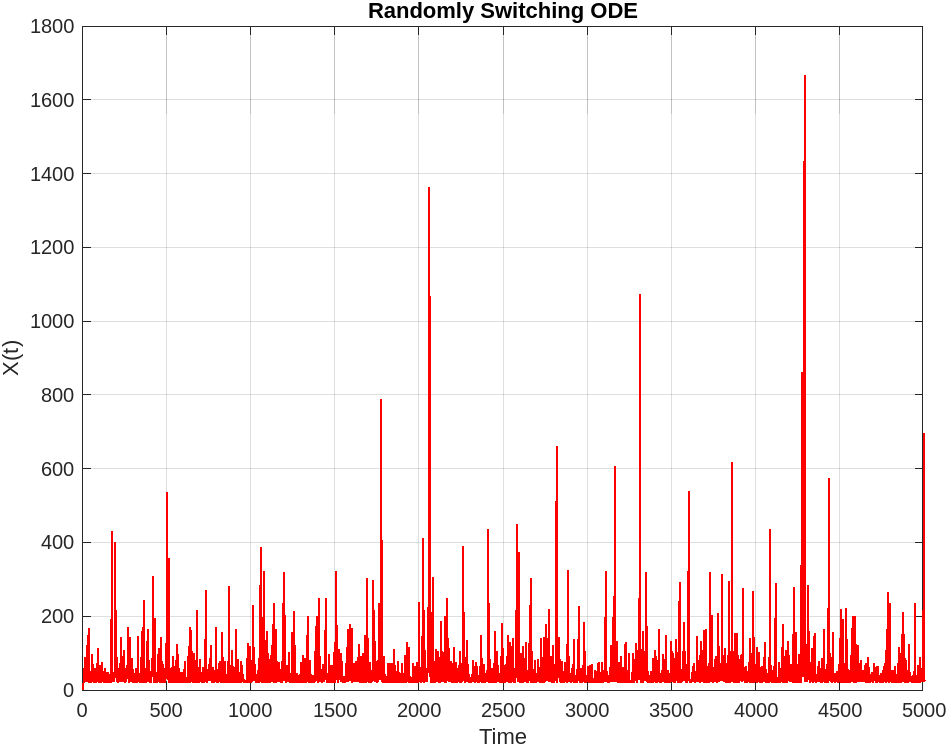}
    \caption{This is a sample path of the process $(X(t),r(t))$ from \ref{s:1d_explosion} , with the coefficients, $a_1(1)=0.5$, $a_1(2)=1$, and $b_1(2)=0.05$. The switching rates $q_{12}=q_{21}=2$, and $X(0)=1$. In this setting $\lambda(\delta_0\times\nu)>0$.}
    \label{1D-a_12pos}
\end{figure}

\begin{figure}
    \centering
    \begin{subfigure}{0.47\textwidth}
    \centering
    \includegraphics[width=\textwidth]{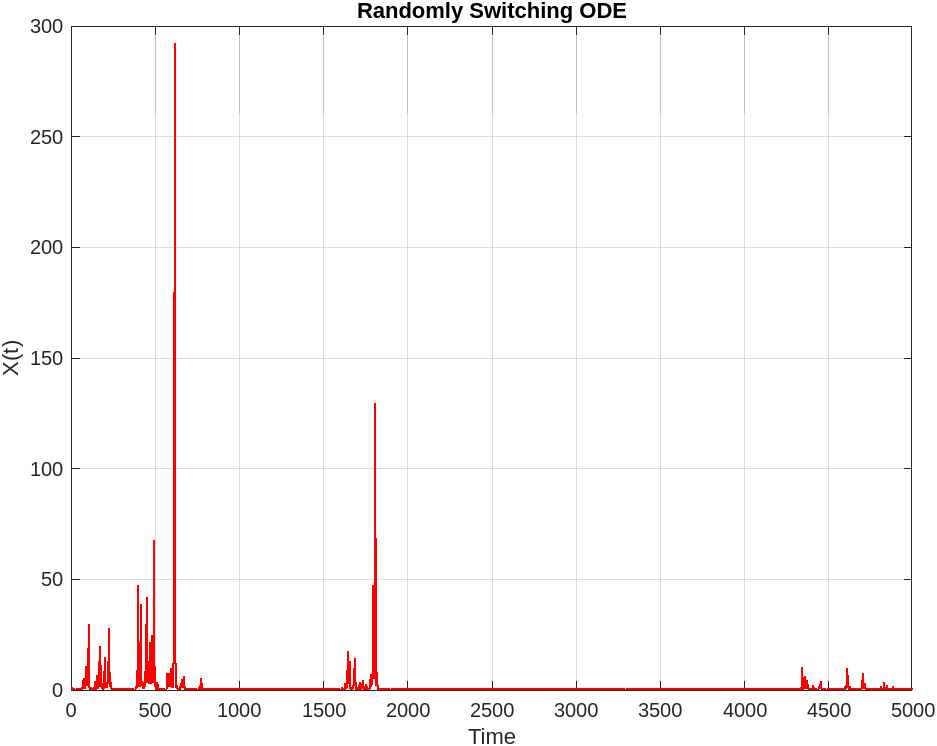}
\caption{$\lambda(\delta_0\times\nu)<0$}  
    \end{subfigure}
      \begin{subfigure}{0.47\textwidth}
    \centering
    \includegraphics[width=\textwidth]{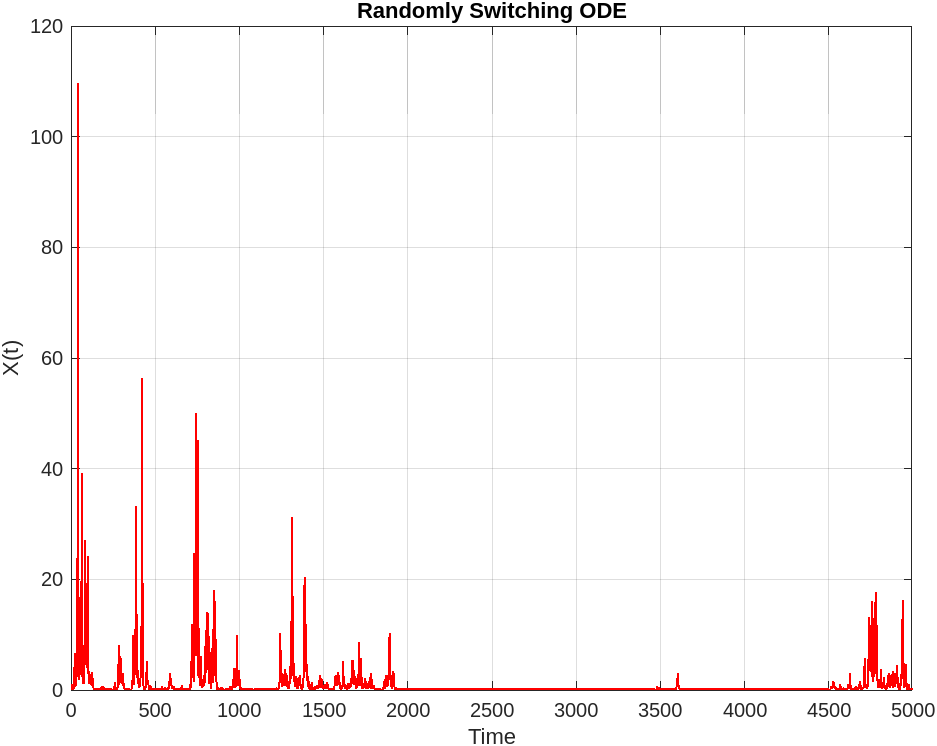}
\caption{$\lambda(\delta_0\times\nu)>0$}  
    \end{subfigure}        
    \caption{The figures shown here have sample paths for the process $(X(t),r(t))$ governed by \ref{s:1d_explosion} when $a_1(2)<0$. On the left we have the sample path corresponding to coefficients $a_1(1)=0.5$, $a_1(2)=-0.505$ and $b_1(2)=0.05$, making $\lambda(\delta_0\times\nu)<0$. On the right we see the sample path for the coefficients $a_1(1)=0.5$, $a_1(2)=0.45$ and $b_1(2)=0.05$, making $\lambda(\delta_0\times\nu)>0$. For both plots we have taken the switching rates $q_{12}=q_{21}=2$, and $X(0)=1$. }
    \label{1D-a_12neg}
\end{figure}

We get the following classification of the long-term dynamics
\begin{itemize}
    \item If $a_1(2)>0$, then $\lambda(\delta_0\times\nu)>0$, and therefore Assumption \ref{a2-pdm} holds. Since Assumption \ref{a3-pdm} is satisfied trivially, it follows from Theorem \ref{thm5.1}, that we have persistence, and there exists a unique invariant distribution $\mu\in \R_+\times\CN$ such that the transition probability measure of the process $(X(t),r(t))$ converges to $\mu$ in total variation. 
    
    \item If $a_1(2)<0$, and $q_{21}>-\frac{q_{12}a_1(2)}{a_1(1)}$, then $\lambda(\delta_0\times\nu)>0$, and we get the same conclusion as in the $a_1(1)>0$ case. This is interesting because it implies that switching between explosion and decay of $X(t)$, can lead to stability. 
    \item If $a_1(2)<0$, and $q_{21}<-\frac{q_{12}a_1(2)}{a_1(1)}$, then $\lambda(\delta_0\times\nu)<0$. Assumptions \ref{a4-pdm} and \ref{a5-pdm} are satisfied and by Theorem \ref{thm5.3} for any $x>0, k\in\CN$ we have extinction, so that
    $$\PP_{x,k}\left(\lim_{t\to\infty}\frac{\ln(X(t))}{t}=\lambda(\delta_0\times\nu)<0\right)=1.$$

\end{itemize}

\subsection{2 Species Competition Dynamics-stability vs explosion.}  \label{s:2d_explosion}

Given that $r(t)\in\CN=\{1,2\}$, and $i\in\{1,2\}$ we consider the following 2-species model 
\begin{equation}\label{e:2d_expl}
\begin{cases}
\frac{dX_1}{dt}(t)=X_1(t)f_1(\BX(t),r(t)) \\
\frac{dX_2}{dt}(t)=X_2(t)f_2(\BX(t),r(t)) 
\end{cases}
\end{equation}
where
$$f_1(x_1,x_2,k)=\begin{cases}
\displaystyle  a_1(1) &\text{if }       k=1\\[.1truein]
\displaystyle  (a_1(2)-b_1(2)x_1-c_1(2)x_2)  &\text{if } k=2,
\end{cases} $$
and
$$f_2(x_1,x_2,k)=\begin{cases}
\displaystyle  a_2(1) &\text{if }       k=1\\[.1truein]
\displaystyle  (a_2(2)-b_2(2)x_2-c_2(2)x_1)  &\text{if } k=2.
\end{cases} $$
We define $Q(x)=(q_{ij}(x))_{2\times 2}$, where $q_{12}(x),q_{21}(x)>0$. Furthermore, we assume that
$$ c\leq\min\{q_{12}(x),q_{21}(x)\}$$ 
$$ \max\{q_{12}(x),q_{21}(x)\}\leq C $$

\begin{rmk}
Using the same idea as in Remark \ref{non_comp_1D}, we can see that there exists no compact subset $K$ of $\R^2_+$, such that if $(X_1(0),X_2(0),r(t))\in K$, then $(X_1(t),X_2(t),r(t))\in K$ for all $t\geq 0$ almost surely. 
\end{rmk}

Define the function
$$F(\bx,k)=\begin{cases}
\displaystyle \alpha(1+\sqrt{x_1}+\sqrt{x_2})  &\text{if }       k=1\\[.1truein]
\displaystyle \alpha\beta(1+\sqrt{x_1}+\sqrt{x_2})   &\text{if } k=2
\end{cases} $$
Where $\alpha>0$, and $\beta\in (0,1)$.

If we have  $\displaystyle c-\frac{\max\{a_1(1),a_2(1)\}}{2}>0$ then we have for $k=1$
\begin{equation}
\begin{split}
\left[\frac{\Lom F(\bx,1)}{F(\bx,1)}+ \delta_0\left(\max_{i}\{|f_i(\bx,1)|\}\right)\right] &= \frac{a_1(1)\alpha\sqrt{x_1}+ a_2(1)\alpha\sqrt{x_2}-2c\alpha(1-\beta)(1+\sqrt{x_1}+\sqrt{x_2})}{\alpha(1+\sqrt{x_1}+\sqrt{x_2})}\\
&+\frac{2\delta_0\max\{a_1(1),a_2(1)\}\alpha(1+\sqrt{x_1}+\sqrt{x_2})}{\alpha(1+\sqrt{x_1}+\sqrt{x_2})}.
\end{split}
\end{equation}
Now if we take $\displaystyle c^*=\min\biggl\{2c-\frac{a_1(1)}{1-\beta},2c-\frac{a_2(1)}{1-\beta}\biggr\}$, then for $\delta_0>0$ and $\beta\in(0,1)$ small enough we have $c^*-\delta_0\max\{a_1(1),a_1(2)\}>0$ and therefore

$$\limsup_{\|x\|\to\infty}\left[\frac{\Lom F(\bx,1)}{F(\bx,1)}+ \delta_0\left(\max_{i}\{|f_i(\bx,1)|\}\right)\right] <0.
$$
This shows that if $\displaystyle c-\frac{\max\{a_1(1),a_2(1)\}}{2}>0$ then Assumption \ref{a1-pdm} is satisfied by the model. From the definition of $F$ it follows immediately that Assumption \ref{a:bound} is also satisfied. 

\begin{rmk}\label{2Dexp_strong}
Define $V(1)=x_1f_1(X(t),1)\partial_{x_1}+ x_2f_2(X(t),2)\partial_{x_2}$, and $V(2)=x_1f_1(X(t),1)\partial_{x_1}+ x_2f_2(X(t),2)\partial_{x_2}$. Then 
\[
V(1)-V(2)=\begin{bmatrix}
(a_1(1)-a_1(2))x_1+b_1(2)x_1^2+c_1(2)x_1x_2 \\
(a_2(1)-a_2(2))x_2+b_2(2)x_2^2+c_2(2)x_1x_2 
\end{bmatrix}
\]

\[
[V(1),V(1)-V(2)]= \begin{bmatrix}
a_1(1)b_1(2)x_1^2 + a_2(1)c_1(2)x_1x_2 \\
a_2(1)b_2(2)x_2^2 + a_1(1)c_2(2)x_1x_2
\end{bmatrix}
\]
Taking  $\displaystyle M= \begin{bmatrix}
V(1)-V(2)   &   [V(1),V(1)-V(2)]\\    
\end{bmatrix}$
we have the following expression for the determinant
\[
det(M)= Bx_1^2x_2^2 + A_1x_1x_2^2 +A_2x_1^2x_2,
\]
where $A_1,A_2,B$ are given by 
\begin{equation}
\begin{split}
A_1=& (a_1(1) - a_1(2))a_2(1)b_2(2) - a_2(1)c_1(2)(a_2(1) - a_2(2))\\
A_2=& (a_1(1) - a_1(2))a_1(1)c_2(2) - a_1(1)b_1(2)(a_2(1) - a_2(2))\\
B  =& (a_2(1) - a_1(1))b_1(2)b_2(2) + (a_1(1) - a_2(1))c_1(2)c_2(2).
\end{split}
\end{equation}
Then, unless $A_1=A_2=B=0$, the strong bracket condition is satisfied everywhere except the set of points where $ Bx_1^2x_2^2 + A_1x_1x_2^2 +A_2x_1^2x_2=0$, which either does not lie in $\R^{2,\circ}_+$, or could be a horizontal line, a vertical line, or a segment of a hyperbola. 

\end{rmk}

\begin{rmk}\label{exp_access}
We assume that for the unique equilibrium point $\bx'=(x_1',x_2')$ of the vector field $V_2$, we have $b_1(2)b_2(2)-c_1(2)c_2(2)>0$, which makes $\bx'$ is globally attracting using the isocline analysis given in \cite{HofSig98}. For any $\eps>0$ let $\bx^{\eps}$ be the unique equilibrium point of the vector field $V_{\eps}=\eps V_1+ (1-\eps)V_2$. Since $b_1(2)b_2(2)-c_1(2)c_2(2)>0$, then $\bx^{\eps}$ is also globally attracting. Let $S$ be the line segment joining $\bx'$ and $\bx^{\eps_0}$ for some $\eps\in (0,1/2)$. If $b_2(2)a_1(1)-c_1(2)a_2(1)\neq 0$, and $b_1(2)a_2(1)-c_2(2)a_1(1)\neq 0$, then there exists at least one point on the line $S$ that does not intersect the set $\{(x_1,x_2)| Bx_1^2x_2^2+A_1x_1x_2^2+A_2x_1^2x_2=0\}$. As a result there is point $\bx_0\in S$ where the strong bracket condition holds, and this point is accessible from any other point from $\R_+^{2,\circ}$. 
\end{rmk}

\begin{rmk}\label{fig3_coefs}
The switching rates chosen for both (A) and (B) in Figure \ref{2D_pdmp} are $q_{12}=q_{21}=2$. The coefficients for (A) are 
$$f_1(X_1,X_2,k)=\begin{cases}
 X_1   &\text{ if }  k=1\\
  X_1 - 0.75X_1^2 - 0.05X_1X_2 &\text{ if } k=2
\end{cases} $$
$$
f_2(X_1,X_2,k)=\begin{cases}
 X_2/2   &\text{ if }  k=1\\
  4X_2 - 0.25X_2^2 - 0.025X_1X_2  &\text{ if } k=2.  
\end{cases}
$$
With the choice of coefficients in (A), Assumption \ref{a3-pdm} is satisfied.The conditions of Remarks \ref{2Dexp_strong} and \ref{exp_access} hold, implying that Assumption \ref{a3-pdm} is satisfied. We can therefore apply Theorem \ref{thm5.1} to get the persistence of species $X_1$ and $X_2$.\\
The coefficients chosen for (B) are 
$$f_1(X_1,X_2,k)=\begin{cases}
 (0.95)X_1   &\text{ if }  k=1\\
  (80/9)X_1 - 8X_1^2 - 7.5X_1X_2 &\text{ if } k=2
\end{cases} $$
$$
f_2(X_1,X_2,k)=\begin{cases}
 X_2   &\text{ if }  k=1\\
  10X_2 - 8X_2^2 - 10X_1X_2  &\text{ if } k=2.
\end{cases}
$$
With this choice of coefficients in (B), the Assumptions \ref{a4-pdm}, \ref{a5-pdm} are satisfied, and the set $(0,\infty)\times \{0\}$ is accessible, so by Theorem \ref{thm5.3}, the species $X_2$ goes extinct.    
\end{rmk}

\begin{figure}
    \centering
    \begin{subfigure}{0.47\textwidth}
    \centering
    \includegraphics[width=\textwidth]{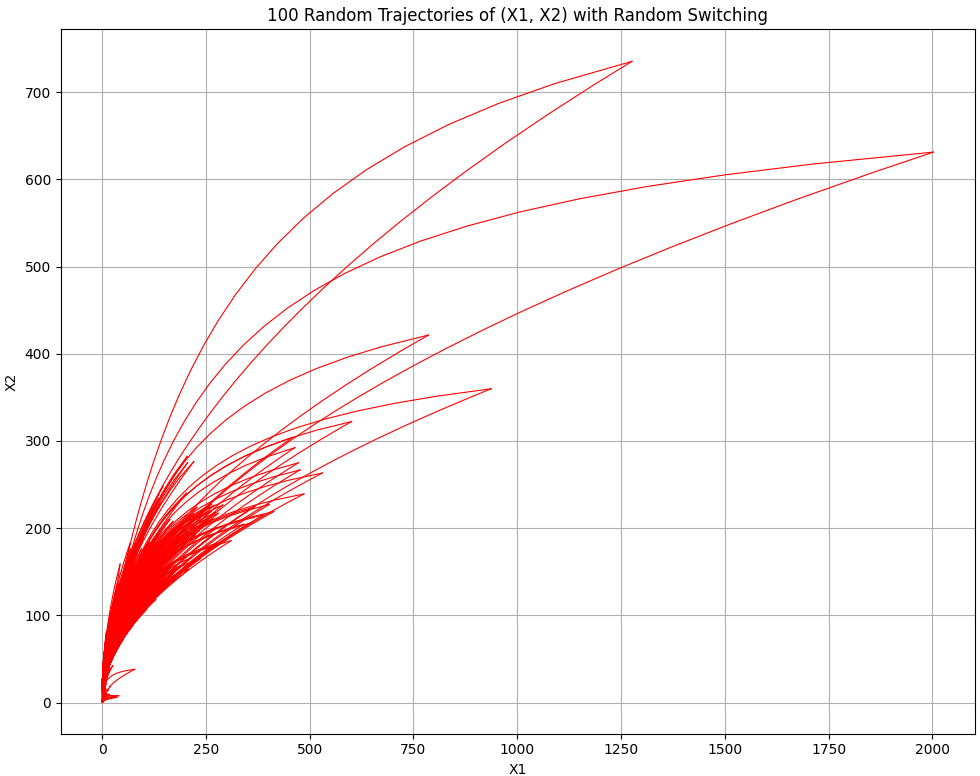}
\caption{}  
    \end{subfigure}
      \begin{subfigure}{0.47\textwidth}
    \centering
    \includegraphics[width=\textwidth]{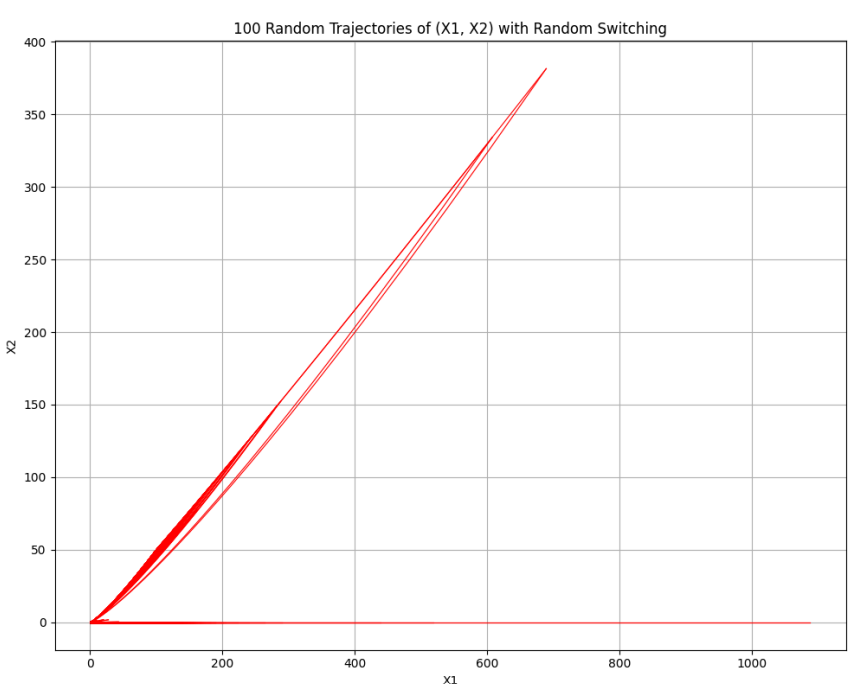}
\caption{}  
    \end{subfigure}        
    \caption{The figures show 100 sample paths of the process from \eqref{e:2d_expl} with $t_{max}=5000$. On the left, the initial condition is $(X_1(0),X_2(0))=(1,1)$, and we see that the sample paths tend to concentrate away from the boundary, showing the persistence of species. On the right, the initial condition chosen is $(50,100)$, the sample paths start away from the boundary, but they concentrate near the $x$ axis, showing extinction of species $X_2$. The coefficients and switching rates chosen are given in Remark \ref{fig3_coefs}.}
    \label{2D_pdmp}
\end{figure}

Let $\nu=(\nu_1,\nu_2)$ be the unique invariant measure on $\CN$, then we know  $\displaystyle (\nu_1,\nu_2)=\left(\frac{q_{21}}{q_{12}+q_{21}},\frac{q_{12}}{q_{12}+q_{21}}  \right)$.
 Then for $\mu_0=\delta_{(0,0)}\times\nu$, we have $$\lambda_1(\mu_0)=\frac{q_{21}a_1(1)}{q_{21}+q_{12}}+\frac{q_{12}a_1(2)}{q_{21}+q_{12}},$$
 and
 $$\lambda_2(\mu_0)=\frac{q_{21}a_2(1)}{q_{12}+q_{21}}+\frac{q_{12}a_2(2)}{q_{12}+q_{21}}.$$
If we restrict our model \ref{e:2d_expl} to $(0,\infty)\times\{0\}\times \CN$ and $\lambda_1(\mu_0)>0$, then by Theorem \ref{thm5.1} there is a unique invariant measure $\mu_1$ supported in $(0,\infty)\times\{0\}\times \CN$, and its density can be explicitly found from Lemma \ref{1d_density}. If we have $\lambda_1(\mu_0)<0$, then there is no invariant measure supported in $(0,\infty)\times\{0\}\times \CN$. Similarly, if $\lambda_2(\mu_0)>0$, we have a unique invariant measure $\mu_2$, supported on $\{0\}\times(0,\infty)\times \CN$.

Using Lemma \ref{1d_density} we can compute $\lambda_2(\mu_1),\lambda_1(\mu_2)$ in terms of the parameters $a_i(k),b_i(k),q_{12},q_{21}$ . The invasion rates are 
\[
\lambda_2(\mu_1) =\int_{supp(\mu_1)} a_2(1) h_{11}(x)dx + \int_{supp(\mu_1)} (a_2(2)-c_2(2)x) h_{12}(x) dx,
\]
and
\[
\lambda_1(\mu_2) =\int_{supp(\mu_2)} a_1(1) h_{21}(x)dx + \int_{supp(\mu_2)} (a_1(2)-c_1(2)x) h_{22}(x) dx.
\]
Where $h_{ik}(x)$ for $i,k\in \{1,2\}$, are smooth functions such that $h_{ik}(x) dx =\mu_i(dx,k)$.

Next, we classify the dynamics.
\begin{itemize}
 \item If either of the following conditions hold
 \begin{enumerate}
     \item $\lambda_1(\mu_0)>0$, $\lambda_2(\mu_0)>0$, $\lambda_2(\mu_1)>0$, $\lambda_1(\mu_2)>0$.
     \item $\lambda_1(\mu_0)>0$, $\lambda_2(\mu_0)<0$, $\lambda_2(\mu_1)>0$
     \item $\lambda_1(\mu_0)<0$, $\lambda_2(\mu_0)>0$, $\lambda_1(\mu_2)>0$
 \end{enumerate}
      then Assumption \ref{a2-pdm} is satisfied, and if the conditions in Remark \ref{exp_access} hold, then Assumption \ref{a3-pdm} is satisfied. We can then use Theorem \ref{thm5.1} to conclude that species $X_1,X_2$ persist and that there exists a unique invariant measure $\pi^*$ supported on $\R_+^{2,\circ}\times\CN$. Moreover, the transition probabilities of the process converge to $\pi^*$ in total variation.
\item If $\lambda_1(\mu_0),\lambda_2(\mu_0)>0$ and $\lambda_2(\mu_1)<0,\lambda_1(\mu_2)<0$, then Assumption \ref{a4-pdm} is satisfied, hence by Theorem \ref{t:exxx_i} since both $\mu_1$ and $\mu_2$ are attractors we get that if we start close enough to either of them with high probability the process converges to that attractor. If in addition $\partial \R_+^2$ is accessible, since Assumption \ref{a5-pdm} is satisfied, we can use Theorem \ref{thm5.3} to conclude that for all $(\bx,k)\in\R^{2,\circ}_+\times\CN$, 
    \[
    \PP_{\bx,k}\{X_2(t)\to \mu_2, X_1(t)\to 0\}+ \PP_{\bx,k}\{X_1(t)\to \mu_1, X_2(t)\to 0\}=1.
    \]

\item If $\lambda_1(\mu_0)>0,\lambda_2(\mu_0)<0$, and $\lambda_2(\mu_1)<0$, then Assumptions \ref{a4-pdm} and  \ref{a5-pdm} are satisfied, so by Theorem \ref{t:exxx_i}, since $\mu_1$ is an attractor, $X_2\to 0$ with high probability if $(X_1(0),X_2(0))$ is sufficiently close to the support of $\mu_1$. If additionally we also have that $(0,\infty)\times \{0\}$ is accessible, then by Theorem \ref{thm5.3} we get 
\[
\PP_{\bx,k}\{X_1\to \mu_1, X_2\to 0\} =1. 
\]

    \item If $\max\{\lambda_1(\mu_0),\lambda_2(\mu_0)\}<0$, then the set $\{(0,0)\}$ is trivially accessible and $\M=\{\delta_{(0,0)}\}$, so by Theorem \ref{thm5.3} we have that for all $(\bx,k)\in\R^{2,\circ}_+\times\CN$, 
    $$ \PP_{\bx,k}\{(X_1(t),X_2(t))\to (0,0)\}=1.$$ 
   
\end{itemize}

\subsection{3d competitive Lotka Volterra models}\label{3D_LV} Consider the piecewise deterministic system 
\begin{equation}\label{e:3d_comp}
\begin{cases}
dX_1(t)=X_1(t)[a_1(r(t))-b_1(r(t))X_1(t)-c_1(r(t))X_2(t)-d_1(r(t))X_3(t)]dt\\
dX_2(t)=X_2(t)[a_2(r(t))-b_2(r(t))X_2(t)-c_2(r(t))X_3(t)-d_2(r(t))X_1(t)]dt\\
dX_3(t)=X_3(t)[a_3(r(t))-b_3(r(t))X_3(t)-c_3(r(t))X_1(t)-d_3(r(t))X_2(t)]dt
\end{cases}
\end{equation}
where $r(t)\in\CN=\{1,2\}$. Assuming $a_i(k),b_i(k),c_i(k),d_i(k)>0$ for $i\in \{1,2,3\}$ and $k\in\CN$.  $Q(x)=(q_{ij}(x))_{2\times2}$ is an irreducible matrix, of the switching intensities, where $q_{12}(x),q_{21}(x)>0$ for all $x\in \R^3_+. $. \\
Taking $F(x)=\alpha(k)(1+x_1+x_2+x_3)$, for $\alpha(k)>0$, $k\in \CN$, assumptions \ref{a1-pdm}, and \ref{a:bound} are satisfied.\\
Since the switching process $r(t)$ is irreducible and aperiodic, there exists a unique invariant measure $\displaystyle\nu=(\nu_1,\nu_2)=\left(\frac{q_{21}}{q_{12}+q_{21}},\frac{q_{12}}{q_{21}+q_{12}}\right)$ on $\CN$. Hence the process has an invariant measure $\displaystyle \mu_0=\delta_{(0,0,0)}\times\CN$.\\
We can see that 
\[
\lambda_i(\mu_0)=a_i(1)\left(\frac{q_{21}}{q_{12}+q_{21}}   \right) + a_i(2) \left(\frac{q_{12}}{q_{12}+q_{21}}\right)>0. 
\]
Since $\lambda_i(\mu_0)>0$, then if we restrict our model on $\R_{i+}$ for $i\in \{1,2,3\}$, then by Theorem \ref{thm5.1} (or \cite{BL16}), there exist invariant measures $\mu_i$ supported on $\R^{\circ}_{i+}\times \CN$, for $i\in \{1,2,3\}$. Then the invasion rates for $\mu_1$ are given by 
    \[
    \lambda_2(\mu_1)=\sum_{k=1}^2 \int \left(a_2(k)-d_2(k)x \right) \mu_1(dx,k)
    \]
    \[
    \lambda_3(\mu_1)=\sum_{k=1}^2\int \left( a_3(k)-c_3(k)x\right) \mu_1(dx,k).    
    \]
    Similarly we can define invasion rates for $\mu_2$ and $\mu_3$ respectively. The invasion rates can be computed explicitly \citep{BL16}.   
   Now, if we restrict our model to $\R^2_{12,+}$, then if  $\lambda_2(\mu_1)>0,\lambda_1(\mu_2)>0$, and either of the following conditions hold
   \begin{enumerate}
       \item $\displaystyle \frac{b_1(1)}{b_2(1)}\neq \frac{b_1(2)}{b_2(2)}$
       \item $\displaystyle \frac{c_1(1)}{c_2(1)}\neq \frac{c_1(2)}{c_2(2)}$
   \end{enumerate}
then by Theorem \ref{thm5.1}, or by \cite{BL16}, there is a unique invariant measure $\mu_{12}$  supported in $\R^{2,\circ}_{12,+}\times \CN$. Using similar conditions we can also show the existence of invariant measures $\mu_{23}$ and $\mu_{31}$ supported on $\R^{2,\circ}_{23,+}\times \CN$ and $\R^{2,\circ}_{31,+}\times \CN$.

If we make the additional assumption that only the birth rates are affected by switching, that is, $b_i(k)=b_i$,  $c_i(k)=c_i$ and $d_i(k)=d_i$, and $b_1b_2-c_1d_2, b_2b_3-c_2d_3\neq 0$ things simplify and we can compute the invasion rates explicitly using the same idea as in example 2.2 from \cite{HN18}. Since we know that $\lambda_i(\mu)=0$ if $i\in I_{\mu}$ by Lemma \ref{lm2.3}, we have

\begin{equation}\label{equi_12}
\begin{split}
\lambda_1(\mu_{12})=&\sum_{k=1}^2\int \left( a_1(k)-b_1x_1-c_1x_2 \right) \mu_{12}(dx,k)=0\\
\lambda_2(\mu_{12})=&\sum_{k=1}^2\int \left( a_2(k)-d_2x_1-b_2x_2 \right) \mu_{12}(dx,k)=0.
\end{split}    
\end{equation}
We also know that $\displaystyle \int \mu_{12}(dx,1)=\frac{q_{21}}{q_{12}+q_{21}}$ and $\displaystyle \int \mu_{12}(dx,2)=\frac{q_{12}}{q_{12}+q_{21}}$. The pair of equations \eqref{equi_12} yield a linear system where we are solving for $\displaystyle \sum_{k=1}^2\int x_1 \mu_{12}(dx,k)$ and $\displaystyle \sum_{k=1}^2\int x_2 \mu_{12}(dx,k)$. As a result, we get
\begin{equation}
\begin{bmatrix}
\sum_{k=1}^{2}\int x_1 \mu(dx,k)   \\
\sum_{k=1}^{2}\int x_2 \mu(dx,k)
\end{bmatrix}=\frac{1}{b_1b_2-c_1d_2} \begin{bmatrix}
b_2 & -c_1 \\
-d_2 & b_1 
\end{bmatrix}\begin{bmatrix}
a_1(1)\frac{q_{21}}{q_{12}+q_{21}}+a_1(2)\frac{q_{12}}{q_{12}+q_{21}}\\
a_2(1)\frac{q_{21}}{q_{12}+q_{21}}+a_2(2)\frac{q_{12}}{q_{12}+q_{21}}.
\end{bmatrix}    
\end{equation}
Now we can calculate $\lambda_3(\mu_{12})$. 
\begin{equation}
\begin{split}
\lambda_3(\mu_{12})=&\sum_{k=1}^2\int \left(a_3(k)-c_3x_1-d_3x_2       \right)\mu_{12}(dx,k)\\
               =& \lambda_3(\mu_0)-\left(\frac{b_2\lambda_1(\mu_0)-c_1\lambda_2(\mu_0)}{b_1b_2-c_1d_2}    \right)c_3 - \left(\frac{b_1\lambda_2(\mu_0)-d_2\lambda_1(\mu_0)}{b_1b_2-c_1d_2} \right)d_3\\
\end{split}    
\end{equation}
With the same idea as above we have the following expressions for $\lambda_2(\mu_{31})$, $\lambda_1(\mu_{23})$,
\begin{equation}
\begin{split}
\lambda_1(\mu_{23})=& \lambda_1(\mu_0) - \left(\frac{b_1\lambda_2(\mu_0)-d_2\lambda_1(\mu_0)}{b_1b_2-c_1d_2}    \right)c_1 -\left(\frac{b_2\lambda_3(\mu_0)-d_3\lambda_2(\mu_0)}{b_2b_3-c_2d_3}\right)d_1\\  
\lambda_2(\mu_{31})=& \lambda_2(\mu_0) - \left(\frac{b_2\lambda_3(\mu_0)-d_3\lambda_2(\mu_0)}{b_2b_3-c_2d_3}    \right)c_2 -\left(\frac{b_2\lambda_1(\mu_0)-c_1\lambda_2(\mu_0)}{b_1b_2-c_1d_2}\right) d_2.   
\end{split}    
\end{equation}

\begin{rmk}
The additional assumption to make only the birth rates depend on the switching process $r(t)$ is a special case, where we are able to calculate invasion rates for measures like $\mu_{12},\mu_{23},\mu_{31}$. \end{rmk}

We are now in a position where we can classify the dynamics. Note that we already assumed that all the species survive on their own so that $\lambda_i(\mu_0)>0, i=1, 2, 3.$
\begin{itemize}

    \item  If any of the following conditions hold
    \begin{enumerate}
        \item $\lambda_2(\mu_1)>0,\lambda_3(\mu_1)>0$, $\lambda_3(\mu_2)>0$,$\lambda_1(\mu_2)>0$, $\lambda_1(\mu_3)>0,\lambda_2(\mu_3)>0$, $\lambda_3(\mu_{12})>0$,$\lambda_1(\mu_{23})>0$,$\lambda_2(\mu_{31})>0$. 
        \item $\lambda_2(\mu_1)>0,\lambda_3(\mu_1)>0$, $\lambda_3(\mu_2)>0$,$\lambda_1(\mu_2)<0$, $\lambda_1(\mu_3)>0,\lambda_2(\mu_3)>0$, $\lambda_1(\mu_{23})>0$,$\lambda_2(\mu_{31})>0$.
        \item $\lambda_3(\mu_2)>0,\lambda_1(\mu_2)>0$, $\lambda_1(\mu_3)>0$,$\lambda_2(\mu_3)<0$, $\lambda_2(\mu_1)<0,\lambda_3(\mu_1)>0$,$\lambda_2(\mu_{31})>0$.
        \item $\lambda_1(\mu_3)>0,\lambda_2(\mu_3)<0$, $\lambda_2(\mu_1)<0$,$\lambda_3(\mu_1)>0$, $\lambda_3(\mu_2)<0,\lambda_1(\mu_2)>0$, $\lambda_2(\mu_{31})>0$.
       
\end{enumerate}
       then Assumption \ref{a2-pdm} is satisfied and the species $X_1,X_2,X_3$ are almost surely stochastically persistent, that is for any $\epsilon>0$ there exists a $\delta>0$ such that $\limsup_{t\to\infty} \tilde\Pi_t(\mathcal{S}^c_{\delta})>1-\epsilon$. If we also know that the accessible set $\Gamma\neq \phi$, and there is a point $x_0\in \Gamma$ such that the strong bracket condition holds at $x_0$, then there exists a unique invariant measure $\pi^*$ supported in $\R^{3,\circ}_+$, such that the transition probability measures converge to $\pi^*$ in total variation. 
       
       \item We assume that if any of the ergodic measures $\mu_{12},\mu_{23},\mu_{31}$ exist, then $\lambda_3(\mu_{12}),\lambda_2(\mu_{31}),\lambda_1(\mu_{23})\neq 0$. In addition, if any of the following conditions hold
       \begin{enumerate}
        \item $\lambda_2(\mu_1)>0,\lambda_3(\mu_1)>0$, $\lambda_3(\mu_2)>0$,$\lambda_1(\mu_2)>0$, $\lambda_1(\mu_3)>0,\lambda_2(\mu_3)>0$, $\min\{\lambda_3(\mu_{12}),\lambda_1(\mu_{23}),\lambda_2(\mu_{31})\}<0$,  . 
        \item $\lambda_2(\mu_1)>0,\lambda_3(\mu_1)>0$, $\lambda_3(\mu_2)>0$,$\lambda_1(\mu_2)<0$, $\lambda_1(\mu_3)>0,\lambda_2(\mu_3)>0$, $\min\{\lambda_1(\mu_{23}),\lambda_2(\mu_{31})\}<0$, .
        \item $\lambda_3(\mu_2)>0,\lambda_1(\mu_2)>0$, $\lambda_1(\mu_3)>0$,$\lambda_2(\mu_3)<0$, $\lambda_2(\mu_1)<0,\lambda_3(\mu_1)>0$,$\lambda_2(\mu_{31})<0$.
        \item $\lambda_1(\mu_3)>0,\lambda_2(\mu_3)<0$, $\lambda_2(\mu_1)<0$,$\lambda_3(\mu_1)>0$, $\lambda_3(\mu_2)<0,\lambda_1(\mu_2)>0$, $\lambda_2(\mu_{31})<0$.   
        \item $\lambda_2(\mu_1)<0$, $\lambda_3(\mu_1)<0$, $\lambda_3(\mu_2)>0$, $\lambda_2(\mu_3)>0$, $\lambda_1(\mu_{23})<0$.
        \item $\lambda_2(\mu_1)<0$, $\lambda_3(\mu_1)<0$, $\lambda_3(\mu_2)>0$, $\lambda_2(\mu_3)>0$, $\lambda_1(\mu_{23})>0$.
        \item $\lambda_2(\mu_1)<0$, $\lambda_3(\mu_1)<0$, $\lambda_3(\mu_2)<0$, $\lambda_1(\mu_2)<0$,   $\lambda_2(\mu_3)>0$.
        \item  $\lambda_3(\mu_1)<0$, $\lambda_2(\mu_1)<0$, $\lambda_3(\mu_2)<0$, $\lambda_1(\mu_2)<0$,   $\lambda_2(\mu_3)<0$, $\lambda_1(\mu_3)<0$.
       \end{enumerate}
       then assumptions \ref{a4-pdm} and \ref{a5-pdm} are satisfied for at least one subset $I\subset\{1,2,3\}$, then one or more of the ergodic measures on the boundary are transversal attractors for our process $(X_1(t),X_2(t),r(t))$. Let $S$ be the set of subsets $I$ such that \ref{a4-pdm} is satisfied, then by Theorem \ref{t:exxx_i}, $X_j\to 0$ for $j\notin I$, with high probability if $(X_1(0),X_2(0))$ is close enough to the support of the ergodic measure corresponding to the set $I$. If we also have that the subspace $\displaystyle\bigcup_{I\in S}\R_+^{I,\circ}$ is accessible for the process, then by Theorem \ref{thm5.3} we have
       \[
       \sum_{I\in S} \PP^I_{\bx,k}=1. 
       \]

\end{itemize}

\begin{rmk}
The classification of the dynamics for the model \eqref{e:3d_comp} uses the ideas of Theorems 3.2 and 3.3 from \cite{HNS22}. The reader is referred to \cite{HNS22} in order to see in more depth the 3d classification of Kolmogorov systems - the theory developed in this paper allows to transfer most of the results of \cite{HNS22} to the PDMP setting. We cannot treat rock-paper-scissor dynamics for \eqref{e:3d_comp} since it violates Assumption \ref{a:bound} and other methods have to be used - the reader is referred to \cite{B23} where this case is treated.   
\end{rmk}

\subsection{Lotka-Volterra food chains}

Here we  recover the results of \cite{Bo23}. The dynamics is given by
\begin{equation}\label{}
\frac{dX_i}{dt}(t) = G^{r(t)}(\BX(t))
\end{equation}
where
\[
G_i^k(\bx) = x_iF_i^k(\bx),
\]
\begin{equation}\label{e:lambda_22}
F_i^k(\bx) =
\begin{cases}
a_{10}(k)-a_{11}x_1-a_{12}x_2 &\text{if } i=1\\
-a_{i0}(k)+a_{i,i-1}x_{i-1}-a_{ii}x_i-a_{i,i+1}x_{i+1}   &\text{if } i=2, \dots,n-1.\\
-a_{n0}(k)+a_{n,n-1}x_{n-1}-a_{nn}x_n &\text{if } i=n,
\end{cases}
\end{equation}
for $k\in\CN=\{1,....N\}$ and $n\geq 3$. It has been shown in \cite{Bo23} that there is a compact set $B\subset \R^n_+$, such that $B\times\CN$ is invariant under $(\BX(t),r(t))$, if $(\BX(0),r(0))\in B\times\CN$. We keep Assumption 1.2 from \cite{Bo23} which says that only the birth rate of the primary producer species ($X_1$) depends on the switching process.

Since the switching process $r(t)$ is irreducible and aperiodic, there exists a unique invariant measure $\nu=(\nu_1,\dots \nu_N)$ on $\CN$. The measure $\mu_0=\delta_{\0}\times\nu$ is ergodic and we can see that 
$\displaystyle\lambda_1(\mu_0)= \sum_{k=1}^N a_{10}(k)\nu_k >0$, and $\lambda_j(\mu_0)=-a_{j0}$ for $j\neq 1$. If we restrict the model to $\R_{1,+}$, then since $\lambda_1(\mu_0)>0$, then by Theorem \ref{thm5.1} there exists a unique invariant measure $\mu_1$ supported on $\R_{1,+}^{\circ}\times \CN$, and the expression for $\lambda_2(\mu_1)$ for $j\neq 1$ is given by 
\[
\lambda_2(\mu_1)= \sum_{j=1}^N \int \left(a_{21}x_1-a_{20}\right) \mu_1(dx,k). 
\]
If $|\CN|=2$, then we can compute $\lambda_2(\mu_1)$ explicitly, using the same idea as in \ref{3D_LV}, since assumption 1.2 in \cite{Bo23} holds. Hence we get $\displaystyle \lambda_2(\mu_1)=\frac{a_{21}\lambda_1(\mu_0)}{a_{11}}-a_{20}>0$. 
We can see that $\lambda_j(\mu_0)=0$ for $j>2$, and also that since $\lambda_i(\mu_0)<0$ for $i\neq 1$, then by Theorem \ref{thm5.3}, there is no invariant distribution supported on $\R_{i,+}^{\circ}\times \CN$. Proceeding similarly, if we restrict the model to $\R^2_{12,+}$, then if $\lambda_2(\mu_1)>0$, then by Theorem \ref{thm5.1} there exists a unique invariant measure $\mu_{12}$ supported in $\R_{12,+}\times \CN$. The expression for $\lambda_3(\mu_{12})$ is given by 
\[
\lambda_3(\mu_{12})= \sum_{j=1}^N\int \left(a_{32}x_2-a_{30}\right) \mu_{12}(dx_1dx_2,k)
\]
If, for simplicity, $|\CN|=2$, we have the following explicit expression for $\lambda_3(\mu_{12})$
\[
\lambda_3(\mu_{12})=a_{32}\left(\frac{a_{21}\lambda_1(\mu_0)-a_{11}a_{20}}{a_{11}a_{22}+ a_{12}a_{21}}\right)-a_{30}.
\]
If we have $\lambda_2(\mu_1)<0$, then by Theorem \ref{thm5.3} there is no invariant measure supported in $\R_{12,+}^{\circ}\times\CN$. Hence, proceeding inductively we can see that there is at most one invariant measure supported in $\R^{I_{k},o}_+$, where $I_k=\{1,\dots, k\}$ for $k\in\{1,\dots, n-1\}$. 

We now classify the dynamics 
\begin{itemize}
\item If we have $\lambda_{k+1}(\mu_{I_k})>0$ where $I_k=\{1, \dots k\}$ for each $k\in \{1,\dots ,n-1\}$, then Assumption \ref{a3-pdm} is satisfied, and since Assumption \ref{a3-pdm} follows from assumption 1.2 in \cite{Bo23}, then by Theorem \ref{thm5.1}, all of the species persist, and there is a unique invariant measure $\pi^*$, supported in $\R^{n,\circ}_+$, such that the transition probability measures converge to $\pi^*$ in total variation. 

\item If for some $k\in \{1,\dots, n-1\}$, we have that $\lambda_{k+1}(\mu_{I_k})<0$, then Assumptions \ref{a4-pdm} and \ref{a5-pdm} are satisfied, and the measure $\mu_{I_k}$ is a transversal attractor. By Theorem \ref{t:exxx_i}, then $X_j(t)\to 0$ at with high probability if $(X(0),r(0))$ is close enough to the support of $\mu_{I_k}$. If additionally we know that $\R^{I_k,\circ}_+$ is accessible for the process then by Theorem \ref{thm5.3}, we have 
\[
\PP_{\bx,k}\{ (X_1,\dots X_k)\to \mu_{I_k}, (X_{k+1},\dots X_n)\to \0 \}=1.
\]
       
\end{itemize}

\section{Preliminary lemmas}\label{s:prep}

In this section we will prove some preliminary things about the process such as well-posedness, uniqueness, and tightness. The proof of the lemmas will be given in Appendix \ref{a:inv measures} but we will provide some ideas and sketches of proofs below.

We first define a Lyapunov type function $V(\bx,k) : \R^{n,\circ}_+\to\R_+$ such that $V(\bx,k) \to \infty$ as $\bx\to \partial \R_+^n$. It will be used to prove various persistence results in Section \ref{s:perm}.

For $\bp=(p_1,\cdots,p_n)\in\R^{n,\circ}_+$, $\|\bp\|\leq\delta_0$,
define the function $V: \R^{n,\circ}_+\to\R_+$ by
\begin{equation}\label{e:V}
V(\bx,k):=\dfrac{F(\bx,k)}{\prod_i x_i^{p_i}}.
\end{equation}
Using the condition \eqref{e:sup} one can define 
\begin{equation}\label{e:H}
\begin{aligned}
H:=\sup\limits_{\bx\in\R^{n}_+,k\in\M}\Bigg\{\dfrac{\Lom F(\bx,k)}{F(\bx,k)}+
\gamma_b+\delta_0\sup_{k\in\M}\{|f_i(\bx,k)|\}\Bigg\}<\infty.
\end{aligned}
\end{equation}

In the first lemma we show that \eqref{e1-pdm} has a unique strong solution which leaves open subspaces of the form $\R^{I,\circ}_+$ invariant. We also prove that the expectation of $V (\BX(t),r(t))$ grows at most exponentially in time. 

\begin{lm}\label{lm2.0}
For any $\bx\in\R^n_+$, there exists a pathwise unique strong solution $(\BX(t),r(t))$ to \eqref{e1-pdm}
with initial value $\BX(0)=\bx, r(0)=k$.
Let $I\subset\{1,\cdots,n\}$ and $\bx\in\R^{I,\circ}_+$
where $$\R^{I,\circ}_+=\left\{\bx\in \R^n_+: x_i=0 \text{ if } i\notin I \text{ and }x_i>0 \text{ if } i\in I\right\}.$$
The solution $(\BX(t),r(t))$ with initial value $\BX(0)=\bx, r(0)=k$ will satisfy $\BX(t)\in\R^{I,\circ}_+, t\geq 0$ with probability 1.
Moreover, for $\bx\in\R^{n,\circ}_+$ and $V$ defined by \eqref{e:V}, we have
\begin{equation}\label{e1-lm2.2}
\E_{\bx,k} V (\BX(t),r(t))\leq \exp( Ht) V (\bx,k).
\end{equation}
\end{lm}

The next lemma gives us some important bounds in time on the expectation of $F(\BX(t),r(t))$ and $\int_0^t  (F(\BX(s),r(s))) \left[1+f^M(\BX(s)\right]ds$ and also shows that the process $(\BX(t),r(t))$ is a nice Feller process. We will be able to use these results in order to prove the tightness of various families of measures.  
\begin{lm}\label{lm2.2}
There are $H_1, H_2>0$ such that for any $\bx\in\R^n_+, t>0$
\begin{equation}\label{e2-lm2.2}
\E_{\bx,k} (F(\BX(t),r(t))) \leq H_1+(F(\bx,k))e^{-\gamma_bt} 
\end{equation}
and
\begin{equation}\label{e3-lm2.2}
\E_{\bx,k}\int_0^t  (F(\BX(s),r(s))) \left[1+f^M(\BX(s)\right]ds\leq H_2((F(\bx,k)) +t).
\end{equation}
Moreover,  the solution process $(\BX(t),r(t))$ is a Markov-Feller process in $\R^n_+$.
\end{lm}
\begin{proof}[Idea behind proof]
The proof of the lemma is straightforward and makes use of Dynkin's formula and the bound \eqref{e:H}.    
\end{proof}
Define the family of measures
\begin{equation}\label{e:Pi_meas}
\Pi^{\bx,k}_t(\cdot):=\dfrac1t\int_0^t\PP_{\bx,k}\{(X(s),r(s))\in\cdot\}ds,\,(\bx,k)\in\R^n_+\times\CN, t>0.
\end{equation}
These measures are important as they are Cesaro-type integral averages involving the distribution of the process on the interval $[0,t]$.

The next result gives us a handle on some properties of the invariant probability measures of the process. The key fact is showing that
\[
\sum_{k\in\CN}\int_{\R^n_+} \Lom \ln F(\bx,k) \mu(d\bx,k)=0
\]
which will then be used in the persistence proofs together with the Feller property in order to show that certain integral averages of $\Lom \ln F(\bx,k)$ are close to zero. We also prove that while at `equilibrium' within the measure $\mu$ the invasion rates of `internal' species are zero and 
\[
\lim_{t\to\infty} \frac{\ln X_i(t)}{t}=\lambda_i(\mu)=0, i\in I_\mu.
\]
\begin{lm}\label{lm2.3}
Let $\mu$ be an invariant probability measure of $\BX$.
Then
$$\sum_{k\in\CN}\int_{\R^n_+}(F(\bx,k)) \left(1+f^M(\bx,k)\right)\mu(d\bx,k)\leq H_2$$
and
$$\sum_{k\in\CN}\int_{\R^n_+} \Lom \ln F \mu(d\bx,k) =\sum_{k\in\CN}\int_{\R^n_+}\left(\dfrac{\Lom F(\bx,k)}{F(\bx,k)}-\sum_{\ell\in\CN} q_{k\ell}(\bx)\left(\frac{F(\bx,\ell)}{F(\bx,k)}-1-\ln\frac{F(\bx,\ell)}{F(\bx,k)} \right)\right)\mu(d\bx,k)=0.$$

Moreover, if $\mu$ is ergodic and $i\in I^\mu$ then
\[
\lambda_i(\mu)=0
\]
and for $\mu$ almost all $(\bx,k)\in supp(\mu) $ if $(\BX(0),r(0))=(\bx,k)$ we have 
\[
\PP_{\bx,k}\left(\lim_{t\to\infty} \frac{\ln(F(\BX(t),r(t)))}{t}=0\right)=1.
\]
and 
\[
\PP_{\bx,k}\left(\lim_{t\to\infty} \frac{\ln X_i(t)}{t}=\lambda_i(\mu)=0\right), i\in I_\mu.
\]
\end{lm}

\begin{proof}[Idea behind proof]
Using Ito's lemma one has
\begin{equation}\label{e:expl1}
\frac{\ln(F(\BX(t),r(t)))}{t}=\frac{\ln(F(\BX(0),r(0)))}{t}+\frac{1}{t}\int_0^t  \Lom(\ln(F(\BX(s),r(s)))) ds  +  \frac{M(t)}{t}
\end{equation}
where $(M(t))$ is a local martingale. We can show that it is actually a true martingale and that it is $L^2$. The hardest part is showing that the strong law of large numbers holds so that with probability one
\[
\lim_{t\to\infty}\frac{M(t)}{t}=0.
\]
This is done by proving bounds for the predictable quadratic variation by making use of the bound \eqref{a:bound}.
Using Birkhoff's ergodic theorem we will get
\[
\lim_{t\to\infty}\frac{1}{t}\int_0^t \Lom(\ln(F(\BX(s),r(s)))) ds =\sum_{k\in\CN}\int_{\R^n_+} \mathscr{\Lom}(\ln(F(\bx,k)) \mu(dx,k).
\]
Another application of Birkhoff's ergodic theorem can be shown to imply that 
\[
\frac{\ln(F(\BX(t),r(t)))}{t}=0.
\]
Using these facts in \eqref{e:expl1} yields 
\[
\sum_{k\in\CN}\int_{\R^n_+} \Lom \ln F(\bx,k) \mu(d\bx,k)=0
\]
\end{proof}

The next technical lemma allows us to take the limits of integrals with respect to the Cesaro average measures from \eqref{e:Pi_meas}. 
\begin{lm}\label{lm2.4}
Suppose the following
\begin{itemize}
\item  The sequences $(\bx_m,k_m)_{m\in \N}\subset \R_+^n, (T_m)_{m\in \N}\subset \R_+$ are such that $\|\bx_m\|\leq M$, $T_m>1$ for all $m\in \N$ and $\lim_{m\to\infty}T_m=\infty$.

\item The sequence of probability measures $(\Pi^{\bx_m,k_m}_{T_m})_{m\in \N}$ converges weakly to an invariant probability measure
$\pi$.

\item Let $h:\R^n_+\times\CN\to\R$ be a continuous function satisfying
$|h(\bx,k)|<K_h(F(\bx,k))^\delta(1+f^M(\bx,k)))$, $\bx\in \R^n_+$,
for some $K_h\geq 0$, $\delta<1$.
\end{itemize}
Then one has
\[\lim_{m\to\infty}\sum_{k\in\CN}\int_{\R^n_+}h(\bx,k)\Pi^{\bx_m,k_m}_{T_m}(d\bx,k_m)= \sum_{k\in\CN}\int_{\R^n_+}h(\bx,k)\pi(d\bx,k).\]
\end{lm}

Finally, the last lemma, which was proven in \cite{B23}, will help us prove uniform estimates for a Taylor expansion.
\begin{lm}\label{lm2.5}
Let $Y$ be a random variable, $\theta_0>0$ a constant, and suppose $$\E \exp(\theta Y)+\E \exp(-\theta Y)\}\leq K_1$$ for any $\theta\in[0,\theta_0]$.
Then the log-Laplace transform
$\phi(\theta)=\ln\E\exp(\theta Y)$
is twice differentiable on $\left[0,\frac{\theta_0}2\right)$ and
$$\dfrac{d\phi}{d\theta}(0)= \E Y,$$
$$0\leq \dfrac{d^2\phi}{d\theta^2}(\theta)\leq K_2\,, \theta\in\left[0,\frac{\theta_0}2\right)$$
 for some $K_2>0$.
\end{lm}

\section{Persistence}\label{s:perm}
This section is devoted to finding conditions under which the process $(\BX(t),r(t))$ exhibits persistence. In particular we show when the process converges to a unique invariant probability measure supported on $\R^{n,\circ}_+$.

We begin the section by giving an overview of the lemmas and propositions and by sketching the proof of the main results.

It is shown in \cite[Lemma 4]{SBA11} by the minmax principle that Assumption \ref{a2-pdm} is equivalent to the existence of $\mathbf p>0$ such that
\begin{equation}\label{e.p}
\min\limits_{\mu\in\M}\left\{\sum_{i}p_i\lambda_i(\mu)\right\}:=2\rho^*>0.
\end{equation}
By rescaling if necessary, we can assume that $\|\bp\|=\delta_0$.

\begin{proof}[Idea of proof of Theorem \ref{thm3.1}]
The main ingredient in proving persistence is showing that there are constants $\theta\in (0,1), \kappa\in (0,1), \tilde K>0, \bar T>0$ such that there is a `contraction' of the semigroup, in the sense that for all $(\bx,k)\in\R^{n,\circ}_+\times\CN$ one has
\begin{equation}\label{e:expl_pers}
\E_{\bx,k} V^\theta(\BX(\bar T), r(\bar T))\leq \kappa V^\theta(\bx,k)+\tilde K
\end{equation}
where 
$$V^\theta(\bx,k) = (V(\bx,k))^\theta= \left(\dfrac{F(\bx,k)}{\prod_i x_i^{p_i}}\right)^\theta.$$

This is done by first proving an estimate of the form \eqref{e:expl_pers} in Proposition \ref{prop2.1} when $\|\bx\|\leq M$. To achieve this we look at
\begin{equation}\label{e:G_expl}
\ln V(\BX(T), r(T))=\ln V(\BX(0), r(0)) + G(T)
\end{equation}
where
\begin{equation}\label{e:G2_expl}
\begin{aligned}
G(T)=\int_0^T\Phi(\BX(t), r(t))dt+ M_G(T). 
\end{aligned}
\end{equation}
Note that it is key here to use the persistence Assumption \ref{a2-pdm} which tells us intuitively that the boundary $\partial \R_+^n$, i.e. the invariant measures living on the boundary, is `repelling'. 
Rewriting sumption \ref{a2-pdm} as
\[
\min\limits_{\mu\in\M}\left\{\sum_{i}p_i\lambda_i(\mu)\right\}:=2\rho^*>0.
\]
allows us to prove in Lemma \ref{lm3.1} that for for any $T$ large enough, and $\bx\in\partial\R^n_+, \|\bx\|\leq M$ one has
\begin{equation}
\dfrac1T\int_0^T\E_{\bx,k}\Phi(\BX(t), r(t))dt\leq-\rho^*.
\end{equation}
This estimate ensures that the constant $\kappa$ from \eqref{e:expl_pers} is sub-unitary.

Using Lemma \ref{lm3.1}, and the fact that $M_G(t)$ is an $L^2$ martingale, together with Lemma \ref{lm2.5} and the Feller property allows one to control $\E_{\bx,k} G(T)$ uniformly for $\|\bx\|\leq M$ close enough to the extinction boundary $\partial \R_+^n$ and for $T$ in a compact interval. This will allow one to finish the proof of Proposition \ref{prop2.1}.

Next, one can show that for large values of $\|\bx\}$, i.e., for $\|\bx\|>M$ one has
\[
\Lom V^\theta (\bx,k)\leq -\theta \gamma_b V^\theta(\bx,k).
\]
This implies that there is a strong enough drift making the process return into compact sets. The strong Markov property allows us to show that \eqref{e:expl_pers} holds for $\|\bx\|>M$. As a result of this and Proposition \ref{prop2.1} we get that  \eqref{e:expl_pers} for all $\bx\in \R_+^{n,\circ}$. Once this is established, one can prove the various persistence results. Note that as we approach $\partial \R_+^n$ or $\bx\to\infty$ we have 
\begin{equation}\label{e:expl_V}
V^\theta (\bx,k) \to \infty,
\end{equation}
which together with  \eqref{e:expl_pers} allows the control at the boundary and at infinity. For example, by \eqref{e:expl_pers} one can show 
\[
\limsup_{t\to\infty} \E_{\bx,k} V^\theta(\BX(t), r(t))\leq \frac{\exp( \theta H\bar T)}{1-\kappa}.
\]
This with \eqref{e:expl_V} will imply that for any $\eps>0$ there exists $\delta>0$ such that for all $(\bx,k)\in\R^{n,\circ}_+\times\CN$
\[
\liminf_{t\to\infty} \PP_{\bx,k} \{\min_i X_i(t)\geq \delta\} \geq 1-\eps,
\]
which shows persistence in probability. 

If Assumption \ref{a3-pdm} holds one can show that the Markov chain $((\BX(\ell \bar T), r(\ell \bar T))$ is irreducible and aperiodic and that every subset of $\R_+^{n,\circ}\times\CN$ is petite. This then allows to use the theory from \cite{MT} in order to show that $((\BX(\ell \bar T), r(\ell \bar T))$ converges to a unique invariant probability measure on $\R_+^{n,\circ}\times\CN$ 

\end{proof}
\begin{lm}\label{lm3.1}
Suppose that Assumption \ref{a2-pdm} holds. Let $\bp$ and $\rho^*$ be as in \eqref{e.p}.
There exists a $T^*>0$ such that, for any $T>T^*$, $\bx\in\partial\R^n_+, \|\bx\|\leq M$ one has
\begin{equation}
\dfrac1T\int_0^T\E_{\bx,k}\Phi(\BX(t), r(t))dt\leq-\rho^*
\end{equation}
where
\begin{equation}
\begin{aligned}
\Phi(\bx,k):=\dfrac{\Lom F(\bx,k)}{F(\bx,k)}-\sum_{\ell\in\CN} q_{k\ell}(\bx)\left(\frac{F(\bx,\ell)}{F(\bx,k)}-1-\ln\frac{F(\bx,\ell)}{F(\bx,k)} \right)
-\sum_{i}p_if_i(\bx,k).
\end{aligned}
\end{equation}
\end{lm}
\begin{proof}
We argue by contradiction. Suppose that the conclusion of this lemma is not true.
Then, we can find $\bx_m\in\partial\R^n_+, \|\bx_m\|\leq M$
and $T_m>0$, $\lim_{m\to\infty} T_m=\infty$
such that
\begin{equation}\label{e3.9}
\dfrac1{T_m}\int_0^{T_m}\E_{\bx_m,k_m}\Phi(\BX(t), r(t))dt>-\rho^*\,,\,k\in\N.
\end{equation}
It follows from Lemma \ref{lm2.2} that  $\left(\Pi^{\bx_m,k_m}_{T_m}\right)_{m\in \N}$ is tight - see the proof of Lemma \ref{lm2.3} in the appendix for the details. As a result
$\left(\Pi^{\bx_m,k_m}_{T_m}\right)_{k\in \N}$ has a convergent subsequence in the weak$^*$-topology.
Without loss of generality, we can suppose that $\left\{\Pi^{\bx_m,k_m}_{T_m}:k\in\N\right\}$
is a convergent sequence in the weak$^*$-topology.

It can be shown (by \cite[Theorem 9.9]{EK09} or by \cite[Proposition 6.4]{EHS15}) that its limit is an invariant probability measure $\mu$ of $(\BX(t),r(t))$.
As a consequence of Lemma \ref{lm2.4}
$$\lim_{m\to\infty}\dfrac1{T_m}\int_0^{T_m}\E_{\bx_m,k_m}\Phi(\BX(t), r(t))dt=\sum_{k\in\CN}\int_{\R^n_+}\Phi(\bx,k)\mu(d\bx,k).$$
In view of Lemma \ref{lm2.3} and \eqref{e.p} we get that
$$\lim_{m\to\infty}\dfrac1{T_m}\int_0^{T_m}\E_{\bx_m,k_m}\Phi(\BX(t), r(t))dt=-\sum_{i=1}^np_i\lambda_i(\mu)\leq -2\rho^*,$$
which contradicts \eqref{e3.9}.
\end{proof}

From now on let $n^*\in\N$ such that
\begin{equation}\label{e:n*}
\gamma_b(n^*-1)>H+1.
\end{equation}

\begin{prop}\label{prop2.1}
Let $V(\cdot)$ be defined by \eqref{e:V} with $\bp$ and $\rho^*$ satisfying \eqref{e.p} and $T^*>0$ satisfying the assumptions of Lemma \ref{lm3.1}.
There are $\theta\in\left(0,\frac{1}2\right)$, $K_\theta>0$, such that for any $T\in[T^*,n^*T^*]$ and $\bx\in\R^{n,\circ}_+, \|\bx\|\leq M$,
\begin{equation}\label{e:V_theta_small_x}
\E_{\bx,k} V^\theta(\BX(T), r(T))\leq \exp(-0.5\theta \rho^*T) V^\theta(\bx,k)+K_\theta.
\end{equation}
\end{prop}
\begin{proof}
We have from It\^o's formula that
\begin{equation}\label{e:G}
\ln V(\BX(T), r(T))=\ln V(\BX(0), r(0)) + G(T)
\end{equation}
where
\begin{equation}
\begin{aligned}
G(T)=\int_0^T\Phi(\BX(t), r(t))dt+ M_G(T).
\end{aligned}
\end{equation}
where
$$M_G(T)=\int_0^T\int_{\R_+} \ln\left(\frac{F(\BX(t),\alpha(t)+h(\BX(t),\alpha(t),y))}{F(\BX(t),\alpha(t))}\right)\tilde N(dt,dy)
$$
where $h(x,i,z)=\sum_{j=1}^{n_0} (j-i)I_{\{z\in \Delta_{ij}(x)\}}$, $\Delta_{ij}(x)$ being the consecutive left-closed, right-open intervals of the real line, each having length $q_{ij}(x)$ as described in \cite{YZ10} . The martingale measure $\tilde N(dt,dy) = N(dt,dy) - dt\times dy$ comes from the Poisson measure $N(dt,dy)$ with intensity $dt\times dy$.
Note that $M_G(T)$ is a martingale from \ref{c:a1}.
In view of \eqref{e:G} and \eqref{e1-lm2.2}
\begin{equation}\label{e3.4_2}
\E_{\bx,k} \exp(G(T))=\dfrac{\E_{\bx,k} V (\BX(T), r(T))}{V (\bx,k)}\leq  \exp( HT).
\end{equation}
where $M_G(t)$ is a martingale with predictable quadratic variation
\[
\langle M_G,M_G\rangle_t = \int_0^T\int_{\R_+} \left(\ln\left[ 
 \frac{F(\BX(t),\alpha(t)+h(\BX(t),\alpha(t),y))}{F(\BX(t),r(t))}  \right ]^2    \right) dy dt
\]
We note that the proof that $M_G$ is a martingale appears in the appendix, when we prove the Lemmas from Sections \ref{s:prep} and \ref{s:extin} - see for example the proof of Lemma \ref{lm2.3}.

We define the function $\hat V:\R^{n}_+\times\CN\to \R_+$ as 
\[
\hat V(\bx,k)= F(\bx,k)\left(\prod_{i=1}^n \bx_i^{p_i}\right)
\]
then computing the generator we get 
\[
\Lom(\hat V(\bx,k))= \sum_{i=1}^nf_i(\bx,k)\left(\bx_i\frac{\partial F}{\partial x_i}(\bx,k)+ p_iF(\bx,k)     \right)\left(\prod_{i=1}^n\bx_i^{p_i}\right) + \sum_{l\in\CN}q_{kl}(\bx)\hat V(\bx,k)
\]
we can simplify the above expression to 
\[
\Lom(\hat V(\bx,k))=\Lom(F(\bx,k))\left(\prod_{i=1}^n \bx_i^{p_i}\right)+ \sum_{i=1}^n p_if_i(\bx,k)\hat V(\bx,k)
\]
hence from \eqref{e:H}
\begin{equation}\label{Lvhat}
\Lom(\hat V(\bx,k))\leq H \hat V(\bx,k)    
\end{equation}

Applying Dynkin's Formula to $\phi(t)=\hat V(\BX(t),r(t))$, then 
\begin{equation}\label{vhatDyn}
\E_{\bx,k}(\hat V(\BX(t),r(t)))=\hat V(\bx,k)+ \E_{\bx,k}\left(\int_0^t \Lom(\hat V(\BX(s),r(s))) ds   \right).   
\end{equation}
then using \eqref{Lvhat},\eqref{vhatDyn} and Gronwall's inequality we get
\begin{equation}\label{vhat-1}
\dfrac{\E_{\bx,k} \hat V (\BX(T), r(T))}{\hat V (\bx,k)}\leq  \exp( HT).
\end{equation}
Note that
\begin{equation}\label{vhat-2}
V^{-1}(\bx,k)=\hat V (\bx,k)(F(\bx,k))^{-2}\leq\frac{1}{c^2}( \hat V (\bx,k)).
\end{equation}
Applying \eqref{vhat-2} to \eqref{vhat-1} yields
\begin{equation}\label{e3.5}
\begin{aligned}
\E_{\bx,k} \exp(-  G(T))=&\dfrac{\E_{\bx,k} V^{- 1}(\BX(T), r(T))}{V^{- 1}(\bx,k)}\\
\leq&\dfrac{\E_{\bx,k}\hat V (\BX(T), r(T))}{c^2V^{-1}(\bx,k)}\\
\leq& \dfrac{\E_{\bx,k}\hat V (\BX(T), r(T))}{c^2\hat V (\bx,k)}(F(\bx,k))^{2}\\
\leq& \left(\frac{F(\bx,k)}{c}\right)^{2}\exp(  HT).
\end{aligned}
\end{equation}

By \eqref{e3.4_2} and \eqref{e3.5} the assumptions of Lemma \ref{lm2.5} hold for $G(T)$. Therefore,
there is $\tilde K_2\geq 0$ such that
$$0\leq \dfrac{d^2\tilde\phi_{\bx,k,T}}{d\theta^2}(\theta)\leq \tilde K_2\,\text{ for all }\,\theta\in\left[0,\frac{1}2\right),\, \bx\in\R^{n,\circ}_+, \|\bx\|\leq M, T\in [T^*,n^*T^*]$$
where
$$\tilde\phi_{\bx,k,T}(\theta)=\ln\E_{\bx,k} \exp(\theta G(T)).$$
In view of Lemma \ref{lm3.1} and the Feller property of $(\BX(t),r(t))$,
there exists $\tilde\delta>0$ such that
if $\|\bx\|\leq M$, $\dist(\bx,\partial\R^n_+)<\tilde\delta$ and $T\in [T^*,n^*T^*]$
then
\begin{equation}\label{e3.6}
\begin{aligned}
\E_{\bx,k} G(T)=\int_0^T\E_{\bx,k} \dfrac{\Lom F(\BX(t),r(t)) }{F(\BX(t),r(t))}
-\sum_{i=1}^np_i\int_0^T\E_{\bx,k}f_i(\BX(t),r(t))dt\leq -\dfrac34\rho^*T.
\end{aligned}
\end{equation}
An application of Lemma \ref{lm2.5} yields
$$\dfrac{d\tilde\phi_{\bx,k,T}}{d\theta}(0)=\E_{\bx,k} G(T)\leq -\dfrac34\rho^*T.$$
By a Taylor expansion around $\bx=0$, for $\|\bx\|\leq M, \dist(\bx,\partial\R^n_+)<\tilde\delta, T\in [T^*,n^*T^*]$ and $\theta\in\left[0,\frac{1}2\right)$ we have
$$\tilde\phi_{\bx,k,T}(\theta)\leq -\dfrac34\rho^*T\theta+\theta^2\tilde K_2 .$$
If we choose any $\theta\in\left(0,\frac{1}2\right)$ satisfying
$\theta<\frac{\rho^*T^*}{4\tilde K_2}$, we obtain that
\begin{equation}\label{e3.10}
\tilde\phi_{\bx,k,T}(\theta)\leq -\dfrac12\rho^*T\theta\,\,\text{ for all }\,\bx\in\R^{n,\circ},\|\bx\|\leq M, \dist(\bx,\partial\R^n_+)<\tilde\delta, T\in [T^*,n^*T^*].
\end{equation}
In light of \eqref{e3.10}, we have for such $\theta$ and $\|\bx\|\leq M, 0<\dist(\bx,\partial\R^n_+)<\tilde\delta, T\in [T^*,n^*T^*]$ that
\begin{equation}\label{e3.11}
\dfrac{\E_{\bx,k} V^\theta(\BX(T), r(T))}{V^\theta(\bx,k)}=\exp \tilde\phi_{\bx,k,T}(\theta)\leq\exp(-0.5p^*T\theta).
\end{equation}
In view of \eqref{e1-lm2.2},
we have for $\bx$ satisfying $\|\bx\|\leq M, \dist(\bx,\partial\R^n_+)\geq\tilde\delta$ and $T\in  [T^*,n^*T^*]$ that
\begin{equation}\label{e3.12}
\E_{\bx,k} V^\theta(\BX(T), r(T))\leq \exp(\theta n^*T^*H)\sup\limits_{\|\bx\|\leq M, \dist(\bx,\partial\R^n_+)\geq\tilde\delta}\{V(\bx,k)\}=:K_\theta<\infty.
\end{equation}
Then \eqref{e:V_theta_small_x} follows by combining \eqref{e3.11} and \eqref{e3.12} as follows. For $\bx$ satisfying $\|\bx\|\leq M$ and $T\in  [T^*,n^*T^*]$ we have

\begin{equation*}
\begin{split}
\E_{\bx,k} V^\theta(\BX(T), r(T)) &=1_{\dist(\bx,\partial\R^n_+)\geq\tilde\delta}\E_{\bx,k} V^\theta(\BX(T), r(T)) + 1_{\dist(\bx,\partial\R^n_+)\leq\tilde\delta}\E_{\bx,k} V^\theta(\BX(T), r(T)) \\
&\leq \exp(-0.5p^*T\theta)V^\theta(\bx,k) + K_\theta.
\end{split}
\end{equation*}
\end{proof}

\begin{rmk}
One might be tempted to just use Dynkin's formula in trying to prove this proposition. However, this will not yield the correct result. Assume we use Dynkin's formula to get
\[
 \frac{1}{T}\E_{\bx,k} \ln V^\theta (\BX(T),r(T)) =  \frac{1}{T} \ln V^\theta (\bx,k) +  \frac{1}{T}\theta\E_{\bx,k} \int_0^T \Phi(\BX(t), r(t))dt 
\]
Then for $\|\bx\|\leq M, 0<\dist(\bx,\partial\R^n_+)<\tilde\delta, T\in [T^*,n^*T^*]$ we get from Lemma \ref{lm3.1} and the Feller property that
\begin{equation}
\begin{split}
 \frac{1}{T}\E_{\bx,k}\ln V^\theta (\BX(T),r(T)) \leq \frac{1}{T} \ln V^\theta (\bx,k) - \frac{1}{2}\theta \rho^*
\end{split}
\end{equation}
Now the next natural step would be to use Jensen's inequality but this goes in the wrong direction as
\[
\E_{\bx,k}\ln V^\theta (\BX(T),r(T)) \leq \ln \E_{\bx,k} V^\theta (\BX(T),r(T)).
\]
It is why the significantly more technical proof presented above is required. 
\end{rmk}

The next result is the main persistence theorem.
\begin{thm}\label{thm3.1}
Suppose that Assumptions \ref{a1-pdm}, \ref{a:bound} and \ref{a2-pdm} hold.
Let $\theta$ be as in Proposition \ref{prop2.1}, $n^*$ as in \eqref{e:n*}.
There are $\kappa=\kappa(\theta,T^*)\in(0,1)$, $\tilde K=\tilde K(\theta,T^*)>0$   such that
\begin{equation}\label{e:lya}
\E_{\bx,k} V^\theta(\BX(n^*T^*), r(n^*T^*))\leq \kappa V^\theta(\bx,k)+\tilde K\,\text{ for all }\, (\bx,k)\in\R^{n,\circ}_+\times\CN.
\end{equation}
As a result,
$(\BX(t),r(t))$ is stochastically persistent in probability and almost surely stochastically persistent. 

In addition, if Assumption \ref{a3-pdm} is satisfied, then there is a unique probability measure $\pi^*$ on $\R^{n,\circ}_+\times\CN$ and the convergence of the transition probability of $(\BX(t),r(t))$ in total variation to $\pi^*$ on $\R^{n,\circ}_+$ is
exponentially fast. For any initial value $\mathbf{x}\in\R^{n,\circ}_+$ and any $\pi^*$-integrable function $f$ we have
\begin{equation}\label{slln}
\PP_{\bx,k}\left\{\lim\limits_{T\to\infty}\dfrac1T\int_0^Tf\left(\BX(t),r(t)\right)dt=\sum_{k\in\CN}\int_{\R_+^{n,\circ}}f(\mathbf{u},k)\pi^*(d\mathbf{u},k)\right\}=1.
\end{equation}

\end{thm}
\begin{proof}
By direct calculation, we get

\begin{equation}
\Lom V^{\theta}(\bx,k)=\theta V^{\theta}(\bx,k)\left(\sum_{i=1}^n \frac{\bx_if_i(\bx,k)}{F(\bx,k)}\frac{\partial F}{\partial x_i}(\bx,k)-\sum_{i=1}^n p_if_i(\bx,k)    \right)+ \sum_{l\in\CN}q_{kl}(\bx)V^{\theta}(\bx,l)    
\end{equation}
which implies
\begin{equation}
\begin{split}
\Lom V^{\theta}(\bx,k)&=\theta V^{\theta}(\bx,k)\left(\sum_{i=1}^n \frac{\bx_if_i(\bx,k)}{F(\bx,k)}\frac{\partial F}{\partial x_i}(\bx,k)-\sum_{i=1}^n p_if_i(\bx,k)    \right)+ \sum_{l\in\CN}q_{kl}(\bx)V^{\theta}(\bx,l)\\
&= \theta V^{\theta}(\bx,k)\left(\frac{\Lom F(\bx,k)}{F(\bx,k)}- \sum_{i=1}^n p_if_i(\bx,k) \right) - \sum_{l\in\CN}q_{kl}(\bx) \theta V^{\theta}(\bx,k)\frac{F(\bx,l)}{F(\bx,k)} + \sum_{l\in\CN}q_{kl}(\bx)V^{\theta}(\bx,l)
\end{split}
\end{equation}
This can be rewritten as
\begin{equation}\label{gentheta}
\begin{split}
\Lom V^{\theta}(\bx,k)=&\theta V^{\theta}(\bx,k)\left(\frac{\Lom F(\bx,k)}{F(\bx,k)}- \sum_{i=1}^n p_if_i(\bx,k) \right) - \theta V^{\theta}(\bx,k)\sum_{l\neq k}q_{kl}(\bx)\left(\frac{F(\bx,l)}{F(\bx,k)}-1\right)\\
+&V^{\theta}(\bx,k)\left(\sum_{l\neq k}q_{kl}(\bx)\left(\left(\frac{F(\bx,l)}{F(\bx,k)}\right)^{\theta}-1\right)  \right)  \\
=&\theta V^{\theta}(\bx,k)\left(\frac{\Lom F(\bx,k)}{F(\bx,k)}- \sum_{i=1}^n p_if_i(\bx,k) \right) \\
+& V^{\theta}(\bx,k)\sum_{l\neq k}q_{kl}(\bx)\left(- \theta\left(\frac{F(\bx,l)}{F(\bx,k)}-1\right) +\left(\left(\frac{F(\bx,l)}{F(\bx,k)}\right)^{\theta}-1\right) \right)
\end{split}    
\end{equation}
Note that for any  $y\geq 0$ and $\theta\in(0,1)$ the following inequality holds
\begin{equation}\label{e:ineq_theta}
g(y)=-\theta (y-1) + y^\theta-1 \leq 0
\end{equation}
This holds since $g(0)=\theta-1<0, g(1)=0$, $g'(y)\geq 0$ for $y\leq 1$ and $g'(y)\leq 0$ for $y\geq 1$.
This implies if we set $y=\frac{F(\bx,l)}{F(\bx,k)}$ and note that $q_{kl}(\bx)\geq 0$ for $k\neq l$ that
\begin{equation}\label{theta_trick}
V^{\theta}(\bx,k)\sum_{l\neq k}q_{kl}(\bx)\left(- \theta\left(\frac{F(\bx,l)}{F(\bx,k)}-1\right) +\left(\left(\frac{F(\bx,l)}{F(\bx,k)}\right)^{\theta}-1\right) \right)\leq 0.    
\end{equation}
Combining this with \eqref{gentheta} and using \eqref{e:sup}, we have
\begin{equation}\label{et1.1}
\Lom V^\theta(\bx,k)\leq -\theta\gamma_bV^\theta(\bx,k)  \text{ if } \|x\|> M .
\end{equation}
Define
\begin{equation}\label{e:tau}
\tau=\inf\{t\geq0: \|\BX(t)\|\leq M\}.
\end{equation}
In view of \eqref{et1.1}, we can obtain from Dynkin's formula that
$$
\begin{aligned}
\E_{\bx,k}&\left[ \exp\left(\theta\gamma_b(\tau\wedge n^*T^*)\right)V^\theta(\BX(\tau\wedge n^*T^*),r(\tau\wedge n^*T^*))\right]\\
&\leq V^\theta(\bx,k) +\E_{\bx,k} \int_0^{\tau\wedge n^*T^*}\exp(\theta\gamma_b s)[\Lom V^\theta(\BX(s),r(s))+ \theta\gamma_bV^\theta(\BX(s),r(s))]ds\\
&\leq V^\theta(\bx,k).
\end{aligned}
$$
Thus,
\begin{equation}\label{et1.2}
\begin{aligned}
V^\theta(\bx,k)\geq&
\E_{\bx,k}\left[ \exp\left(\theta\gamma_b(\tau\wedge n^*T^*)\right)V^\theta(\BX(\tau\wedge n^*T^*),r(\tau\wedge n^*T^*))\right]\\
=&
 \E_{\bx,k} \left[\1_{\{\tau\leq (n^*-1)T^*\}}\exp\left(\theta\gamma_b(\tau\wedge n^*T^*)\right)V^\theta(\BX(\tau\wedge n^*T^*),r(\tau\wedge n^*T^*))\right]\\
 &+\E_{\bx,k} \left[\1_{\{ (n^*-1)T^*<\tau<n^*T^*\}}\exp\left(\theta\gamma_b(\tau\wedge n^*T^*)\right)V^\theta(\BX(\tau\wedge n^*T^*),r(\tau\wedge n^*T^*))\right]\\
&+ \E_{\bx,k} \left[\1_{\{\tau\geq n^*T^*\}}\exp\left(\theta\gamma_b(\tau\wedge n^*T^*)\right)V^\theta(\BX(\tau\wedge n^*T^*),r(\tau\wedge n^*T^*))\right]\\
\geq&
 \E_{\bx,k} \left[\1_{\{\tau\leq (n^*-1)T^*\}}V^\theta(\BX(\tau),r(\tau))\right]\\
 &+\exp\left(\theta\gamma_b (n^*-1)T^*\right)\E_{\bx,k} \left[\1_{\{ (n^*-1)T^*<\tau<n^*T^*\}}V^\theta(\BX(\tau),r(\tau))\right]\\
&+\exp\left(\theta\gamma_b n^*T^*\right) \E_{\bx,k} \left[\1_{\{\tau\geq n^*T^*\}}V^\theta(\BX(n^*T^*), r(n^*T^*))\right].\\
 \end{aligned}
\end{equation}
By the strong Markov property of $(\BX(t),r(t))$ and
Proposition \ref{prop2.1}, we obtain
\begin{equation}\label{et1.3}
\begin{aligned}
\E_{\bx,k}&\left[ \1_{\{\tau\leq (n^*-1)T^*\}}V^\theta(\BX(n^*T^*), r(n^*T^*))\right]\\
&\leq
 \E_{\bx,k} \left[\1_{\{\tau\leq (n^*-1)T^*\}}\left(K_{\theta}+ \exp\left(-0.5\theta p^*(n^*T^*-\tau)\right)\left(V^\theta(\BX(\tau),r(\tau))\right)\right)\right]\\
 &\leq \exp(-0.5\theta \rho^*T^*)\E_{\bx,k}\left[\1_{\{\tau\leq (n^*-1)T^*\}}\left(V^\theta(\BX(\tau),r(\tau))\right)\right] 
 +K_\theta\PP_{\bx,k}(\tau\leq (n^*-1)T^*)
 \end{aligned}
\end{equation}
By the strong Markov property of $(\BX(t),r(t))$ and
\eqref{e1-lm2.2}, we obtain
\begin{equation}\label{et1.4}
\begin{aligned}
\E_{\bx,k}&\left[ \1_{\{(n^*-1)T^*<\tau<n^*T^*\}}V^\theta(\BX(n^*T^*), r(n^*T^*))\right]\\
&\leq
 \E_{\bx,k} \left[\1_{\{(n^*-1)T^*<\tau<n^*T^*\}}\exp\left(\theta H(n^*T^*-\tau)\right)\left(V^\theta(\BX(\tau),r(\tau))\right)\right]\\
 &\leq \exp(\theta HT^*)\E_{\bx,k}\left[\1_{\{(n^*-1)T^*<\tau<n^*T^*\}}\left(V^\theta(\BX(\tau),r(\tau))\right)\right].
 \end{aligned}
\end{equation}
Applying \eqref{et1.3} and \eqref{et1.4} to \eqref{et1.2} yields
\begin{equation}\label{et1.5}
\begin{aligned}
V^\theta(x,k)
\geq&
 \E_{\bx,k} \left[\1_{\{\tau\leq (n^*-1)T^*\}}V^\theta(\BX(\tau),r(\tau))\right]\\
 &+\exp\left(\theta\gamma_b (n^*-1)T^*\right)\E_{\bx,k} \left[\1_{\{ (n^*-1)T^*<\tau<n^*T^*\}}V^\theta(\BX(\tau),r(\tau))\right]\\
&+\exp\left(\theta\gamma_b n^*T^*\right) \E_{\bx,k} \left[\1_{\{\tau\geq n^*T^*\}}V^\theta(\BX(n^*T^*), r(n^*T^*))\right]\\
\geq& \exp(0.5\theta \rho^*T^*)\E_{\bx,k}\left[\1_{\{\tau\leq (n^*-1)T^*\}}V^\theta(\BX(n^*T^*), r(n^*T^*))\right]-\exp(0.5\theta\rho^*T^*)K_\theta\\
 &+\exp(-\theta HT^*)\exp\left(\theta\gamma_b (n^*-1)T^*\right)\E_{\bx,k} \left[\1_{\{ (n^*-1)T^*<\tau<n^*T^*\}}V^\theta(\BX(n^*T^*), r(n^*T^*))\right]\\
&+\exp\left(\theta\gamma_b n^*T^*\right) \E_{\bx,k} \left[\1_{\{\tau\geq n^*T^*\}}V^\theta(\BX(n^*T^*), r(n^*T^*))\right]\\
\geq & \exp( m\theta T^*)\E_{\bx,k} V^\theta(\BX(n^*T^*), r(n^*T^*))-K_\theta\exp(0.5\theta\rho^*T^*)
 \end{aligned}
\end{equation}
where $m=\min\{0.5 \rho^*, \gamma_b n^*, \gamma_b (n^*-1)- H\}>0$ by \eqref{e:n*}.
The proof of \eqref{e:lya} is complete by taking $\kappa=\exp(-m\theta T^*)$ and $\tilde K=\exp(-m\theta T^*)K_\theta\exp(0.5\theta\rho^*T^*)$.

\begin{clm}\label{stoch_persist}
The process $(\BX(t),r(t))$ is persistent in probability.    
\end{clm}
\begin{proof}
    
By using \eqref{e:lya} together with the strong Markov property recursively one can see that if we let $\bar T = n^*T^*$
\[
\E_{\bx,k} V^\theta(\BX(2\bar T), r(2\bar T))\leq \kappa^2 V^\theta(\bx,k)+\tilde K(1+\kappa)
\]
and more generally that
\[
\E_{\bx,k} V^\theta(\BX(n\bar T), r(n\bar T))\leq \kappa^n V^\theta(\bx,k)+\tilde K(1+\kappa+\kappa^2+\dots+\kappa^{n-1})\leq \kappa^n V^\theta(\bx,k) + \frac{1}{1-\kappa}.
\]
Since we have by \eqref{e3.4_2}
\[
\sup_{s\in [0,\bar T]}\dfrac{\E_{\bx,k} V^\theta (\BX(s), r(s))}{V^\theta (\bx,k)}\leq  \exp( \theta H\bar T)
\]
we get that
\[
\limsup_{t\to\infty} \E_{\bx,k} V^\theta(\BX(t), r(t))\leq \frac{\exp( \theta H\bar T)}{1-\kappa}.
\]
This together with the fact that
\[
\lim_{\min_ix_i \to 0} V^\theta(\bx,k) = \infty
\]
implies that for any $\eps>0$ there exists $\delta>0$ such that for all $(\bx,k)\in\R^{n,\circ}_+\times\CN$
\[
\liminf_{t\to\infty} \PP_{\bx,k} \{\min_i X_i(t)\geq \delta\} \geq 1-\eps,
\]
which shows persistence in probability. 
\end{proof}

\begin{clm}
The process $(\BX(t),r(t))$ exhibits almost sure stochastic persistence. 
\end{clm}
\begin{proof}

By Lemma \ref{lm4.6} below for any initial condition $(\BX(0),r(0))=(\bx,k)\in\R^n_+\times \CN$,
the family $\left\{\wtd \Pi_t(\cdot), t\geq 1\right\}$ is tight in $\R^n_+\times \CN$,
and its weak$^*$-limit set, denoted by $\U=\U(\omega)$
is a family of invariant probability measures of $(\BX(t),r(t))$ with probability 1. If $\BX(0)\in \R_+^{n,\circ}$ we can see by the proofs of Lemmas \ref{lm4.4} and \ref{lm4.9} that if we take $t\to \infty$ in
\[
\sum_i p_i \frac{\ln X_i(t)}{t} = \sum_i p_i \frac{\ln x_i}{t} + \sum_i p_i \frac{1}{t} \int_0^t f_i(\BX(s),r(s))\,ds
\]
yields that for any limit point $\mu$ of $\wtd \Pi_t(\cdot)$ we have
\begin{equation}\label{e:sum1}
\sum_i p_i \lambda_i(\mu)\leq 0.
\end{equation}

Since $\R_+^n\times \CN$ and $\R_+^{n,\circ}\times \CN$ are both invariant we can see that there are invariant measures $\mu_1,\mu_2$ such that $\mu=\gamma \mu_1 + (1-\gamma)\mu_2$ for $\mu_1(\R_+^{n,\circ}\times \CN)=0$ and $\mu_2(\R_+^{n,\circ}\times \CN)=1$.
Then by Lemma \ref{lm2.3}
\begin{equation}\label{e:sum2}
\sum_i p_i \lambda_i(\mu_2)=0
\end{equation}
while by Assumption \ref{a2-pdm}
\begin{equation}\label{e:sum3}
\sum_i p_i \lambda_i(\mu_1)>0.
\end{equation}

Equations \eqref{e:sum1}, \eqref{e:sum2} and \eqref{e:sum3} imply that
\[
\gamma=0.
\]
This shows that if the process starts on $\R_+^{n,\circ}\times \CN$ then any weak$^*$-limit point of $\left\{\wtd \Pi_t(\cdot), t\geq 1\right\}$ is an invariant probability measure supported on $\R_+^{n,\circ}\times \CN$ i.e. the measure puts no mass on the extinction set. We next follow \cite{hairer2006ergodic, benaim2019persistence} to finish the proof. Let $M>0$. 
\begin{equation}
\begin{split}
\sum\int (V^\theta \wedge M)(\bx ,k)  \mu (d\bx ,k)  &= \sum\int P_{\bar T}(V^\theta \wedge M)  \mu (d\bx ,k) \\
&= \sum\int \E_{\bz,k}(V^\theta(\BX(\bar T), r(\bar T)) \wedge M)  \mu (d\bx ,k)\\
&\leq  \sum\int ( \kappa V^\theta(\bx,k)+\tilde K) \wedge M  \mu (d\bx ,k)
\end{split}
\end{equation}
where we first used the fact that $\mu$ is an invariant probability measure, followed by Jensen's inequality and then \eqref{e:lya}.
If we iterate the above $m$ times we get
\[
\sum\int (V^\theta \wedge M)(\bx ,k)  \mu (d\bx ,k) \leq \sum\int \left( \kappa^m V^\theta(\bx,k)+\frac{\tilde K}{1-\kappa}\right) \wedge M  \mu (d\bx ,k).
\]
Letting $m\to \infty$ and using the dominated convergence theorem yields
\[
\sum\int (V^\theta \wedge M)(\bx ,k)  \mu (d\bx ,k) \leq \frac{\tilde K}{1-\kappa}.
\]
Next, we let $M\to\infty$ and use the dominated convergence theorem again
\[
\E_\mu V^\theta (\BX(t) ,r(t)) = \sum\int V^\theta (\bx ,k)  \mu (d\bx ,k) \leq \frac{\tilde K}{1-\kappa}.
\]
Note that for any $\eta>0$ we can bound $V^\theta$ below on $S_\mu\cap \R_+^{n,\circ}$ where $S_\mu:=\{(\bx,k)\in \R_+^{n,\circ}\times \CN, \min_i x_i \leq \eta\} $ as follows
\[
\inf_{\min_i x_i \leq \eta} V^\theta (\bx,k) = \inf_{\min_i x_i \leq \eta} \dfrac{F^\theta (\bx,k)}{\prod_{i=1}^n x_i^{\theta p_i}}\geq c^\theta \eta^{-\theta (n-1) \max_i p_i}>0
\]
Thus
\[
\mu(S_\mu)\inf_{\bx\in S_\mu} V^\theta (\bx,k) \leq \sum \int V^\theta (\bx ,k)  \mu (d\bx ,k) \leq \frac{\tilde K}{1-\kappa}
\]
which finally implies
\[
\mu(S_\mu) \leq \frac{\tilde K}{1-\kappa}  c^{-\theta} \eta^{\theta (n-1) \max_i p_i}.
\]
This finishes the proof of the claim since
\[
\limsup_{t\to \infty} \wtd \Pi_t(S_\mu) \leq \mu(S_\mu) \leq \frac{\tilde K}{1-\kappa}  c^{-\theta} \eta^{\theta (n-1) \max_i p_i}.
\]
\end{proof}

\begin{clm}
Under Assumption \eqref{a3-pdm} for any $\BX(0)\in\R^{n,\circ}_+$ the transition function of the process $(\BX(t),r(t))$ converges in total variation, exponentially fast, as $t\to\infty$ to its unique invariant probability measure $\pi^*$.
\end{clm}
\begin{proof}

Under Assumption \eqref{a3-pdm},
it is shown implicitly in \cite[Theorem 4.4]{BBMZ15} that
the Markov chain
$\{\BX(\ell \bar T), r(\ell \bar T)):\ell\in\N\}$
is irreducible and aperiodic. Moreover, every compact subset $K$ of $\R^{n,\circ}_+\times\CN$
is petite.

Applying the second corollary of \cite[Theorem 6.2]{MT},
we deduce from \eqref{et1.5} that
\begin{equation}\label{e.cxr}
\|P(\ell \bar T,\bx,k,\cdot)-\pi^*(\cdot)\|_{TV}\leq C_{\bx,k} u^\ell
\end{equation}
  where $\pi^*$ is an invariant probability measure of $((\BX(\ell \bar T),r(\ell \bar T)), \ell\in\N)$ on $\R^{n,\circ}_+\times\CN$, for some $u\in(0,1)$ and $C_{\bx,k}>0$ a constant depending on $(\bx, k) \in\R^{n,\circ}_+\times\CN$.

On the other hand, it follows from \eqref{et1.5} and \cite[Theorem 6.2]{MT},
that for any compact set $K\subset\R^{n,\circ}_+\times\CN$,
we have $\E_\bx\tau_K^*<\infty$ where $\tau_K^*$ is the first time the Markov chain $((\BX(\ell \bar T),r(\ell \bar T)), \ell\in\N)$  enters $K$.
Thus, the process $(\BX(t),r(t))$ is positive recurrent,
or equivalently, $(\BX(t),r(t))$  has a unique invariant probability measure on $\R^{n,\circ}_+$
(see e.g. \cite[Chapter 4]{RK}).
Because of \eqref{e.cxr}, the unique invariant probability measure of the process $(\BX(t),r(t))$
must be $\pi^*$.
Moreover, it is well-known that $\|P(t,\bx,\cdot)-\pi^*(\cdot)\|_{TV}$ is decreasing in $t$
(it can be shown easily using the Kolmogorov-Chapman equation).
We therefore obtain an exponential upper bound for $\|P(t,\bx,\cdot)-\pi^*(\cdot)\|_{TV}$.

\end{proof}
This finishes the proof of the theorem.

\end{proof}
\section{Extinction}\label{s:extin}
This section is devoted to the study of conditions under which some of the species will go extinct with  positive probability.

We first give the intuition behind the proofs of some of the main results.

\begin{proof}[Intuition behind the proofs of Theorems \ref{thm4.1} and \ref{t:exxx}]
According to Assumption \ref{a4-pdm} we can show that $\R_+^I$ attracts the process, in the sense that every ergodic measure $\mu\in \M^{I,\circ}$ is an attractor, while the invariant probability measures on the boundary $\partial \R_+^I\times \CN$ are repellers. This last condition ensures that the process does not spend a lot of time close to $\partial \R_+^I$ where it could be pushed away from $\R_+^I$. 

It is key to note that our assumptions allows us to show that one can find constants $\hat p_i, \check p, \rho_e>0$ such that for all $\nu\in\M^I$,
\[
\sum_{i\in I}\hat p_i\lambda_i(\nu)-\check p\max_{i\in I^c}\left\{\lambda_i(\nu)\right\}>3\rho_e.
\]

This allows us to construct the Lyapunov function
$$U_\theta(\bx,k)=\sum_{i\in I^c}\left[(F(\bx,k))\dfrac{x_i^{\check p}}{\prod_{j\in I} x_j^{\hat p_j}}\right]^\theta, \bx\in\R^{n,\circ}_+,$$
which vanishes on $\R_+^{I,\circ}$. In Proposition \ref{prop4.1} we show that 
\[
\E_{\bx,k}U_\theta(\BX(T),r(T)) \leq U_\theta (\bx,k)
\]
if $\bx\in \R_+^{n,\circ}$ is in a compact set $\|\bx\|\leq M$ and close enough to $\R_+^I$ and $T$ is large. One can use this together with other estimates which ensure that $U_\theta$ does not grow too fast (see Lemma \ref{lm3.3}) to show that for $$\tilde U_\theta(\bx,k) = U_\theta(\bx,k) \wedge \varsigma$$
we have that for all $\bx \in \R_+^{n,\circ}$ close enough to $\R_+^I$ and for $T>0$ large enough
\[
\E_{\bx,k}\tilde U_\theta(\BX(T),r(T)) \leq \tilde U_\theta (\bx,k).
\]

This allows us to prove using the Markov property that the sequence
$\{Y(\ell): \ell\in\N\}$ where $Y(\ell):=\tilde U_\theta(\BX(\ell T))$ is a supermartingale. This then in turn will imply that for $\bx\in \R_+^{n,\circ}$ close enough to $\partial \R_+^i$
\[
\PP_{\bx,k}\left\{\lim_{t\to\infty}X_i(t)=0, i\in I^c\right\}\geq 1-\varepsilon.
\]
and
\begin{equation}
 \PP_{\bx,k}\left( \limsup_{t\to\infty}\frac{\ln(\BX_i(t))}{t} \leq -\frac{\alpha_I}{\theta \check p_i}  \right)\geq 1-\eps.   
\end{equation}
In particular one can also show there is no invariant measure on $\R_+^{n,\circ}\times \CN$.
The proofs of Theorems \ref{thm4.1} and \ref{t:exxx} can then be established easily.
\end{proof}

\begin{proof}[Intuition behind the proofs of Theorems \ref{thm4.2} and \ref{t:ex_one}]
Let $S$ be the family of subsets $I$ that satisfy Assumption \ref{a4-pdm}. The subspaces $\R_+^I$ have in their interiors only ergodic measures which are attractors. Then $S^c$ will be the subsets of $\{1,\dots,n\}$ which do not satisfy Assumption \ref{a4-pdm}. The key insight here is that we require in Assumption \ref{a5-pdm} that
$$\max_{i=1,\dots,n}\left\{\lambda_i(\nu)\right\}>0\text{ for any }\nu\in\Conv\left(\displaystyle\bigcup_{J\notin S}\M^{J,\circ}\right).$$

In terms of ergodic probability measures, we let $\M_1$ be the attractors and $\M_2=\M\setminus \M_1$ be the repellers.

The above condition ensures that the process, that is, the randomized occupation measures, cannot converge to an invariant probability measure that is not among the ergodic `attractors' from $\M_1$. Lemmas \ref{lm4.4} and \ref{lm4.6} show that the occupation measures are tight and have as weak* limits invariant probability measures of the process. Lemma \ref{lm4.5} shows that we can take limits of integrals with respect to converging sequences of occupation measures.

In Lemma \ref{lm4.7} we prove that for every $I$ satisfying Assumption \ref{a4-pdm} there is a compact set $\K_I^0 \subset \R_+^{I,\circ}$ such that for all invariant probability measures from $\R_+^{I,\circ}\times\CN$ one has
\[
\mu(\K_I^0\times \CN)>0.
\]
In Lemma \ref{lm4.9} we show that if a sequence of occupation measures converges to some invariant probability measure on $\R_+^I$, then because of Assumption \ref{a5-pdm} that invariant probability measure has to live on $\R_+^{I,\circ}$.

Define
$$\K^{\ell,\Delta}_I:=\{\bx\in\R^{n,\circ}_+, \ell^{-1}\leq x_i\leq \ell\text{ for } i\in I, x_i<\Delta\text{ for }i\in I^c\}.$$

One can find $\ell>\ell_0$ and $\Delta>0$ such that
\begin{equation}\label{e3-thm4.2}
\PP_{\bx,k}\left\{\lim_{t\to\infty} X_i(t)=0, i\in I^c\right\}>1-\eps
\end{equation}
for all $I\in \mathcal{S}$ and
$\bx\in\K_I^{\ell,\Delta}.$ Since $\bigcup_{I\in S} S^I_+$ is accessible we get by Theorem \ref{thm4.1} also that no invariant probability measure exists on $\R_+^{n,\circ}\times\CN$ so that all limit points of the families of occupation measures are in $\Conv(\M_1)\setminus\Conv(\M_2)$.
One can then show that since  $\pi_1(\bigcup_{I\in\mathcal{S}}\K_I^0)>0$ for any $\pi_1\in\Conv(\cup_{I\in \mathcal{S}}\M^{I,\circ})$ we get
\[
\PP_{\bx,k}\left\{\liminf_{t\to\infty}\dfrac1t\int_0^t\1_{\left\{\BX(s)\in \bigcup_{I\in\mathcal{S}}\K_I^{\ell,\Delta}\right\}}ds>0\right\}=1,\,\,\bx\in\R^{n,\circ}_+, k\in \CN.
\]
This implies that if $\BX(0)\in\R^{n,\circ}_+$ then $\BX(t)$ almost surely will enter the set $\bigcup_{I\in\mathcal{S}}\K_I^{\ell,\Delta}$. This means that the process gets close enough to the `attractors' so that extinction will happen with high probability according to \eqref{e3-thm4.2}.

It is not hard to see that this then implies
$$\sum_{I\in\mathcal{S}} P_{\bx,k}^I>1-\eps,\,\bx\in\R^{n,\circ}_+$$
where
$$P_{\bx,k}^I=\PP_{\bx,k}\left\{\U(\omega)\subset\Conv(\M^{I,\circ})\text{ and }\lim_{t\to\infty}\dfrac{\ln X_i(t)}t\in\left\{\lambda_i(\mu),\mu\in\Conv(\M^{I,\circ})\right\}, i\in I_\mu^c\right\}.$$
Letting $\eps\downarrow 0$ will yield the desired result.

\end{proof}

\begin{lm}\label{lm4.1}
For any $\mu\in\M$ and $i\in I_\mu$ we have
$\lambda_i(\mu)=0.$
\end{lm}
\begin{proof}
In view of It\^o's formula,
$$
\dfrac{\ln X_i(t)}t=\dfrac{\ln X_i(0)}t+\dfrac1t\int_0^tf_i(\BX(s),r(s))ds
$$
In the same manner as in the second part of the proof of Lemma \ref{lm2.3},
we can show that if $\BX(0)=\bx_0\in\R_+^{\mu,\circ}$ and $i\in I_\mu$,
then
$$
\lim_{t\to\infty}\dfrac1t\int_0^tf_i(\BX(s),r(s))ds=\lambda_i(\mu)~~~\PP_{\bx_0,k}\text{-  a.s.}
$$

On the other hand, $X_i(t), i\in I_\mu$ can go to neither $0$ nor $\infty$ as $t\to\infty$. Thus
\[
\lim_{t\to\infty}\dfrac{\ln X_i(t)}t = 0, ~~~\PP_{\bx_0,k}\text{-  a.s.}, i \in I_\mu
\]
which implies the desired result.
\end{proof}

Condition \eqref{ae3.2} is equivalent to the existence of
$0<\hat p_i<\delta_0, i\in I$
such that for any $\nu\in\partial \M^I$ , we have
$$\sum_{i\in I}\hat p_i\lambda_i(\nu)>0.$$
By rescaling we can also assume that $\sum_{i\in I}\hat p_i\leq \delta_0$.

Thus, there is $\check p\in (0,\delta_0)$ sufficiently small such that
\begin{equation}\label{e3.2}
\begin{aligned}
\sum_{i\in I}\hat p_i\lambda_i(\nu)-\check p\max_{i\in I^c}\{\lambda_i(\nu)\}>0 \text{ for any }\nu\in\partial \M^{I}
\end{aligned}
\end{equation}
In view of \eqref{e3.2}, \eqref{ae3.1} and Lemma \ref{lm4.1}, there is $\rho_e>0$ such that for any $\nu\in\M^I$,
\begin{equation}\label{e3.3}
\sum_{i\in I}\hat p_i\lambda_i(\nu)-\check p\max_{i\in I^c}\left\{\lambda_i(\nu)\right\}>3\rho_e.
\end{equation}

Define 
\begin{equation}\label{e:bar_Phi}
\bar \Phi(\BX(s),r(s)):=\Lom(\ln(F(\BX(s),r(s)))).
\end{equation}
\begin{lm}\label{lm4.2}
Suppose that Assumption \ref{a4-pdm} holds.
Let $M$ be as in \eqref{a1-pdm}, $H$ as in \eqref{e:H} and $\hat p_i, \check p, \rho_e$ as in \eqref{e3.3}. Let $n_e\in\N$ such that $\gamma_b(n_e-1)>H+1$.
There are $T_e\geq 0$, $\delta_e>0$ such that, for any $T\in[T_e,n_eT_e]$, $\|\bx\|\leq M, x_i<\delta_e, i\in I^c$, we have
\begin{equation}\label{e3.4}
\begin{aligned}
\dfrac1T&\int_0^T\E_{\bx,k}\left(\bar \Phi(\BX(t),r(t))\right)dt-\sum_{i\in I}\hat p_i\dfrac1T\int_0^T\E_{\bx,k} f_i(\BX(t),r(t))dt\\
&+\check p\max_{i\in I^c}\left\{\dfrac1T\int_0^T\E_{\bx,k}f_i(\BX(t),r(t))dt\right\}
\leq -\rho_e.
\end{aligned}
\end{equation}
\end{lm}

\begin{proof}
Analogously to Lemma \ref{lm3.1}, using \eqref{e3.3}, one can show there exists $T_e>0$ such that for any $T>T_e, \bx\in\R_+^\mu, \|\bx\|\leq M$, we have
\begin{equation}
\begin{aligned}
\dfrac1T&\int_0^T\E_{\bx,k}\left( \bar\Phi(\BX(t),r(t))\right)dt-\sum_{i\in I}\hat p_i\dfrac1T\int_0^T\E_{\bx,k} f_i(\BX(t),r(t))dt\\
&+\check p\max_{i\in I^c}\left\{\dfrac1T\int_0^T\E_{\bx,k}f_i(\BX(t),r(t))dt\right\}
\leq  -2\rho_e
\end{aligned}
\end{equation}
By the Feller property of $(\BX(t),r(t))$ and compactness of the set $\{\bx\in\R_+^\mu, \|\bx\|\leq M\}$, there is $\delta_e>0$ such that
for any $T\in[T_e,n_eT_e]$, $\|\bx\|\leq M, x_i<\delta_e, i\in I^c$ , the estimate
\eqref{e3.4} holds.
\end{proof}
\begin{prop}\label{prop4.1}
Suppose that Assumption \ref{a4-pdm} holds.
Let $\delta_0>0$ be as in \ref{a1-pdm}.
There is $\theta\in(0,\delta_0)$ such that for any $T\in[T_e,n_eT_e]$ and $\bx\in\R^{n,\circ}_+$ satisfying $ \|\bx\|\leq M,$ $x_i<\delta_e,$ $i\in I^c$ one has
$$\E_{\bx,k} U_\theta(\BX(T), r(T))\leq \exp(-0.5\theta \rho_eT) U_\theta(\bx,k)$$
where
$M, T_e, \hat p_i,\check p, \delta_e, n_e$ are as in Lemma \ref{lm4.2} and
$$U_\theta(\bx,k)=\sum_{i\in I^c}\left[(F(\bx,k))\dfrac{x_i^{\check p}}{\prod_{j\in I} x_j^{\hat p_j}}\right]^\theta, \bx\in\R^{n,\circ}_+.$$
\end{prop}
\begin{proof}
For $i\in I^c$, let
$U_i(\bx,k)=(F(\bx,k))\dfrac{x_i^{\check p}}{\prod_{j\in I} x_j^{\hat p_j}}.$
Using Ito's formula, we get that 
\begin{equation}
\ln(U_i(\BX(T),r(T)))= \ln(U_i(\bx,k))+\int_0^T \Lom(\ln(U_i(\BX(s),r(s)))) ds + M_U(T)    
\end{equation}
Where $M_U(t)$ is the martingale 
\begin{equation}
M_U(T)=\int_0^T\int_{\R_+}\ln\left[\frac{U_i(\BX(s),r(s)+h(\BX(s),r(s),y))}{U_i(\BX(s),r(s))}\right] \tilde{N}(dy, ds).    
\end{equation}

Using $\bar \Phi(\BX(s),r(s))=\Lom(\ln(F(\BX(s),r(s))))$ instead of $\Phi$, by Lemma \ref{lm4.2} and using arguments basically identical to those used in the proof of Proposition \ref{prop2.1} we have that one can find $\theta>0$ such that for $T\in[T_e,n_eT_e]$, $\bx\in\R^{n,\circ}_+$ with $\|\bx\|\leq M,$ and $x_i<\delta_e$ we have 
\begin{equation}
\begin{split}
\E_{\bx,k} U_i^\theta(\BX(T), r(T))\leq \exp(-0.5\theta \rho_eT) U_\theta(\bx,k)
\end{split}    
\end{equation}

The proof is complete by noting that
$$U_\theta(\bx,k)=\sum_{i\in I^c}U_i^\theta(\bx,k).$$
\end{proof}

\begin{lm}\label{lm3.3}
Let $H$ be defined by \eqref{e:H}. For $\theta\in[0,\frac12]$ we have
$$\E_{\bx,k} U_\theta(\BX(t),r(t))\leq \exp(\theta Ht) U_\theta(\bx,k)\,,\, (\bx,k)\in\R^{n,\circ}_+\times\CN.$$
\begin{equation}\label{u3}
\E_{\bx,k}\sup_{t\in[0,T]} U_\theta(\BX(t),r(t))\leq  \hat H_{t,\theta} U_\theta(\bx,k)\,,\, (\bx,k)\in\R^{n,\circ}_+\times\CN.
\end{equation}
\end{lm}
\begin{proof}
By the arguments from the proof of \eqref{e1-lm2.2}, for $\theta<1$, $i\in I^c$
we have
\begin{equation}\label{u1}
\E_{\bx,k} U_i^{\theta}(\BX(t),r(t))\leq \exp(\theta Ht) U_i^{\theta}(\bx,k)\,, \bx\in\R^{n,\circ}_+.
\end{equation}
From this estimate, we can take the sum over $I^c$ to obtain the first desired result.

Moreover,
$$
\begin{aligned}
d&U_i^{\theta}(\BX(t),r(t))=\Lom U_i^{\theta}(\BX(t),r(t))dt\\
&+\int_R \left(U_i^{\theta}(\BX(t),r(t)+h(\BX(t),r(t),y))-U_i^{\theta}(\BX(t),r(t))\right)\tilde N(dt,dy)
\end{aligned}
$$
The quadratic variation for the Martingale $M_U(t)$ is given by
\begin{equation}
\langle M_U,M_U\rangle_t =\int_0^t\int_{\R_+}\left|U_i^{\theta}(\BX(s),r(s)+h(\BX(s),r(s),y))-U_i^{\theta}(\BX(s),r(s))\right|^2 dy ds. 
\end{equation}
Which we can rewrite as 
\begin{equation}
\langle M_U,M_U\rangle_t =\int_0^t\sum_{j\neq r(s)}q_{r(s)j}(\BX(s),r(s))\left|U_i^{\theta}(\BX(s),j)-U_i^{\theta}(\BX(s),r(s))\right|^2 ds. 
\end{equation}

\begin{equation}
\begin{split}
\langle M_U,M_U\rangle_t \leq&\int_0^t\sum_{j\neq r(s)}q_{r(s)j}(\BX(s),r(s))U_i^{2\theta}(\BX(s),j) ds\\
+&\int_0^t \sum_{j\neq r(s)}q_{r(s)j}(\BX(s),r(s)) U_i^{2\theta}(\BX(s),r(s)) ds 
\end{split}    
\end{equation}
Then using \ref{a:bound} and the fact that 
\[
U_i^{\theta}(\bx,k)=F^{\theta}(\bx,k)\left(\frac{\check \bx_i}{\prod_{j\in I}\hat \bx_j }    \right),
\]
we get
\begin{equation}
\langle M_U,M_U\rangle_t \leq\int_0^t\sum_{j\neq r(s)}q_{r(s)j}(\BX(s),r(s))(1+M_F)U_i^{2\theta}(\BX(s),r(s) ds.
\end{equation}
Since we know that 
\[
\max_{ij}\sup_{(\bx,k)\in\R^n_+\times\CN}|q_{ij}(\bx)|\leq C_0
\]
we get 
\begin{equation}\label{qvarbound}
\langle M_U,M_U\rangle_t \leq C_0|\CN|(1+M_F)\int_0^t U_i^{2\theta}(\BX(s),r(s)) ds.    
\end{equation}

In view of the Burkholder-Davis-Gundy Inequality and \eqref{qvarbound}, and denoting\\
$\tilde{c}=\tilde{C}C_0|\CN|(1+M_F)$

\begin{equation}\label{BDG_2theta}
\E_{\bx,k}\sup_{s\in[0,t]} |M_U(s)|^2
\leq \tilde{C}\E_{\bx,k}(\langle M_U,M_U\rangle_t)\leq \tilde c \E_{\bx,k}\left(\int_0^t U_i^{2\theta}(\BX(t),r(t))dt\right) 
\end{equation}
Since we know that 
$$U_i^{\theta}(\bx,k)= (F(\bx,k))^{\theta}\left(\frac{\BX_i^{\check p_i}}{\prod_{j\in I}\BX_j^{\hat p_j}} \right)$$

\begin{equation}
\begin{split}
\Lom U_i^{\theta}(\BX(t),r(t))=& \theta U_i^{\theta-1}(\BX(t),r(t))\sum_{k=1}^n\BX_k(t)f_k(\BX(t),r(t))\frac{\partial U_i}{\partial \bx_k}(\BX(t),r(t))\\
+&\sum_{j\in\CN}q_{r(s)j}(\BX(t))U_i^{\theta}(\BX(t),j)    
\end{split}    
\end{equation}
We simplify the above to 
\begin{equation}
\begin{split}
\Lom U_i^{\theta}(\BX(t),r(t))=& \theta U_i^{\theta}(\BX(t),r(t))\sum_{k=1}^n \frac{\BX_k(t)f_k(\BX(t),r(t))}{F(\BX(t),r(t))}\frac{\partial F}{\partial x_k}(\BX(t),r(t))\\
+& \theta U_i^{\theta}(\BX(t),r(t))\left(\check p_if_i(\BX(t),r(t))-\sum_{j\in I}\hat p_j f_j(\BX(t),r(t))
\right)\\
+&\sum_{j\in\CN}q_{r(s)j}(\BX(t))U_i^{\theta}(\BX(t),j)
\end{split}    
\end{equation}

\begin{equation}
\begin{split}
\Lom U_i^{\theta}(\BX(t),r(t))
=&\left(\frac{\BX_i(t)^{\check p_i}}{\prod_{j\in I}\BX_j(t)^{\hat p_j}}\right)\left(\theta F^{\theta}(\BX(t),r(t))\sum_{k=1}^n \frac{\BX_k(t)f_k(\BX(t),r(t))}{F(\BX(t),r(t))}\frac{\partial F}{\partial x_k}(\BX(t),r(t))\right)\\
+&\left(\frac{\BX_i(t)^{\check p_i}}{\prod_{j\in I}\BX_j(t)^{\hat p_j}}\right)\left(\sum_{j\in\CN}q_{r(s)j}(\BX(t))F^{\theta}(\BX(t),j)        \right)\\
+&\left(\frac{\BX_i(t)^{\check p_i}}{\prod_{j\in I}\BX_j(t)^{\hat p_j}}\right)\theta F^{\theta}(\BX(t),r(t))\left(\check p_if_i(\BX(t),r(t))-\sum_{j\in I}\hat p_j f_j(\BX(t),r(t))
\right)
\end{split}    
\end{equation}
We can simplify the above equation to 
\begin{equation}
\begin{split}
\Lom U_i^{\theta}(\BX(t),r(t))
=&\left(\frac{\BX_i(t)^{\check p_i}}{\prod_{j\in I}\BX_j(t)^{\hat p_j}}\right)\Lom F^{\theta}(\BX(t),r(t))\\
+& \theta U_i^{\theta}(\BX(t),r(t))  )\left(\check p_if_i(\BX(t),r(t))-\sum_{j\in I}\hat p_j f_j(\BX(t),r(t))
\right)
\end{split}    
\end{equation}

Using Dynkin's Formula for $U_i^{\theta}(\BX(t),r(t))$, we get that 
\begin{equation}
\begin{split}
\E_{\bx,k}(U_i^{\theta}(\BX(t),r(t)))=& U_i^{\theta}(\bx,k) +\E_{\bx,k}\left(\int_0^t \Lom(U_i^{\theta}(\BX(s).r(s)) ds\right)\\
\E_{\bx,k}(U_i^{\theta}(\BX(t),r(t)))\leq&  U_i^{\theta}(\bx,k) +H\theta\cdot\E_{\bx,k}\left(\int_0^t  U^{\theta}_i(\BX(s),r(s)) ds\right)\\
\end{split}
\end{equation}
Using Gronwall's inequality, we get the first desired inequality
\begin{equation}\label{res_1}
\begin{split}
\E_{\bx,k}(U_i^{\theta}(\BX(t),r(t)))\leq& e^{H\theta t} U_i^{\theta}(\bx,k)\\
\E_{\bx,k}\left(U_{\theta}(\BX(t),r(t))\right)=&\sum_{i\in I^c}\E_{\bx,k}\left(U_i^{\theta}(\BX(t),r(t))\right)\leq \sum_{i\in I^c}e^{H\theta t}U_i^{\theta}(\bx,k)=e^{H\theta t}U_{\theta}(\bx,k)\\
\end{split} 
\end{equation}

We write Ito's Formula for $U_i^{2\theta}(\BX(s),r(s))$, 
\begin{equation}\label{Ito_2theta}
\begin{split}
U^{2\theta}_i(\BX(t),r(t))=&U^{2\theta}_i(\bx,k)+ \int_0^t \Lom U_i^{2\theta}(\BX(s),r(s)) ds+ M_{2U}(t)\\
\end{split}     
\end{equation}
Where $M_{2U}(t)$ denotes the martingale
\begin{equation}
M_{2U}(t)=\int_0^t\int_{\R_+}\left(U_i^{2\theta}(\BX(s),r(s)+h(\BX(s),r(s),y))-U_i^{2\theta}(\BX(s),r(s))\right)\tilde{N}(dy,ds)    
\end{equation}

Observe that we have the following identity
\begin{equation}
\begin{split}
&\Lom(U_i^{2\theta}(\BX(s),r(s)))- 2U^{\theta}_i(\BX(s),r(s))\Lom(U_i^{\theta}(\BX(s),r(s)))\\
=&\sum_{j\neq r(s)}q_{r(s)j}(\BX(s))|U_i^{\theta}(\BX(s),j)-U_i^{\theta}(\BX(s),r(s))|^2
\end{split}
\end{equation}
Hence as a result this gives us 
\begin{equation}\label{theta_reduc}
\begin{split}
&\int_0^t \Lom(U_i^{2\theta}(\BX(s),r(s))) ds = \int_0^t 2U^{\theta}_i(\BX(s),r(s))\Lom(U_i^{\theta}(\BX(s),r(s))) ds + \langle M_U,M_U\rangle_t     
\end{split}    
\end{equation}
Combining \eqref{Ito_2theta} and \eqref{theta_reduc} we have 
\begin{equation}\label{Ito_2theta2}
\begin{split}
U^{2\theta}_i(\BX(t),r(t))\leq &U^{2\theta}_i(\bx,k)+ \int_0^t 2U^{\theta}_i(\BX(s),r(s))\Lom(U_i^{\theta}(\BX(s),r(s))) ds + \langle M_U,M_U\rangle_t\\
+& M_{2U}(t)
\end{split}    
\end{equation}
We take supremum on both sides to get
\begin{equation}\label{Ito_sup}
\begin{split}
\sup_{s\in[0,t]}U^{2\theta}_i(\BX(s),r(s))\leq &U^{2\theta}_i(\bx,k)+ 2H\theta\cdot
\int_0^t U^{2\theta}_i(\BX(s),r(s)) ds + \langle M_U,M_U\rangle_t\\
+& \sup_{s\in[0,t]} M_{2U}(s)
\end{split}    
\end{equation}

Using Burkholder-Davis-Gundy Inequality, we get 
\begin{equation}
\E_{\bx,k}\left(\sup_{s\in[0,t]} |M_{2U}(s)|     \right)\leq \tilde{C}_1 \E_{\bx,k}\left(\left(\langle M_{2U},M_{2U}\rangle_t\right)^{1/2}    \right)\leq \tilde{C}_1\cdot\left(\E_{\bx,k} \langle M_{2U},M_{2U}\rangle_t   \right)^{1/2}   
\end{equation}

\begin{equation}\label{u_2Ueq}
\begin{split}
\E_{\bx,k}\left(\sup_{s\in[0,t]}U^{2\theta}_i(\BX(s),r(s))\right)\leq &U^{2\theta}_i(\bx,k)+ 2H\theta\cdot
\E_{\bx,k}\left(\int_0^t U^{2\theta}_i(\BX(s),r(s)) ds\right) + \E_{\bx,k}\left(\langle M_U,M_U\rangle_t\right)\\
+&\tilde{C}_1 \E_{\bx,k}\left(\left(\langle M_{2U},M_{2U}\rangle_t\right)^{1/2}    \right)
\end{split}    
\end{equation}

\begin{equation}
\langle M_{2U},M_{2U}\rangle_t= \int_0^t \int_{\R_+} \left|U_i^{2\theta}(\BX(s),r(s)+h(\BX(s),r(s),y))-U_i^{2\theta}(\BX(s),r(s))\right|^2 dy ds     
\end{equation}
Using the boundedness assumption \ref{a:bound}
\begin{equation}
\langle M_{2U},M_{2U}\rangle_t\leq \tilde{C}_2\int_0^t U_i^{4\theta}(\BX(s),r(s)) ds 
\end{equation}
By simple Holder inequality we have 
\[
\int_0^t U_i^{4\theta}(\BX(s),r(s)) ds\leq \left(\sup_{s\in[0,t]}U_i^{2\theta}(\BX(s),r(s))\right)\left(\int_0^t U_i^{2\theta}(\BX(s),r(s)) ds\right)
\]
Now putting the above in \eqref{u_2Ueq}
\begin{equation}
\begin{split}
\E_{\bx,k}\left(\sup_{s\in[0,t]}U^{2\theta}_i(\BX(s),r(s))\right)\leq &U^{2\theta}_i(\bx,k)+ 2H\theta\cdot
\E_{\bx,k}\left(\int_0^t U^{2\theta}_i(\BX(s),r(s)) ds\right) + \E_{\bx,k}\left(\langle M_U,M_U\rangle_t\right)\\
+&\tilde{C}_1\sqrt{\tilde{C}_2} \E_{\bx,k}\left(  \left(\sup_{s\in[0,t]}U_i^{2\theta}(\BX(s),r(s))\right)^{1/2}\left(\int_0^t U_i^{2\theta}(\BX(s),r(s)) ds\right)^{1/2}  \right)
\end{split}    
\end{equation}
Now using Young's inequality choose $\eps_1$ such that 
$\sqrt{ab}\leq \eps_1 a+ C_{\eps_1}b$, and $1-\tilde{C}_1\sqrt{\tilde{C}_2}\eps_1\geq 1/2$
As a result we get 
\begin{equation}
\begin{split}
\E_{\bx,k}\left(\sup_{s\in[0,t]}U^{2\theta}_i(\BX(s),r(s))\right)\leq &2U^{2\theta}_i(\bx,k)+ 4H\theta\cdot
\E_{\bx,k}\left(\int_0^t U^{2\theta}_i(\BX(s),r(s)) ds\right) + 2\E_{\bx,k}\left(\langle M_U,M_U\rangle_t\right)\\
+&2\tilde{C}_1\sqrt{\tilde{C}_2}\tilde{C}_{\epsilon} \E_{\bx,k} \left(\int_0^t U_i^{2\theta}(\BX(s),r(s)) ds  \right)    
\end{split}    
\end{equation}
Now denoting $H_0(\theta)=4H\theta+\tilde{c}+ 2\tilde{C}_1\sqrt{\tilde{C}_2}\tilde{C}_{\epsilon}$, we get the following estimate

\begin{equation}\label{Young_abs}
\E_{\bx,k}\left(\sup_{s\in[0,t]} U^{2\theta}_i(\BX(s),r(s))\right)\leq 2U^{2\theta}_i(\bx,k)+ H_0(\theta)\cdot\E_{\bx,k}\left(\int_0^t U^{2\theta}_i(\BX(s),r(s)) ds     \right).        
\end{equation}

From \eqref{Young_abs}, we get the following inequality
\begin{equation}
\E_{\bx,k}\left(\sup_{s\in[0,t]}U^{2\theta}_i(\BX(s),r(s))\right)\leq 2U^{2\theta}_i(\bx,k)+ H_0(\theta)\cdot \E_{\bx,k}\left(\int_0^t \left(\sup_{s\in[0,t]}U^{2\theta}_i(\BX(s),r(s))\right) ds   \right).    
\end{equation}
then by Gronwall's ineqaulity,  
\begin{equation}
\begin{split}
\E_{\bx,k}\left(\sup_{s\in[0,t]}U^{2\theta}_i(\BX(s),r(s))\right)\leq& 2e^{H_0(\theta)t} U_i^{2\theta}(\bx,k) \\
\E_{\bx,k}\left(\sup_{s\in[0,t]}U_{2\theta}(\BX(s),r(s))\right)\leq& \sum_{i\in I^c} \E_{\bx,k}\left(sup_{s\in[0,t]} U_i^{2\theta}(\BX(s),r(s))\right)\leq 2e^{H_0(\theta)t}\sum_{i\in I^c}U_i^{2\theta}(\bx,k)\\ 
\end{split}
\end{equation}
Now take $H_1(\theta)=H_0(\theta)+\ln(2)$, then for $t\geq 1$, we have 
\begin{equation}\label{res_2}
\E_{\bx,k}\left(\sup_{s\in[0,t]}U_{2\theta}(\BX(s),r(s))\right)\leq e^{H_1(\theta)t}U_{2\theta}(\bx,k)   
\end{equation}

This proves the desired inequalities.

\end{proof}
\begin{rmk}
It is key to note that the inequalities \eqref{A.1}, \eqref{e1-lm2.2} and \eqref{et1.1} hold if $|p_i|<\delta_0$ no matter if the $p_i$'s are negative or positive. This then allows us to have the same kind of estimates for $U_\theta$ and $V_\theta$.
\end{rmk}
\begin{thm}\label{thm4.1}
Suppose Assumptions \ref{a1-pdm}, \ref{a:bound} and  \ref{a4-pdm} hold and $\R_+^I$ is accessible. Then for any $\delta<\delta_0$
and any $\bx\in\R^{n,\circ}_+$ we have
\begin{equation}\label{e.extinction}
\lim_{t\to\infty}\E_{\bx,k} \bigwedge_{i=1}^n X_i^{\delta}(t)=0,
\end{equation}
where $\bigwedge_{i=1}^n a_i=\min_{i=1,\dots,n}\{a_i\}.$

\end{thm}
\begin{proof}
Just as in \eqref{et1.1}, we have 
\begin{equation}\label{et3.1}
\Lom U_\theta(\bx,k)\leq -\theta\gamma_bU_\theta(\bx,k) \text{ if } \|x\|\geq M.
\end{equation}
Let
$$C_U:=\sup\left\{\dfrac{\prod_{i\in I} x_i^{\hat p_i}}{F(\bx,k)}: \bx\in\R^{n,\circ}_+\right\}<\infty,$$
$$\varsigma:=\dfrac{\delta_e^{\check p\theta}}{C_U^\theta},$$

and
$$\xi:=\inf\left\{t\geq0: U^\theta(\BX(t),r(t))\geq \varsigma\right\}.$$
Clearly, if $U_\theta(\bx,k)<\varsigma$, then $\xi>0$ and
for any $i\in I^c$, we get
\begin{equation}\label{e:ine}
X_i(t)\leq \delta_e\,, t\in [0,\xi).
\end{equation}
Let $$\tilde U_\theta(\bx,k):=\varsigma\wedge U_\theta(\bx,k).$$
We have from the concavity of $x\mapsto x\wedge \varsigma$ and Jensen's inequality that
$$\E_{\bx,k} \tilde U_\theta(\BX(T), r(T))\leq\varsigma\wedge \E U_\theta(\BX(T), r(T)).$$
Let $\tau$ be defined by \eqref{e:tau}. By \eqref{et3.1} and Dynkin's formula, we have that
$$
\begin{aligned}
\E_{\bx,k}&\left[ \exp\left(\theta\gamma_b(\tau\wedge\xi\wedge n_eT_e)\right)U_\theta(\BX(\tau\wedge\xi\wedge n_eT_e),r(\tau\wedge\xi\wedge n_eT_e))\right]\\
&\leq U_\theta(\bx,k) +\E_{\bx,k} \int_0^{\theta\gamma_b(\tau\wedge\xi\wedge n_eT_e)}\exp(\theta\gamma_b s)[\Lom U_\theta(\BX(s),r(s))+ \theta\gamma_bU_\theta(\BX(s),r(s))]ds\\
&\leq U_\theta(\bx,k).
\end{aligned}
$$
As a result,
\begin{equation}\label{et3.3}
\begin{aligned}
U_\theta(\bx,k)\geq&
\E_{\bx,k}\left[ \exp\left(\theta\gamma_b(\tau\wedge\xi\wedge n_eT_e)\right)U_\theta(\BX(\tau\wedge\xi\wedge n_eT_e),r(\tau\wedge\xi\wedge n_eT_e))\right]\\
\geq& \E_{\bx,k} \left[\1_{\{\tau\wedge\xi\wedge(n_e-1)T_e=\tau\}}U_\theta(\BX(\tau),r(\tau))\right]\\
&+ \E_{\bx,k} \left[\1_{\{\tau\wedge\xi\wedge(n_e-1)T_e=\xi\}}U_\theta(\BX(\xi)),r(\xi))\right]\\
 &+\exp\left(\theta\gamma_b (n_e-1)T_e\right) \E_{\bx,k} \left[\1_{\{(n_e-1)<\tau\wedge\xi<n_eT\}}U_\theta(\BX(\tau\wedge\xi),r(\tau\wedge\xi))\right]\\
&+\exp\left(\theta\gamma_b n_eT_e\right) \E_{\bx,k} \left[\1_{\{\tau\wedge\xi\geq n_eT_e\}}U_\theta(\BX(n_eT_e)),r(n_eT_e))\right].\\
 \end{aligned}
\end{equation}

By the strong Markov property of $(\BX(t),r(t))$ and
Proposition \ref{prop4.1} (which we can use because of \eqref{e:ine})
\begin{equation}\label{et3.4}
\begin{aligned}
\E_{\bx,k}&\left[ \1_{\{\tau\wedge\xi\wedge(n_e-1)T_e=\tau\}}U_\theta(\BX(n_eT_e)),r(n_eT_e))\right]\\
&\leq
 \E_{\bx,k} \left[\1_{\{\tau\wedge\xi\wedge(n_e-1)T_e=\tau\}}\exp\left(-0.5\theta p_e(n_eT_e-\tau)\right)U_\theta(\BX(\tau\wedge\xi),r(\tau\wedge\xi))\right]\\
 &\leq \exp(-0.5\theta p_eT_e)
 \E_{\bx,k}\left[\1_{\{\tau\wedge\xi\wedge(n_e-1)T_e=\tau\}}U_\theta(\BX(\tau\wedge\xi),r(\tau\wedge\xi))\right].
 \end{aligned}
\end{equation}
Similarly, by the strong Markov property of $(\BX(t),r(t))$ and
Lemma \ref{lm3.3}, we obtain
\begin{equation}\label{et3.5}
\begin{aligned}
\E_{\bx,k}&\left[ \1_{\{(n_e-1)T_e<\tau\wedge\xi<n_eT_e\}}U_\theta(\BX(n_eT_e)),r(n_eT_e))\right]\\
&\leq
 \E_{\bx,k} \left[\1_{\{(n_e-1)T_e<\tau\wedge\xi<n_eT_e\}}\exp\left(\theta H(n_eT_e-\tau)\right)U_\theta(\BX(\tau\wedge\xi),r(\tau\wedge\xi))\right]\\
 &\leq \exp(\theta HT_e)\E_{\bx,k}\left[\1_{\{(n_e-1)T_e<\tau\wedge\xi<n_eT_e\}}U_\theta(\BX(\tau\wedge\xi),r(\tau\wedge\xi))\right].
 \end{aligned}
\end{equation}
Since $\tilde U_\theta(\BX(n_eT_e)),r(n_eT_e))\leq U_\theta(\BX(n_eT_e\wedge\xi),r(n_eT_e\wedge \xi))$ one can note that
\begin{equation}\label{et3.6}
\E_{\bx,k} \left[\1_{\{\tau\wedge\xi\wedge(n_e-1)T_e=\xi\}}\tilde U_\theta(\BX(n_eT_e)),r(n_eT_e))\right]\leq \E_{\bx,k} \left[\1_{\{\tau\wedge\xi\wedge(n_e-1)T_e=\xi\}}U_\theta(\BX(\xi)),r(\xi))\right].
\end{equation}
If $U_\theta(\bx,k)<\varsigma$ then applying \eqref{et3.4}, \eqref{et3.5} and \eqref{et3.6} to \eqref{et3.3} yields
\begin{equation}\label{et3.8}
\begin{aligned}
\tilde U_\theta(\bx,k)=U_\theta(\bx,k)
\geq& \E_{\bx,k} \left[\1_{\{\tau\wedge\xi\wedge(n_e-1)T_e=\tau\}}U_\theta(\BX(\tau),r(\tau))\right]\\
&+ \E_{\bx,k} \left[\1_{\{\tau\wedge\xi\wedge(n_e-1)T_e=\xi\}}U_\theta(\BX(\xi)),r(\xi))\right]\\
 &+\exp\left(\theta\gamma_b (n_e-1)T_e\right) \E_{\bx,k} \left[\1_{\{(n_e-1)<\tau\wedge\xi<n_eT\}}U_\theta(\BX(\tau\wedge\xi),r(\tau\wedge\xi))\right]\\
&+\exp\left(\theta\gamma_b n_eT_e\right) \E_{\bx,k} \left[\1_{\{\tau\wedge\xi\geq n_eT_e\}}U_\theta(\BX(n_eT_e)),r(n_eT_e))\right]\\
\geq&\exp\left(0.5\theta p_eT_e)\right)\E_{\bx,k} \left[\1_{\{\tau\wedge\xi\wedge(n_e-1)T_e=\tau\}}U_\theta(\BX(n_eT_e)),r(n_eT_e))\right]\\
&+ \E_{\bx,k} \left[\1_{\{\tau\wedge\xi\wedge(n_e-1)T_e=\xi\}}\tilde U_\theta(\BX(n_eT_e)),r(n_eT_e))\right]\\
 &+\exp\left(\theta T_e\left(\gamma_b (n_e-1)- H\right)\right)\E_{\bx,k} \left[\1_{\{(n_e-1)<\tau\wedge\xi<n_eT\}}U_\theta(\BX(n_eT_e)),r(n_eT_e))\right]\\
&+\exp\left(\theta\gamma_b n_eT_e\right) \E_{\bx,k} \left[\1_{\{\tau\wedge\xi\geq n_eT_e\}}U_\theta(\BX(n_eT_e)),r(n_eT_e))\right]\\
\geq& \E_{\bx,k} \tilde U_\theta(\BX(n_eT_e)),r(n_eT_e)) \,\quad\text{ (since } \tilde U_\theta(\cdot)\leq U_\theta(\cdot))
 \end{aligned}
\end{equation}
Clearly, if $U_\theta(\bx,k)\geq\varsigma$ then
\begin{equation}\label{et3.8a}
\E_{\bx,k} \tilde U_\theta(\BX(n_eT_e)),r(n_eT_e)) \leq \varsigma=\tilde U_\theta(\bx,k).
\end{equation}
As a result of \eqref{et3.8}, \eqref{et3.8a} and the Markov property of $(\BX(t),r(t))$, the sequence
$\{Y(\ell): \ell\in\N\}$ where $Y(\ell):=\tilde U_\theta(\BX(\ell n_eT_e))$ is a supermartingale.
Let $\tilde\xi:=\inf\{\ell\in\N: U_\theta(\BX(t))\geq \varsigma \text{ for some } t\leq \ell n_eT_e. \}$.
If $U_\theta(\bx,k)\leq \varsigma\eps$ we have
\begin{equation}\label{e:EY_ineq}
\E_{\bx,k}  Y(\ell\wedge\tilde\xi)\leq\E_{\bx,k} Y(0)=U_\theta(\bx,k)\leq \varsigma\eps\,\text{ for all }\, \ell\in\N.
\end{equation}
 As a result  \eqref{e:EY_ineq} combined with the Markov inequality yields
$$\PP_{\bx,k}\{\tilde\xi<k\}=\PP_{\bx,k}(Y(\bar\xi\wedge k)\geq\varsigma)\leq \varsigma^{-1}\E_{\bx,k} Y(k\wedge\tilde\xi)\leq \eps.
$$
Next, let $k\to\infty$ to get
\begin{equation}\label{et3.7}
\PP_{\bx,k}\{\xi<\infty\}=\PP_{\bx,k}\{\tilde\xi<\infty\}\leq \eps.
\end{equation}

	We have from the second inequality of \eqref{et3.8} that
	$$
	\E_{\bx,k} \1_{\{\xi>n_eT_e\}} U_\theta(\BX(n_eT_e)),r(n_eT_e))\leq \max\{e^{-0.5\theta p_eT_e}, e^{-\theta T_e}, e^{-\gamma_b n_eT_e}\} U_\theta(\bx,k)
	$$
	which combined with the Markov property of $(\BX(t),r(t))$ implies
	$$
	\E_{\bx,k} \1_{\{\xi>\ell n_eT_e\}} U_\theta(\BX(\ell n_eT_e)),r(\ell n_eT_e))\leq \rho_0^\ell U_\theta(\bx,k)
	$$
	where 
	$\max\{e^{-0.5\theta p_eT_e}, e^{-\theta T_e}, e^{-\gamma_b n_eT_e}\}=\rho_0<1$.
Due to \eqref{u3}, we have
	$$
\E_{\bx,k} \1_{\{\xi>\ell n_eT_e\}}\sup_{t\in[\ell n_eT_e,(\ell+1) n_eT_e]} U_\theta(\BX(t)),r(t))\leq \rho_0^\ell \hat H_{n_eT_e,\theta} U_\theta(\bx,k)
$$
 Picking $\rho_0<\rho_1<1$, and using Markov's inequality we have
\begin{align*}
\PP_{\bx,k}\left\{\1_{\{\xi>\ell n_eT_e\}}\sup_{t\in[\ell n_eT_e,(\ell+1) n_eT_e]} U_\theta(\BX(t)),r(t)\geq\rho_1^{\ell+1}\right\}\leq \left(\dfrac{\rho_0}{\rho_1}\right)^{\ell+1}\hat H_{n_eT_e,\theta}U_\theta(\bx,k).
\end{align*}
Since $\sum_{k=0}^\infty\left(\frac{\rho_0}{\rho_1}\right)^{\ell+1}<\infty$, from an application of Borel-Cantelli lemma, we obtain
$$\lim_{\ell\to\infty} \1_{\{\xi>\ell n_eT_e\}}\sup_{t\in[\ell n_eT_e,(\ell+1) n_eT_e]} U_\theta(\BX(t)),r(t)=0 \,\,\text{ a.s.}  $$
which, in view of $\PP_{\bx,k}\{\xi<\infty\}\leq \epsilon,$ implies 
\begin{equation}\label{e:Utheta_lim}
\PP_{\bx,k}\{ \lim_{t\to\infty}U_{\theta}(\BX(t),r(t))=0\}\geq 1-\epsilon
\end{equation}
Since we know that by definition
\[
U_{\theta}(\bx,k)= \sum_{i\in I^c}\left(F(\bx,k)\left(\frac{ x_i^{\check p_i}}{\prod_{i\in I } x_j^{\hat p_j}}   \right)\right)^{\theta}
\]
Hence for each $i\in I^c$, we have 

\begin{equation}\label{U_i_def}
U_i^{\theta}(\bx,k)=\left(F(\bx,k)\left(\frac{ x_i^{\check p_i}}{\prod_{i\in I} x_j^{\hat p_j}}   \right)\right)^{\theta}\geq  \frac{x_i^{\theta\check p_i}}{C_U^{\theta}}
\end{equation}

This together with \eqref{e:EY_ineq} shows that for all initial conditions such that $U_\theta(\bx,k)\leq \varsigma\eps$ one has
\begin{equation}\label{e:ext_U_theta_I}
\PP_{\bx,k}\left\{\lim_{t\to\infty}X_i(t)=0, i\in I^c\right\}\geq 1-\varepsilon.
\end{equation}

Since we are assuming $\R^{I}_+$ to be accessible from any point $(\bx,k)\in \R^n_{+}\times\CN$ if we define the stopping time
\begin{equation}\label{e:tau_0}
\tau_0=\inf\left(t>0: X_i(t)\leq \eps\delta_e \text{ for all } i\in I^c \right)
\end{equation}
then from Lemma 3.2 in \cite{BBMZ15}, we know 
\begin{equation}\label{e:tau_0_bound}
\PP_{\bx,k}(\tau_0<\infty)>0, ~(\bx,k)\in \R^{n,\circ}\times\CN.
\end{equation}

Then $U_{\theta}(\BX(\tau_0),r(\tau_0))\leq \eps\varsigma$ so combining strong markov property and the supermartingale argument, we get that 

$$\PP_{\bx,k}\left\{\lim_{t\to\infty} X_i(t)=0, i\in I^c\right\}>0.$$
By Assumption \ref{a4-pdm} and noting that we also assumed that $\R_+^I$ is accessible we have that
$$\PP_{\bx,k}\left\{\lim_{t\to\infty} X_i(t)=0 \text{ for some } i\right\}>0, \forall\,\bx\in\R^{n,\circ}_+.$$
This implies that there is no invariant measure on $\R^{n,\circ}_+\times\CN$.

For any
$\pi$ with $\suppo(\pi)\subset\partial\R^n_+$,
we have
$$\sum_{k\in\CN}\int_{\R^n_+}\left(\bigwedge_{i=1}^n x_i^{\delta}\right)\pi(d\bx,k)=0.$$
Note that
\[
 \left(\bigwedge_{i=1}^n x_i^{\delta}\right)\leq K(1+\|\bx\|)^{\delta}.
\]
Then $\{P(t,\bx_0, k, \cdot,\cdot), t\geq 0\}$ is tight and every limit point is an invariant probability measure. By the above two equations we therefore get 
$$\lim_{t\to\infty}\E_{\bx,k} \bigwedge_{i=1}^n X_i^{\delta}(t)=\lim_{t\to\infty}\sum_{j\in\CN}\int_{\R^n_+}\left(\bigwedge_{i=1}^n x_i^{\delta}\right)P(t,\bx_0,k, d\bx,j)=0$$
as desired.
\end{proof}

Without assuming the accessibility of the boundary we can prove the following extinction result.
\begin{thm}\label{t:exxx}
Suppose Assumptions \ref{a1-pdm} and \ref{a:bound} hold. If $I$ satisfies Assumption \ref{a4-pdm} there exists
$\beta_I>0$ such that, for any a compact set $\K^I\subset
\Se^{I}_+$,
we have
$$
\lim_{\dist(\bz,\K^I)\to0, \bz\in
S^\circ}\PP_\bz\left\{\lim_{t\to\infty}\dfrac{\ln X_i(t)}t\leq
-\beta_I, i\in I^c\right\}=1
$$
\end{thm}

\begin{proof}
Consider a sequence $z_n=(\bx_n,k_n)$ such that $\dist(z_n,\K^I)\to 0$. Then for $i\in I^c$, $(\bx_n)_i\to 0$ as $n\to\infty$, and $\|\bx_n\|$ is bounded .Following the same idea, as in Theorem \ref{thm4.1}, given $\eps>0$, there is an $N$ such that for all $n\geq N$ we have $(\bx_n)_i \leq \eps\delta_e$ for $i\in I^c$. Then we have for all $z_n=(\bx_n,k_n)$
\[
\PP_{\bx_n,k_n}\left(\lim_{t\to\infty} U^{\theta}_i(X(t),r(t))=0                   \right)\geq 1-\eps
\]
We can see even further from the proof of Theorem \ref{thm4.1} that for almost every $\omega\in \{\tilde{\xi}=\infty, (\BX(0),r(0)=(\bx_n,k_n)\}$, there is some $k\in \N$ and $\rho_1<1$ such that for all $\ell\geq k$, and $t\in [\ell n_eT_e,\infty)$

\begin{equation}
    U_i^{\theta}(\BX(t,\omega),r(t,\omega))\1_{t\in[\ell n_eT_e,(\ell+1)n_eT_e]}\leq \rho_1^{\ell +1}
\end{equation}
Now let $\alpha_I$ be a positive number such that $\displaystyle \rho_1= e^{-\alpha_{I}n_eT_e}$. then we get 

\begin{equation}
 U_i^{\theta}(\BX(t,\omega),r(t,\omega))\1_{t\in[\ell n_eT_e,(\ell+1)n_eT_e]}\leq \exp(-\alpha_I(\ell+1)n_eT_e)\leq e^{-\alpha_I t}    
\end{equation}
Hence as a result we have for each $i\in I^c$
\begin{equation}\label{lnU_i_ineq}
\PP_{\bx_n,k_n}\left(\lim_{t\to\infty}\frac{\ln(U_i^{\theta}(\BX(t),r(t))}{t} \leq -\alpha_I      \right)\geq 1-\eps
\end{equation}
From the definition of $U_i^{\theta}$ one can see that
\begin{equation}
  \frac{\ln(U_i^{\theta}(\BX(t),r(t))}{t}=\theta\left(\frac{\ln(F(\BX(t),r(t))}{t}\right)+\check{p_i}\theta \left(\frac{\ln(\BX_i(t))}{t}\right)-\sum_{j\in I}\theta \hat p_j \left(\frac{\ln(\BX_j(t))}{t}\right)
\end{equation}

\begin{equation}\label{limsup_lnX_i}
\limsup_{t\to\infty}\frac{\ln(\BX_i(t))}{t}=\frac{1}{\theta \check p}\left( \limsup_{t\to\infty}\frac{\ln(U_i^{\theta}(\BX(t),r(t))}{t}\right)+ \frac{1}{\theta}\left( \limsup_{t\to\infty} \frac{1}{t}\ln\left( \frac{\prod_{j\in I} \BX_j^{\hat p_j}(t)}{F(\BX(t),r(t))}\right)      \right).      
\end{equation}
Using the fact that $\frac{\prod_{j\in I} \BX_j^{\hat p_j}(t)}{F(\BX(t),r(t))}\leq C_U $, we clearly havbe
\[
\limsup_{t\to\infty} \frac{1}{t}\ln\left( \frac{\prod_{j\in I} \BX_j^{\hat p_j}(t)}{F(\BX(t),r(t))}\right)  \leq 0
\]
Now combining \eqref{lnU_i_ineq} and \eqref{limsup_lnX_i} we get for each $i\in I^c$ 

\begin{equation}
 \PP_{\bx_n,k_n}\left( \limsup_{t\to\infty}\frac{\ln(\BX_i(t))}{t} \leq -\frac{\alpha_I}{\theta \check p_i}  \right)\geq 1-\eps.   
\end{equation}
Taking $$\beta_I=\frac{\alpha_I}{\theta \max_{i\in I^c}\check p_i}$$ and letting $n\to\infty$ and 
\begin{equation}
\liminf_{n\to\infty}\PP_{\dist(z_n,\K^I)}\left( \limsup_{t\to\infty}\frac{\ln(\BX_i(t))}{t} \leq -\beta_I 
\text{ for all } i\in I^c\right)\geq  1-\eps.    
\end{equation}
Hence we get by taking $\eps\to 0$
\begin{equation}
\lim_{\dist(z,\K^I)\to 0} \PP_{z}\left( \limsup_{t\to\infty}\frac{\ln(\BX_i(t))}{t} \leq -\beta_I 
\text{ for all } i\in I^c\right)=1    
\end{equation}

\end{proof}

To strengthen our extinction results we need a series of lemmas.
We also need the following lemmas.
\begin{lm}\label{lm4.4}
For any $\delta_1<\delta_0$ there exists $\hat{K}>0$ such that for every $(\bx,k)\in\R^n_+\times \CN$  
$$
\PP_{\bx,k}\left\{\limsup_{t\to\infty}\dfrac1t\int_0^t(F(\BX(s),r(s)))^{\delta_1}\left(1+f^M(\BX(s),r(s)) \right) ds\leq\tilde K\right\}=1
$$
There is $\hat K_1>1$ such that
\begin{equation}\label{e1-lm4.7}
\PP_{\bx,k}\left\{\liminf_{t\to\infty} \dfrac{1}t\int_0^t\1_{\{\|\BX(s)\|\leq \hat K_1\}}ds\geq\dfrac12\right\}=1,\,\bx\in\R^n_+.
\end{equation}

\end{lm}
\begin{lm}\label{lm4.7}
Under Assumption \ref{a1-pdm} \ref{a4-pdm}, there exists $\ell_0>0$ such that
$\mu(\K^0_I)>0$ for all $\mu\in\Conv(\M^{I,\circ})$ where $$\K^0_I=\{\bx\in\R^I_+: x_i\wedge x_i^{-1}\leq \ell_0, i\in I\}.$$
\end{lm}
\begin{proof}
	It follows straightforwardly from applying Theorem \ref{thm3.1} to the space $\R^I$.
\end{proof}
\begin{lm}\label{lm4.5}
Let Assumption \ref{a1-pdm} be satisfied.
Suppose we have a sample path  of $(\BX(t),r(t))$ satisfying
$$\limsup_{t\to\infty}\dfrac1t\int_0^t(F(\BX(s),r(s)))^{\delta_1}\left(1+f^M(\BX(s))\right)ds\leq \hat K$$
and that there is a sequence $(T_m)_{k\in\N}$ such that $\lim_{m\to\infty}T_m=\infty$ and
$\left(\wtd \Pi_{T_m}(\cdot)\right)_{k\in\N}$ converges weakly to an invariant probability measure $\pi$ of $\BX$
when $m\to\infty$ .
Then for this sample path, we have
$\sum_{k\in\CN}\int_{\R^n_+}h(\bx,k)\wtd\Pi^{\bx_m}_{T_m}(d\bx,k)\to \sum_{k\in\CN}\int_{\R^n_+}h(\bx,k)\pi(d\bx,k)$
for any continuous function $h:\R^n_+\to\R$ satisfying
$|h(\bx,k)|<K_h(F(\bx,k))^{\delta}(1+f^M(\bx,k)))\,,\,\bx\in \R^n_+$,
with $K_h$ a positive constant and $\delta\in[0,\delta_1)$.
\end{lm}
The proofs of Lemmas \ref{lm4.4} and \ref{lm4.7} are given in Appendix \ref{extinction lemmas} while that of Lemma \ref{lm4.5} is almost the same as that of Lemma \ref{lm2.4} and is left for the reader.
\begin{lm}\label{lm4.6}Let Assumption \ref{a1-pdm} be satisfied.
For any initial condition $(\BX(0), r(0))=(\bx,k)\in\R^n_+\times\CN$,
the family $\left\{\wtd \Pi_t(\cdot), t\geq 1\right\}$ is tight in $\R^n_+\times \CN$,
and its weak$^*$-limit set, denoted by $\U=\U(\omega)$
is a family of invariant probability measures of $(\BX(t),r(t))$ with probability 1.
\end{lm}
\begin{proof}
The tightness follows from Lemma \ref{lm4.4}.
The property of the weak$^*$-limit set of normalized occupation measures
was first proved in  \cite[Theorems 4, 5]{SBA11} for compact state spaces and then generalized to a locally compact state space in \cite[Theorem 4.2]{EHS15}. Similar results for general Markov processes can be found in \cite{B23}.
\end{proof}
\begin{lm}\label{lm4.9}
Suppose that $I$ satisfies Assumption \ref{a4-pdm}.
For any $\bx\in\R^{n,\circ}_+$,
$$\PP_{\bx,k}\Big\{\U(\omega)\subset\Conv(\M^I)\Big\}=\PP_{\bx,k}\Big\{\U(\omega)=\{\Conv(\M^{I,\circ}\}\Big\}$$
\end{lm}
\begin{proof}
Since $I$ satisfies Assumption \ref{a4-pdm},
it follows from \eqref{ae3.2} that there are $p^I_i>0, i\in I$ such that
\begin{equation}\label{e2-lm4.9}
\sum_{i\in I} p^I_i\lambda_i(\nu)>0, \nu\in\Conv(\partial\M^I).
\end{equation}
Note that It\^o's yields
\begin{equation}\label{e:itos}
\ln(F(\BX(t),r(t)))=\ln(F(\BX(0),r(0)))+\int_0^t  \Lom(\ln(F(\BX(s),r(s)))) ds  +  M_G(t)
\end{equation}
where $M_G(t)$ is the $L^2$ martingale
$$M_G(t)=\int_0^t \int_{\R_+} \ln\left[\frac{F(\BX(s),r(s)+h(\BX(s),r(s),y))}{F(\BX(s),r(s))}\right]\tilde N(dt,dy).
$$
Arguments very similar to those from Claim \ref{c:B2} show that the strong law of large numbers holds for the $L^2$ martingale $M_G$, so that
\begin{equation}\label{e8-lm4.9}
\PP_{\bx,k}\left\{\lim_{t\to\infty}\dfrac{M_G(t)}t=0\right\}=1.
\end{equation}
As a result of Lemmas \ref{lm2.3} and \ref{lm4.6} one gets
\begin{equation}\label{e:Lom_lnF}
\PP_{\bx,k}\left\{\lim\limits_{t\to\infty}\dfrac1t\int_0^t  \Lom(\ln(F(\BX(s),r(s))) ds=0\right\}=1.
\end{equation}
In light of It\^o's formula \eqref{e:itos}, it follows \eqref{e8-lm4.9} and \eqref{e:Lom_lnF} that for any $\bx\in\R^{n,\circ}_+$
\begin{equation}\label{e:lnF_to_0}
\PP_{\bx,k}\left\{\limsup_{t\to\infty}\dfrac{\ln (F(\BX(t),r(t)))}t=0\right\}=1.
\end{equation}
Based on Assumption \ref{a1-pdm} we get by the above equation that for all $(\bx,k)\in\R^{n,\circ}_+\times \CN$ one has
\begin{equation}\label{e7-lm4.9}
\PP_{\bx,k}\left\{\limsup_{t\to\infty}\dfrac{\ln X_i(t)}t\leq0,\,\, i=1,\dots,n\right\}\geq\PP_{\bx,k}\left\{\limsup_{t\to\infty}\dfrac{\ln (F(\BX(t),r(t)))}t=0\right\}=1.
\end{equation}

In view of  \eqref{e7-lm4.9}, to prove the lemma, it suffices to show that if 
\begin{itemize}
\item[a)] $\U(\omega)\subset\Conv(\M^I)$
\item[b)] \begin{equation}\label{e5-lm4.9}
\limsup_{t\to\infty}\dfrac{\ln X_i(t)}t\leq0,\,\, i=1,\dots,n
\end{equation}
\end{itemize}
hold then $\U(\omega)\subset\Conv(\M^{I,\circ}).$

We argue by contradiction. Assume there is a
sequence $\{t_m\}$ with $\lim_{m\to\infty}t_m=\infty$ such that
 $\wtd \Pi_{t_m}(\cdot)$ converges weakly to an invariant probability of the form
$\pi=(1-\rho)\pi_1+\rho\nu$
where $\rho\in(0,1]$, $\pi_1\in\Conv(\M^{I,\circ})$ and $\nu\in\Conv (\partial\M^I)$.
It follows from Lemmas \ref{lm4.4}, \ref{lm4.5} and \eqref{e2-lm4.9} that
\begin{equation}\label{e4-lm4.9}
\begin{aligned}
\lim_{m\to\infty}\dfrac1{t_m}\sum_{i\in I} p^I_i\int_0^{t_m}f_i(\BX(s),r(s))ds
=&\sum_{i\in I} p^I_i\lambda_i(\pi)\\
=&(1-\rho)\sum_{i\in I}p^I_i\lambda_i(\pi_1)+\rho \sum_{i\in I}p^I_i\lambda_i(\nu)\\
=&(1-\rho)\sum_{i\in I}p^I_i\lambda_i(\pi_1) \,\,\text{ (due to Lemma \ref{lm4.1})}\\
>&0.
\end{aligned}
\end{equation}
As a result of  \eqref{e4-lm4.9} 
$$
\begin{aligned}
\lim_{m\to\infty}\sum_{i\in I}p^I_i\dfrac{\ln X_i(t_m)}{t_m}
=&\lim_{m\to\infty}\dfrac1{t_m}\sum_{i\in I} p^I_i\int_0^{t_m}f_i(\BX(s),r(s))ds>0
\end{aligned}
$$
which contradicts \eqref{e5-lm4.9}.
This finishes the proof.
\end{proof}
\begin{thm}\label{thm4.2}
Suppose that Assumptions \ref{a1-pdm}, \ref{a:bound}, \ref{a4-pdm} and \ref{a5-pdm} are satisfied and $\M^1\neq \emptyset$. Suppose furthermore that $\bigcup_{I\in S} S^I_+$ is accessible.
Then for any $\bx\in\R^{n,\circ}_+, k\in \CN$
\begin{equation}\label{e0-thm4.2}
\sum_{I\in S} P_{\bx,k}^I=1
\end{equation}
where for $\bx\in\R^{n,\circ}_+, k\in \CN, I \in S$

$$P_{\bx,k}^I:=\PP_{\bx,k}\left\{\emptyset\neq\U(\omega)\subset\Conv\left(\M^{I,+}\right)
~\text{and}~\lim_{t\to\infty}\frac{\ln
X_j(t)}{t}\in\left\{\lambda_j(\mu):\mu\in\Conv\left(\M^{I,+}\right)\right\},
j\notin I\right\}.$$
Moreover, if $\R_+^I $ is accessible from $(\bx,k)$ then $P_{\bx,k}^I>0$.
\end{thm}
\begin{proof}
First, suppose that Assumption \ref{a5-pdm} is satisfied with nonempty $\M^2$.
Then, there is $\bq=(q_1,\dots,q_n)\in\R^{n,\circ}_+$ such that
$\|\bq\|=1$ and
\begin{equation}\label{e1-thm4.2}
\min_{\nu\in\M^2}\left\{\sum_iq_i\lambda_i(\nu)\right\}>0.
\end{equation}
Using \eqref{e1-thm4.2} and arguing by contradiction, similarly to the argument from Lemma \ref{lm4.9},
we can show that with probability 1,
$\U(\omega)$ is a subset of $\Conv(\M)\setminus\Conv(\M^2)$.
In other words,  each invariant probability $\pi\in\U(\omega)$ has the form
$\pi=(1-\rho)\pi_1+\rho\pi_2$
where $\rho\in[0,1), \pi_1\in\Conv(\M^1), \pi_2\in\Conv(\M^2)$.
Let
$\ell_0>1$ be sufficiently large that $\mu(\K^0_I)>0$ for all $\mu\in\Conv(\M^{I,\circ})$ where $$\K^0_I=\{\bx\in\R^I_+: x_i\wedge x_i^{-1}\leq \ell_0, i\in I\}.$$
Define
$$\K^{\ell,\Delta}_I:=\{\bx\in\R^{n,\circ}_+, \ell^{-1}\leq x_i\leq \ell\text{ for } i\in I, x_i<\Delta\text{ for }i\in I^c\}.$$

By Theorem \ref{thm4.1}, there are $\ell>\ell_0$ and $\Delta>0$ such that
\begin{equation}\label{e3-thm4.2}
\PP_{\bx,k}\left\{\lim_{t\to\infty} X_i(t)=0, i\in I^c\right\}>1-\eps
\end{equation}
for all $I\in \mathcal{S}$ and
$\bx\in\K_I^{\ell,\Delta}, k\in \CN.$ Moreover, since  $\bigcup_{I\in S} S^I_+$ is accessible we get by Theorem \ref{thm4.1} also that no invariant probability measure exists on $\R_+^{n,\circ}\times\CN.$

Let $\psi(\cdot):\R^n_+\to[0,1]$ be a continuous function satisfying
$$
\psi(\bx,k)=
\begin{cases}
1 \text{ if }\bx\in\bigcup_{I\in\mathcal{S}}\K_I^0\\
0 \text{ if }\bx\in\R^{n,\circ}_+\setminus \left(\bigcup_{I\in\M^1}\K_I^{\ell,\Delta}\right).
\end{cases}
$$

Since, by Lemma \ref{lm4.7}, $\pi_1(\bigcup_{I\in\mathcal{S}}\K_I^0)>0$ for any $\pi_1\in\Conv(\cup_{I\in \mathcal{S}}\M^{I,\circ})$
and $\U(\omega)$ is a subset of $\Conv(\M)\setminus\Conv(\M^2)$ with probability 1,
we have from Lemma \ref{lm4.6} that
\begin{equation}\label{e4-thm4.2}
\PP_{\bx,k}\left\{\liminf_{t\to\infty}\dfrac1t\int_0^t\psi(\BX(s),r(s))ds>0\right\}=1,\,\,\bx\in\R^{n,\circ}_+, k\in \CN.
\end{equation}
Since $\psi(\bx,k)=0$ if $\bx\in\R^{n,\circ}_+\setminus \left(\bigcup_{I\in\mathcal{S}}\K_I^{\ell,\Delta}\right)$,
we deduce from \eqref{e4-thm4.2} that
\begin{equation}\label{e5-thm4.2}
\PP_{\bx,k}\left\{\liminf_{t\to\infty}\dfrac1t\int_0^t\1_{\left\{\BX(s)\in \bigcup_{I\in\mathcal{S}}\K_I^{\ell,\Delta}\right\}}ds>0\right\}=1,\,\,\bx\in\R^{n,\circ}_+, k\in \CN.
\end{equation}
Thus, if $\BX(0)\in\R^{n,\circ}_+$ then $\BX(t)$ will enter $\bigcup_{I\in\mathcal{S}}\K_I^{\ell,\Delta}$ with probability 1.
This fact, combined with \eqref{e3-thm4.2}, the strong Markov property of $\{\BX(s)\}$, together with Lemmas \ref{lm4.5}, \ref{lm4.6} and \ref{lm4.9} implies
that
$$\sum_{I\in\mathcal{S}} P_{\bx,k}^I>1-\eps,\,\bx\in\R^{n,\circ}_+, k\in \CN$$
where
$$P_{\bx,k}^I=\PP_{\bx,k}\left\{\U(\omega)\subset\Conv(\M^{I,\circ})\text{ and }\lim_{t\to\infty}\dfrac{\ln X_i(t)}t\in\left\{\lambda_i(\mu),\mu\in\Conv(\M^{I,\circ})\right\}, i\in I_\mu^c\right\}.$$
Letting $\eps\downarrow 0$ we obtain \eqref{e0-thm4.2}.

Next, we consider the case when $\mathcal{S}^c=\emptyset$.
Then, we claim that $\mathcal{S}=\{0\}$.
Indeed, if there exists $I\in\mathcal S$  with $\R_+^I\ne\{\0\}$,
then $\bdelta^*\in\partial\M^I$, where $\bdelta^*$ be measure where $\BX(t)$ all equal $\0$.
Since $I$ satisfies Assumption \ref{a4-pdm}, in view of \eqref{ae3.2} , $\bdelta^*\in\mathcal{S}^c$ which results in a contradiction.
Thus, $\mathcal{S}=\{0\}$.
As a result, $\U(\omega)=\{\bdelta^*\}$ with probability 1.
Then, we can easily deduce with probability 1 that
$$\lim_{t\to\infty}\dfrac{\ln X_i(t)}t=\lambda_i(\bdelta^*)<0, i=1,\dots,n$$
since $\bdelta^*$ satisfies \eqref{ae3.1}.
\end{proof}

If every subspace has at most one attractor we get the following result.
 \begin{thm}\label{t:ex_one}
Suppose that Assumptions \ref{a1-pdm}, \ref{a:bound}, \ref{a4-pdm} and \ref{a5-pdm} are satisfied and $\M^1\neq \emptyset$. Suppose furthermore that $\bigcup_{I \in S} S^I_+$ is accessible and that every subset $I\in S$ is such that $\M^{I,\circ}=\{\mu_I\}$. Then for all $(\bx,k)\in \R_+^{n,\circ}$ we have 
$$
\sum_{I\in S} P_{\bx,k}^I=1
$$
where
$$P_{\bx,k}^I:=\PP_{\bx,k}\left\{\U(\omega)=\{\mu_I\}\,\text{ and }\,\lim_{t\to\infty}\dfrac{\ln X_i(t)}t=\lambda_i(\mu_I)<0, i\in I^c\right\}.$$ 
Moreover, if $\R_+^I $ is accessible from $(\bx,k)$ then $P_{\bx,k}^I>0$.

\end{thm}
\begin{proof}
This is an immediate corollary of Theorem \ref{thm4.2}.   
\end{proof}
\textbf{Acknowledgments:} The authors thank Dang Nguyen for helpful suggestions which led to an improved manuscript. A. Hening acknowledges generous
support from the NSF through the CAREER grant DMS-2339000.

\appendix

\section{Proofs of Lemmas in Section 4}\label{a:inv measures}

\textit{Proof of Lemma \ref{lm2.0} :} Since the $f_i$'s are locally Lipschitz we know from \cite{BBMZ15} that the sample paths are differential inclusions and that there exist unique pathwise strong solutions. By uniqueness if $X_i(0)=0$, then $X_i(t)=0, t\geq 0$ almost surely. Further, if we have $X_i(0)>0$ then 
\[
X_i(t)= X_i(0)e^{\int_0^t f_i(X(s),r(s)) ds}
\]
which shows the solutions stay positive up to a possible blow up time. Since $f_i$ are locally Lipschitz, and solutions are continuous, with a bootstrap argument we can see that for any finite $t$, $X_i(t)$ is finite. If there exists for any sample path and $t_0>0$ such that $X_i(t_0)=0$, then the equality above implies that $X_i(0)=0$ which would be a contradiction. Hence the for any $I\subset\bigl\{{1,....n}\bigr\}$,  $\R^{I,o}_+$ is invariant under the process. 
\\
Using the definition of the generator \eqref{e:gen} and \eqref{e:V} we note that 
\begin{equation}\label{e:LV}
\mathscr{L}V(x,k)=\left( \frac{\mathscr{L}F(x,k)}{F(x,k)}- \sum_{i\in \CN}p_if_i(x,k)\right) V (x,k)
\end{equation}
The bound  \eqref{e:sup} yields
\begin{equation}\label{A.1}
\mathscr{L}V(x,k)\leq HV  (x,k), ~x\in \R_+^{n,\circ}.   
\end{equation}
Observe that, the function $V$ blows up on $\partial\R^n_+$ and 
\begin{equation}\label{e:boundary_time}
\tau_0:= \inf\biggl\{ t>0:\min_i X_i(t)=0\biggr\} = \infty
\end{equation}
almost surely. For any $K>0$ define the stopping time 
\[
\tau_K=\inf\left\{ t>0:\|\BX(t)\|\geq  K \vee \frac{1}{K} \right\}
\]
Note that by the above non-blowup discussion and \eqref{e:boundary_time} we have 
\begin{equation}\label{e:tauk_blowup}
\PP_{\bx,k}\left(\lim_{K\to \infty} \tau_K = \infty\right)=1.
\end{equation}

Next, apply Dynkin's Formula to the function $V$ and the Feller process $(\BX(\tau_K\wedge t), r( \tau_K\wedge t))_{t\geq 0}$ we get, using \eqref{e:tauk_blowup}
and the inequality \eqref{A.1} we have
\[
\E_{x,k}(V(\BX(\tau_K\wedge t),r(\tau_K\wedge t)))\leq  V(x,k) + H\cdot\int_0^{\tau_K\wedge t}V(\BX(s),r(s)) ds
\]
Now using Gronwall's Inequality we get 
\[
\E_{x,k}(V(\BX(\tau_K\wedge t),r(\tau_K\wedge t)))\leq e^{H(\tau_K\wedge t)}V(\bx,k)\leq e^{Ht}V(\bx,k).
\]
Since the right hand side does not depend on $K$ we let $K\to\infty$. Using Fatou's Lemma and \eqref{e:tauk_blowup} we have 
\[
\E_{x,k}(V(\BX(t),r(t)))\leq e^{Ht}V(\bx,k)
\]
which proves the result.

\begin{proof}[Proof of Lemma \ref{lm2.2}:]
By the Assumption from \eqref{e:sup} we have that  
for every $k\in \CN$ 
\[
\limsup_{\|x\|\to\infty} \left( \frac{\mathscr{L}F(x,k)}{F(x,k)}+ \delta_0\max_{i} |f_i(x,k)| \right)< 0.
\]

As a result, there exists $M>0$ such that for any $x$ with $\|x\|>M$ and for all $k\in\CN$ we have
\[
\left( \frac{\mathscr{L}F(x,k)}{F(x,k)}+\gamma_b + \delta_0\max_i |f_i(x,k)| \right)<0.
\] 
Remembering that 
\[
H= \sup_{x\in \mathbf{R^n_+}} \left( \frac{\mathscr{L}F(x,k)}{F(x,k)}+\gamma_b+ \delta_0\max_{i} |f_i(x,k)| \right) < \infty,
\]
 we get
\[
\frac{\Lom(F(\bx,k))}{F(\bx,k)}\leq H-\gamma_b-\delta_0\cdot\max_i|f_i(\bx,k)|.
\]
This implies
\begin{equation}\label{A.2}
  \mathscr{L}F(x,k)\leq HF(x,k)\mathbf{1}_{\|x\|\leq M}-\gamma_b F(x,k)-\delta_0F(x,k)\cdot(\max_i |f_i(x,k)| ).
\end{equation}
Define $H_1=H\cdot(\sup_{\|x\|\leq M} F(\bx,k))$ and note that
\[
\mathscr{L}F(x,k)\leq H_1-\gamma_bF(\bx,k)-\delta_0F(x,k)\max_i |f_i(x,k)| 
\]

We define the following stopping time; 
\[
\tau_{\ell}=\inf\left(t>0: \|X(t)\|\geq \ell \right)
\]
We next apply Dynkin's formula and get 
\begin{equation}
\begin{split}
&\E_{x,k}(e^{\gamma_b(\tau_{\ell}\wedge t)} F(X(\tau_{\ell}\wedge t),r(\tau_{\ell}\wedge t)))=F(x,k)+\E_{x,k}\left(\int_0^{\tau_{\ell}\wedge t} e^{\gamma_bs}\left(\gamma_b F(X(s),r(s))+ \mathscr{L}F(X(s),r(s))\right)     ds\right).  
\end{split}    
\end{equation}
It follows from \eqref{A.2}
\begin{equation}
 \E_{x,k}(e^{\gamma_b(\tau_{\ell}\wedge t)} F(X(\tau_{\ell}\wedge t),r(\tau_{\ell}\wedge t)))\leq F(x,k) + \int_0^{\tau_{\ell}\wedge t} H_1e^{\gamma_bs} ds  .     
\end{equation}

Hence by taking $H_1'=H_1/\gamma_b$, we get 
\[
\E_{x,k}(e^{\gamma_b(\tau_{\ell}\wedge t)} F(X(\tau_{\ell}\wedge t),r(\tau_{\ell}\wedge t)))\leq F(x,k) + H_1'e^{\gamma_bt}  
\]
now since the right hand side does not depend on $\tau_{\ell}$, we let $\ell\to\infty$ then by Fatou's Lemma
we get 
\[
\E_{x,k}(e^{\gamma_bt}F(X(t),r(t)))\leq F(x,k)+H_1'e^{\gamma_bt}.
\]
Hence we have the desired inequality
\begin{equation}\label{A.3}
\E_{x,k}(F(X(t),r(t)))\leq H_1'+F(x,k)e^{-\gamma_bt}.    
\end{equation}

Applying Dynkin's Fromula again yields
\[
\E_{x,k}( F(X(\tau_{\ell}\wedge t),r(\tau_{\ell}\wedge t)))=F(x,k) + \E_{x,k}\left(\int_0^{\tau_{\ell}\wedge t} \mathscr{L}F(X(s),r(s)) ds\right)
\]
using \ref{A.2} we get
\[
\E_{x,k}( F(X(\tau_{\ell}\wedge t),r(\tau_{\ell}\wedge t)))\leq F(x,k) + \int_0^{\tau_{\ell}\wedge t}H_1' ds 
\]
\[
-\delta_0\cdot\E_{x,k}\left(\int_0^{\tau_k\wedge t} F(X(s),r(s))(1+\max_i|f_i(X(s),r(s))|) ds\right) 
\]
\[
\delta_0\cdot\E_{x,k}\left(\int_0^{\tau_k\wedge t} F(X(s),r(s))(1+\max_i|f_i(X(s),r(s))|) ds\right) \leq 2\cdot(H_1't+F(\bx,k))
\]

as a result, combining Fatou's Lemma and \ref{A.3}
\[
\delta_0\E_{x,k}\left(\int_0^{t} F(X(s),r(s))(1+\max_i|f_i(X(s),r(s))|) ds\right)\leq 2\cdot(F(x,k)+H_1't)
\]
define $H_2'=\max(2/\delta_0, 2H_1'/\delta_0) $
and we get the desired inequality,
\begin{equation}\label{A.4}
\E_{x,k}\left(\int_0^{t} F(X(s),r(s))(1+\max_i|f_i(X(s),r(s))|) ds\right)\leq H_2'(F(x,k)+t)    
\end{equation}

Using the above estimates one can modify Proposition 2.1 from \cite{BBMZ15}, or use a truncation argument combined with a modification of Theorem 2.9.3 from \cite{MAO}, to show that $(X(t),r(t))$ is a Markov-Feller process, in the sense that the semigroup maps $C_b(\R_+^n\times\CN)$ into itself.
 \end{proof}
\begin{rmk}\label{R.1}
    We can see that \ref{A.3} shows that for all $(\bx,k)\in\R^n_+\times\CN$ the measures $\PP^t_{\bx,k}$, defined as 
\[
\PP^t_{\bx,k}(A)= \E_{\bx,k}\left(\1_{(\BX(t),r(t))\in A} \right),\text{  
     }  \forall A\in\mathscr{B}(\R^n_+\times\CN)
\]
define a tight family $(\PP^t_{\bx,k})_{t\geq 0}$. To see this, let's assume that the measures are not tight, hence there exists $\epsilon_0>0\in(0,1)$ such that for any $n$, there exists a sequence of times $(t_n)_{n=1}^{\infty}$, with the property  $\lim_{n\to\infty}t_n\to\infty$ such that
\[
\PP_{\bx,k}^{t_n}\biggl\{(\bx,k)\notin (\mathbf{B}(0,n)\cap\R^n_+)\times\CN \biggr\} \geq \epsilon_0
\]
however from \ref{A.3}, and Markov Inequality we get that for every $n$
\[
\epsilon_0(1+n)\leq \E_{x,k}(F(X(t_n),r(t_n)))\leq H_1'+F(x,k)
\]
which is clearly a contradiction. 

\end{rmk} 
\begin{proof}[Proof of Lemma \ref{lm2.3}:]  If $\mu$ in an invariant measure for the process, then if $(x,k)\in supp(\mu)$ then we have
\[
\lim_{t\to\infty} \|\mathbf{P}^t_{x,k}(.)-\mu(.)\|_{TV}=0
\]
and hence as a result for any $g\in L^{\infty}(\R^n_+\times\CN)$
\[
\lim_{t\to\infty}\sum_{k\in\CN}\int_{\R^n_+} g(x,k) \mathbf{P}^t_{x,k}(dx,k)= \sum_{k\in\CN}\int_{\R^n_+} g(x,k) \mu(dx,k)
\]
this implies that 
\[
\lim_{t\to\infty} \frac{1}{t}\int_0^t\left( \sum_{k\in\CN}\int_{\R^n_+} g(x,k) \mathbf{P}^s_{x,k}(dx,k)         \right) ds =  \sum_{k\in\CN}\int_{\R^n_+} g(x,k) \mu(dx,k)
\]
observe that 
\[
\sum_{k\in\CN}\int_{\R^n_+} g(x,k) \mathbf{P}^s_{x,k}(dx,k)=\E_{\bx,k}(g(\BX(s),r(s)))
\]
hence using Fubini's theorem we get that 
\[
\sum_{k\in\CN}\int g(x,k) \Pi^t_{x,k}(dx,k)=\E_{x,k}\left(\frac{1}{t}\int_0^{t} g(\BX(s),r(s)) ds\right)=\frac{1}{t}\int_0^t \E_{\bx,k}(g(\BX(s),r(s))) ds
\]
since we know \eqref{A.4}, we get 

\begin{equation}\label{A.5}
\sum_{k\in\CN}\int F(x,k)(1+\max_i |f_i(x,k)|) \Pi^t_{x,k}(dx,k)\leq H_2'\left(1+\frac{F(x,k)}{t}\right)    
\end{equation}

Define $\psi(\bx,k)=F(\bx,k)(1+\max_i|f_i(\bx,k)|)$, take an arbitrary constant $K>0$ and use \eqref{A.5} to get
\[
\lim_{t\to\infty} \E_{x,k}\left(\frac{1}{t}\int_0^{t} (\psi(\BX(s),r(s))\wedge K) ds\right)
=\sum_{k\in\CN}\int (\psi(\bx,k)\wedge K) \mu(dx,k)
 \leq H_2'
\]
now we can let $K\to\infty$ and by Fatou's Lemma we get the desired inequality. 
\begin{equation}\label{intFonmu}
\sum_{k\in\CN}\int_{\R^n_+}F(x,k)(1+\max_i |f_i(x,k)|) \mu(dx,k) \leq H_2'     
\end{equation}

Observe that \ref{A.5} implies there exists a $T_0>0$, such that for all $t\geq T_0$, and for any $M_0>0$ such that $\|x\|\leq M_0$ the measures $\bigl\{\Pi^t_{x,k}\bigr\}$ are tight. Well to prove this, let's assume that they are not tight; then there exists $\epsilon>0$ such that for every compact subset $V_n=B(0,n)\cap \mathbf{R^n_+}$, there exists $t_n\geq T_0$ such that $\Pi^{t_n}_{x,k}(\mathbf{R^n_+}\setminus V_n)>\epsilon$. As a result 
\[
\sum_{k\in\CN}\int_{\mathbf{R^n_+}\setminus V_n} F(x,k) \Pi^{t_n}_{x,k}(dx,k) > c(1+n)\epsilon 
\]
this contradicts \ref{A.5}, hence the measures $\{\Pi^t_{\bx,k} \}_{t\geq T_0}$ are tight. Next, we want to show that for any invariant measure $\mu$ we have 
\[
\sum_{k\in\CN}\int_{\R_+^n} \left( \frac{\mathscr{L}F(x,k)}{F(x,k)}-\sum_{l\in\CN}q_{kl}(x,k)\left(\frac{F(x,l)}{F(x,k)}-1-\ln\frac{F(x,l)}{F(x,k)}\right)\right)\mu(dx,k) =0.
\]

A direct computation shows that
\[
\mathscr{L}(\ln(F(\bx,k))=\frac{\mathscr{L}F(\bx,k)}{F(\bx,k)}-\sum_{l\in\CN}q_{kl}(\bx)\left(\frac{F(\bx,l)}{F(\bx,k)}-\ln\left(\frac{F(\bx,l)}{F(\bx,k)}\right)-\ln(F(\bx,k))\right).
\]
Since we have $\sum_{l\in\CN}q_{kl}(\bx)=0$, this reduces to 
\begin{equation}\label{e:gen_lnF}
\mathscr{L}(\ln(F(\bx,k))=\frac{\mathscr{L}F(\bx,k)}{F(\bx,k)}-\sum_{l\in\CN}q_{kl}(\bx)\left(\frac{F(\bx,l)}{F(\bx,k)}-\ln\left(\frac{F(\bx,l)}{F(\bx,k)}\right)-1\right)
\end{equation}
Note that again by using  $\sum_{l\in\CN}q_{kl}(\bx)=0$
\[
\mathscr{L}(\ln(F(\bx,k))= \frac{\mathscr{L}F(\bx,k)}{F(\bx,k)}-\sum_{l\neq k}q_{kl}(\bx)\left(\frac{F(\bx,l)}{F(\bx,k)}-1  \right) +  \sum_{l\neq k}q_{kl}(\bx)\left(\ln\left(\frac{F(\bx,l)}{F(\bx,k)}\right) \right)
\]
Using $\ln(1+x)\leq x$ one gets
\begin{equation*}
\begin{split}
\mathscr{L}(\ln(F(\bx,k))&\leq  \frac{\mathscr{L}F(\bx,k)}{F(\bx,k)}-\sum_{l\neq k}q_{kl}(\bx)\left(\frac{F(\bx,l)}{F(\bx,k)}-1  \right) +  \sum_{l\neq k}q_{kl}(\bx)\left(\frac{F(\bx,l)}{F(\bx,k)}\right)\\
& = \frac{\mathscr{L}F(\bx,k)}{F(\bx,k)} + \sum_{l\neq k}q_{kl}(\bx).
\end{split}
\end{equation*}
Note that
\[
\max_{i,j}\sup_{(\bx,k)\in\R^n_+\times\CN}|q_{ij}(\bx)|\leq C_0 
\]
Using this in conjunction with \eqref{A.2}, yields
\[
\mathscr{L}(\ln(F(\bx,k))\leq  H-\gamma_b-\delta_0\cdot\max_i|f_i(\bx,k)| + |\CN|\cdot C_0 \leq H + |\CN|\cdot C_0
\]
Since the given function is bounded above, and assuming without loss of generality, that $\mu$ is an ergodic invariant measure, we can use continuous time version of Birkhoff's ergodic theorem from Corrolary 25.9 in \cite{Kberg2021modprob}. Hence for a.e $(\bx,k)\in supp(\mu)$, $(\BX(0),r(0))=(\bx,k)$ we get 
\begin{equation}\label{e:birk_L_lnF}
\lim_{t\to\infty}\frac{1}{t}\int_0^t \Lom(\ln(F(\BX(s),r(s)))) ds =\sum_{k\in\CN}\int_{\R^n_+} \mathscr{\Lom}(\ln(F(\bx,k)) \mu(dx,k).
\end{equation}

Let $N(dt,dy) $ be the Poisson measure with intensity $dt\times dy$ on $(0,\infty)\times \R_+^n$ and $ \tilde N(dt,dy) =N(dt,dy) - dt\times dy$ the compensated Poisson measure.
From the generalized Ito's formula (see \cite{YZ10}) we have 
\begin{equation}\label{e:ln_F}
\ln(F(\BX(t),r(t)))=\ln(F(\BX(0),r(0)))+\int_0^t  \Lom(\ln(F(\BX(s),r(s)))) ds  +  M(t)
\end{equation}
where $M(t)$ is a local martingale, given by 
$$M(t)=\int_0^t \int_{\R_+} \ln\left[\frac{F(\BX(s),r(s)+h(\BX(s),r(s),y))}{F(\BX(s),r(s))}\right]\tilde N(dt,dy)
$$
with $h(\bx,i,z)=\sum_{j=1}^{n_0} (j-i)I_{\{z\in \Delta_{ij}(x)\}}$. Here, for $i\neq j$ we define $\Delta_{ij}(\bx)$ as the consecutive left-closed, right-open intervals of the real line, each having length $q_{ij}(\bx)$ (see \cite{YZ10}, page 29). Note that $h(\BX(t), r(t), z) = j-r(t)$ if $z \in \Delta_{r(t),j} (\BX(t)) $ for some $j$
and $0$ otherwise.

We next want to prove that $\lim_{t\to\infty}\frac{M(t)}{t}=0$. We first show the following.

\begin{clm}\label{c:a1}
    
$M(t)$ is an $L^2$ martingale.
\end{clm}
\begin{proof}
    
 We know that $M(t)$ is a local martingale \cite{YZ10} Consider a sequence of stopping times $\tau_k$ such that $\lim_{k\to\infty}\tau_k\to\infty$ and $\tau_{k+1}>\tau_k$ almost surely for all $k$.

\[
\E_{\bx,k}(|M(t\wedge \tau_k)|^2)=\E_{\bx,k}\left(\left|\int_0^{t\wedge\tau_k}\int_{\R_+} \left(\ln(F(\BX(s),r(s)+h(\BX(s),r(s),y))-\ln(F(\BX(s),r(s)))\right) \tilde{N}(dy,ds)\right|^2 \right)
\]
Using Ito's isometry, see Proposition 2.4 from \cite{kunita2010ito}, we have 
\begin{equation}\label{e:M_isom}
\E_{\bx,k}(|M(t\wedge \tau_k)|^2)= \E_{\bx,k}\left(\int_0^{t\wedge\tau_k}\int_{\R_+}|\ln(F(\BX(s),r(0)+h(\BX(s),r(s),y))-\ln(F(\BX(s),r(s)))|^2 dyds\right)
\end{equation}
By Assumption \ref{a1-pdm} the transition rates are are uniformly bounded over $\R^n_+\times\CN$. As a result there is $C_0>0$ such that 
\begin{equation}\label{bound_rates}
0<\max_{i\neq j}\sup_{(\bx,k)\in\R^n_+\times\CN}q_{ij}(\bx)\leq C_0.
\end{equation}
Observe that from the definition of $h$ we get that 
\begin{equation}\label{jump_int}
\int_0^{\infty}\left|\ln\left(\frac{F(\BX(s),r(s)+h(\BX(s),r(s),y))}{\ln(F(\BX(s),r(s))}\right)\right|^2 dy = \sum_{j\neq r(s)}\left|\ln\left(\frac{F(\BX(s),j)}{F(\BX(s),r(s))}\right)\right|^2q_{r(s),j}(\BX(s))
\end{equation}
Since $|\ln(x)|=|-\ln(1/x)|=|\ln(1/x)|$ we can get from \eqref{e:M_isom} and \eqref{jump_int} that
\begin{equation}\label{L2mod}
\E_{\bx,k}(|M(t\wedge \tau_k)|^2)=\E_{\bx,k}\left(\int_0^{t\wedge\tau_k}\left(\sum_{j\neq r(s)}\left|\ln\left(\frac{F(\BX(s),r(s))}{F(\BX(s),j)}\right)\right|^2q_{r(s),j}(\BX(s))\right)  ds   \right)
\end{equation}
Using Assumption \ref{a:bound} and \eqref{bound_rates} and setting $\bar M = |\ln M_F|$ yields
\begin{equation}\label{e:ML2}
\begin{split}
\E_{\bx,k}(|M(t\wedge \tau_k)|^2)=&\E_{\bx,k}\left(\int_0^{t\wedge\tau_k}\left(\sum_{j\neq r(s)}\left|\ln\left(\frac{F(\BX(s),r(s))}{F(\BX(s),j)}\right)\right|^2q_{r(s),j}(\BX(s))\right)\,ds  \right)\\
&\leq C_0 \bar M ^2|\CN| t
\end{split}    
\end{equation}
\end{proof}

\begin{clm}\label{c:a2}
For almost every $(\bx, k)$ in the support of $\mu$ we have 

\[
\PP_{\bx,k}\left(\lim_{t\to\infty} \frac{M_t}{t}=0\right)=1.
\]
\end{clm}
\begin{proof} For the martingale $M(t)$, by Claim \ref{c:a1}, we can use \cite{kunita2010ito} Proposition 2.4 to get the quadratic variation 
\[
[M,M]_t=\int_0^t\int_0^{\infty}\left( \ln\left[ \frac{F(\BX(s),r(s)+h(\BX(s),r(s),y))}{F(\BX(s),r(s))}\right]\right)^2 N(dy,ds)
\]
as well as the predictable quadratic variation
\[
\langle M,M\rangle _t=\int_0^t\int_0^{\infty}\left( \ln\left[ \frac{F(\BX(s),r(s)+h(\BX(s),r(s),y))}{F(\BX(s),r(s))}\right]\right)^2 dyds.
\]

By \eqref{jump_int}, and the property that $|\ln(x)|=|\ln(1/x)|$ we get
\begin{equation}\label{e:QV}
\langle M,M\rangle _t=\int_0^t\left( \sum_{j\neq r(s)}q_{r(s)j}(\BX(s))\left(\ln\left[\frac{F(\BX(s),r(s))}{F(\BX(s),j)}\right]       \right)^2 \right) ds  
\end{equation}
Then by the computations from \eqref{e:ML2}, which can be done for \eqref{e:QV} just as they were done for \eqref{L2mod}, we get

\begin{equation}\label{e:QVM}
\langle M,M\rangle _t \leq C_0 \bar M^2|\CN| t
\end{equation}
We finish the proof as follows using an argument from \cite{B23}.  For any $n\in\N$ and $\eps>0$ define the event $$E_{n,\eps}:= \left\{\sup_{t\in [2^n, 2^{n+1}]}\frac{|M(t)|}{t}\geq \eps\right\}.$$
As $M(t)$ is a right continuous martingale we can use Doob's inequality to get
\begin{equation}\label{e:doob}
    \begin{split}
    \PP_{\bx,k}(E_{n,\eps}) &= \PP_{\bx,k} \left( \sup_{t\in [2^n, 2^{n+1}]}\frac{|M(t)|}{t}\geq \eps\right) \\
    &\leq   \PP_{\bx,k} \left( \sup_{t\leq  2^{n+1}}|M(t)|\geq 2^n \eps\right) \\
        &\leq \frac{1}{\eps^2 2^{2n}}\E_{\bx,k}\langle M,M\rangle_{2^{n+1}}\\
        &\leq \frac{1}{\eps^2 2^{2n}} C_0 \bar M^2|\CN| 2^{n+1}\\
        &= \frac{ 2C_0 \bar M^2|\CN| }{\eps^2} \frac{1}{2^n}
    \end{split}
\end{equation}
where the last inequality comes from \eqref{e:QVM}.
Since 
\[
 \sum_{n\geq 1}\frac{1}{2^n} <\infty
\]
we get that $\sum_{n\geq 1} \PP_{\bx,k}(E_{n,\eps})<\infty$. By Borel-Cantelli 
\begin{equation}\label{e:BC}
\PP_{\bx,k}\{\omega: \omega \in E_{n,\eps} ~\text{i.o.}\} = 0.
\end{equation}
Since this holds for arbitrarily small $\eps>0$ we get that
\[
\PP_{\bx,k}\left(\limsup_{t\to\infty} \frac{|M(t)|}{t}=0    \right)
=1\]
which then clearly implies that
\[
\PP_{\bx,k}\left(\lim_{t\to\infty} \frac{M(t)}{t}=0    \right)
=1.\]
This finishes the proof of the claim.

\end{proof}

Let $\Lambda_1$ and $\Lambda_2$ be subsets of $supp(\mu)$ defined as
\[
\Lambda_1=\left( (\bx,k)=(\BX(0),r(0))\in supp(\mu) : \lim_{t\to\infty}\frac{1}{t}\int_0^t \Lom(\ln(F(\BX(t),r(t)))) ds=\sum_{k\in\CN}\int_{\R^n_+} \mathscr{\Lom}(\ln(F(\bx,k)) \mu(dx,k)\right)
\]
\[
\Lambda_2=\left( (\bx,k)=(\BX(0),r(0))\in supp(\mu) : \lim_{t\to\infty}\frac{1}{t}\int_0^t (F(\BX(s),r(s))^{\delta} ds=\sum_{k\in\CN}\int_{\R^n_+} (F(\bx,k))^{\delta} \mu(dx,k)\right)
\]
Since $\mu(\Lambda_1)=\mu(\Lambda_2)=1$, taking $(\BX(0),r(0))\in\Lambda_1\cap\Lambda_2$, and note that by \eqref{e:ln_F} we have
\[
\frac{\ln(F(\BX(t),r(t)))}{t}=\frac{\ln(F(\BX(\bx,k))}{t}+\frac{1}{t}\int_0^t  \Lom(\ln(F(\BX(s),r(s)))) ds  +  \frac{M(t)}{t}
\]

Letting $t\to \infty$ we get by  Claims \ref{c:a1}, \ref{c:a2} and equation \eqref{e:birk_L_lnF} that there exists $\ell\geq 0$ such that almost surely
\begin{equation}\label{e:lnF}
\lim_{t\to\infty} \frac{\ln(F(\BX(t),r(t)))}{t} = \sum_{k\in\CN}\int_{\R^n_+} \mathscr{\Lom}(\ln(F(\bx,k)) \mu(dx,k)=\ell\geq 0.
\end{equation}

We claim that $\ell=0$. We argue by contradiction. Suppose $\ell>0$. Then
\begin{equation}\label{e:F1}
\lim_{t\to\infty} \frac{\ln(F(\BX(t),r(t)))}{t} = \ell>0
\end{equation}

By the Birkhoff ergodic theorem we have that the process $\BX(t)$ has to visit a given compact subset $\K$, where $x_i\geq \gamma>0$ for all $\bx\in\K$, and $\mu(\K)\geq 1/2$,  of the state space infinitely often for $\mu$ almost all $(\bx,k)$, $\PP_{\bx,k}$ almost surely. Suppose $t_n(\omega)$ are increasing such that $t_n\uparrow \infty$ and such that $(X(t_n))\in \K$. This will contradict \eqref{e:F1} as $\ln F$ is bounded on $\K$. As such $\ell=0$. This together with \eqref{e:gen_lnF} finishes the proof of the first part of the lemma.

Using similar arguments one can show that

\begin{equation}\label{e:lambda}
\lim_{t\to \infty } \frac{\ln X_i(t) }{t}=\lambda_i(\mu).
\end{equation}

By the Birkhoff ergodic theorem we have that the process $\BX(t)$ has to visit a given compact subset $\K$, where $x_i\geq \gamma>0$ for all $\bx\in\K$ and $\mu(\K)\geq 1/2$,  of the state space infinitely often for $\mu$ almost all $(\bx,k)$, $\PP_{\bx,k}$ almost surely. 

This implies that almost surely
\[
\lim_{t\to \infty } \frac{\ln X_i(t) }{t} = \lambda_i(\mu)=0,
\]
as otherwise $X_i(t)\to\infty$ or $X_i(t)\to 0$, contradicting that $\BX(t)$ visits $\K$ infinitely often.

\end{proof}

\textit{Proof of Lemma \ref{lm2.4}:} From the assumptions and \ref{A.4} we get;
\[
\sum_{k\in\CN}\int |h(x,k)| \Pi^{x_k,k_m}_{T_m}(dx,k) \leq H_2'(1+\max_{k\in\CN}\sup_{\|x\|\leq M}F(x,k))<\infty
\]
Consider continuous functions $\phi_j$ such that $\phi_j(x,k)=1$ for $x\in B(0,j)$ and $\phi_j(x,k)=0$ for $x\in B(0,2j)^c$, and $0\leq \phi_j\leq 1$.
then we have
\begin{equation}
\begin{split}
&\limsup_{m\to\infty}\left|\sum_{k\in\CN}\int h(x,k) \Pi^{x_k,k_m}_{T_m}(dx,k)- \sum_{k\in\CN}\int h(x,k)\mu(dx,k)\right|\\
&\leq \lim_{m\to\infty} \left|\sum_{k\in\CN}\int h(x,k)\phi_j(x,k) \Pi^{x_k,k_m}_{T_m}(dx,k)- \sum_{k\in\CN}\int h(x,k)\phi_j(x,k)\mu(dx,k)\right|\\
&+\limsup_{m\to\infty}\sum_{k\in\CN}\int |h(x,k)|(1-\phi_j(x,k)) \Pi^{x_k,k_m}_{T_m}(dx,k)+\sum_{k\in\CN}\int |h(x,k)|(1-\phi_j(x,k))\mu(dx,k)
\end{split}    
\end{equation}
for $\epsilon>0$ we choose $j$ to be large enough so that $\sum_{k\in\CN}\int |h(x,k)|(1-\phi_{j}(x,k))\mu(dx,k)<\epsilon$ 
and
\[
\sum_{k\in\CN}\int |h(x,k)|(1-\phi_{j}(x,k)) \Pi^{x_k,k_m}_{T_m}(dx,k)<\epsilon
\]
this follows since we have $|h(x,k)|(1-\phi_j(x,k))< K_h\frac{F(x,k)}{c^{1-\delta}(1+\|j\|)^{1-\delta}}(1+\max_i |f_i(x,k)|)$
hence the integrals are uniformly small hence we have for every $\epsilon>0$
\[
\limsup_{m\to\infty}\left|\sum_{k\in\CN}\int h(x,k) \Pi^{x_k,k_m}_{T_m}(dx,k)- \sum_{k\in\CN}\int h(x,k)\mu(dx,k)\right|\leq 2\epsilon
\]
hence the result is proved. 
\\
\\
\textit{Proof of Lemma \ref{lm2.5} :} This has been proven in \cite{B23, HN18}.\\
\\
\section{Proof of Lemmas in Section 6}\label{extinction lemmas}

\begin{proof}[Proof of Lemma \ref{lm4.4}:] 
\[
\Lom(F(\bx,k))^{\delta_1}=\delta_1(F(\bx,k))^{\delta_1}\left(\sum_{i=1}^n\frac{x_if_i(\bx,k)}{F(\bx,k)}\frac{\partial F}{\partial x_i}(\bx,k)+ \dfrac{1}{\delta_1}\sum_{l\in\CN}q_{kl}(\bx)\left(\frac{F(\bx,l)}{F(\bx,k)}\right)^{\delta_1}\right)
\]
\[
\Lom(F(\bx,k))^{\delta_1}=\delta_1(F(\bx,k))^{\delta_1}\left(\frac{\Lom(F(\bx,k))}{F(\bx,k)}-\sum_{l\in\CN}\frac{q_{kl}(\bx)F(\bx,l)}{F(\bx,k)}+ \dfrac{1}{\delta_1}\sum_{l\in\CN}q_{kl}(\bx)\left(\frac{F(\bx,l)}{F(\bx,k)}\right)^{\delta_1}\right)
\]
\begin{equation}
\Lom((F(\bx,k))^{\delta_1})=(F(\bx,k))^{\delta_1}\left(\delta_1\left(\frac{\Lom(F(\bx,k))}{F(\bx,k)}\right)-\delta_1\sum_{l\in\CN}\frac{q_{kl}(\bx)F(\bx,l)}{F(\bx,k)}+ \sum_{l\in\CN}q_{kl}(\bx)\left(\frac{F(\bx,l)}{F(\bx,k)}\right)^{\delta_1}  \right)     
\end{equation}
We simplify the above equation to
\begin{equation}
\begin{split}
\frac{\Lom((F(\bx,k))^{\delta_1})}{(F(\bx,k))^{\delta_1}}=&\delta_1\left(\frac{\Lom(F(\bx,k))}{F(\bx,k)}\right)-\delta_1\sum_{l\neq k}q_{kl}(\bx)\left(\frac{F(\bx,l)}{F(\bx,k)}-1\right)\\
+&\sum_{l\neq k}q_{kl}(\bx)\left(\left(\frac{F(\bx,l)}{F(\bx,k)}\right)^{\delta_1}-1\right).
\end{split}    
\end{equation}
Now using \eqref{e:ineq_theta}, since $\delta_1<\delta_0<1$, and \eqref{e:sup2} we get that 
\begin{equation}\label{B.1}
 \Lom((F(\bx,k))^{\delta_1})\leq \delta_1H(\sup_{\|x\|\leq M}(F(\bx,k))^{\delta}) -\delta_1\gamma_b(F(\bx,k)^{\delta_1}(1+\max_i\{|f_i(\bx,k)|\})   
\end{equation}

Now using generalized Ito's Formula from \cite{YZ10}, we have 
\begin{equation}\label{B.2}
0\leq (F(\BX(t),r(t)))^{\delta_1}=(F(\BX(0),r(0)))^{\delta_1}+ \int_0^t \Lom(F(\BX(s),r(s))^{\delta_1}) ds + M(t).   
\end{equation} 
$M(t)$ is local martingale given by 
\begin{equation}
 M(t)= \int_0^t\int_{\R_+} \left((F(X(s),r(s)+h(X(s),r(s),y)))^{\delta_1}-(F(\BX(s),r(s)))^{\delta_1}\right)  \tilde{N}(dy,ds)
 \end{equation}
 
where $\tilde{N}(dy,ds)$ is the compensated Poisson random measure given by  $\tilde{N}(ds,dy)=N(ds,dy)-ds\times dy$.\\
If we combine \eqref{B.1} and \eqref{B.2} while taking $(X(0),r(0))=(\bx,k)$, we get
\begin{equation}\label{B.3}
\int_0^t (F(\BX(s),r(s))^{\delta_1}(1+\max_i\{|f_i(\BX(s),r(s))|\}) ds\leq \frac{F(\bx,k)}{\delta_1\gamma_b} +\frac{M(t)}{\delta_1\gamma_b}+H^*t.     
\end{equation}
where $\displaystyle H^*=\frac{H}{\gamma_b}\left(\sup_{\|x\|\leq M}(F(\bx,k))^{\delta})\right)$.

\begin{clm}\label{c:b1} If $\delta_1<\frac{1}{2}$, then  $M(t)$ is an $L^2$ Martingale
\end{clm}
\begin{proof}
 Let $(\tau_k)_{k=1}^{\infty}$ be a sequence of stopping times, such that $\tau_k<\tau_{k+1}$ and $\lim_{k\to\infty}\tau_k\to\infty$ almost surely for all $k$. 

\begin{equation}
\begin{split}
&\E_{\bx,k}\left(|M(t\wedge\tau_k)|^2\right)\\
&=\E_{\bx,k}\left(\left|\int_0^{\tau_k\wedge t}\int_{\R_+} \left((F(X(s),r(s)+h(X(s),r(s),y)))^{\delta_1}-(F(\BX(s),r(s)))^{\delta_1}\right) \tilde{N}(dy,ds)\right|^2\right).
\end{split}
\end{equation}
Using Ito's Isometry, from \cite{kunita2010ito}, we have 
\begin{equation}\label{isodelta1}
\E_{\bx,k}\left(|M(t\wedge\tau_k)|^2\right)=\E_{\bx,k}\left(\int_0^{\tau_k\wedge t}\int_{\R_+} \left|(F(\BX(s),r(s)+h(X(s),r(s),y)))^{\delta_1}-(F(\BX(s),r(s)))^{\delta_1}\right|^2 dyds\right).   
\end{equation}
We proceed identically as in \eqref{jump_int} to rewrite \eqref{isodelta1} in the following form
\begin{equation}
 \E_{\bx,k}\left(|M(t\wedge\tau_k)|^2\right)=\E_{\bx,k}\left(\int_0^{\tau_k\wedge t}\left(\sum_{j\neq r(s)} \left|(F(\BX(s),j))^{\delta_1}-(F(\BX(s),r(s)))^{\delta_1}\right|^2q_{r(s)j}(\BX(s))\right) ds\right).   
\end{equation}
Using assumption \ref{a:bound} we can see that if $M_{F,\delta}=M_F^\delta$ we get
\begin{equation}\label{e:M2_B}
\begin{split}
 \E_{\bx,k}\left(|M(t\wedge\tau_k)|^2\right)&=\E_{\bx,k}\left(\int_0^{\tau_k\wedge t}\left(\sum_{j\neq r(s)} \left|(F(\BX(s),j))^{\delta_1}-(F(\BX(s),r(s)))^{\delta_1}\right|^2q_{r(s)j}(\BX(s))\right) ds\right)\\
 &=\E_{\bx,k}\left(\int_0^{\tau_k\wedge t}\left(\sum_{j\neq r(s)} (F(\BX(s),r(s)))^{2\delta_1}\left|\frac{(F(\BX(s),j))^{\delta_1}}{(F(\BX(s),r(s)))^{\delta_1}}-1\right|^2q_{r(s)j}(\BX(s))\right) ds\right)\\
 &\leq C_0|\CN|(M_{F,\delta}-1)^2 \E_{\bx,k} \int_0^t(F(\BX(s),r(s)))^{2\delta_1}\,ds
\end{split}
\end{equation}
If $2\delta_1<1$ by Lemma \eqref{lm2.2} note that
\begin{equation}
\E_{\bx,k}\int_0^t(F(\BX(s),r(s)))^{2\delta_1}\,ds \leq H_2(F(\bx,k)+t)
\end{equation}

This together with \eqref{e:M2_B} yields
\begin{equation}\label{e:M2_B2}
\E_{\bx,k}\left(|M(t\wedge\tau_k)|^2\right)\leq C_0|\CN|(M_{F,\delta}-1)^2 H_2(F(\bx,k)+t)
\end{equation}

\end{proof}
 
\begin{clm}\label{c:B2} For all $(\bx,k)\in \R^n_+\times\CN$ we have 

\[
\PP_{\bx,k}\left(\lim_{t\to\infty} \frac{M(t)}{t}=0\right)=1.       
\]

\end{clm}
\begin{proof}

The predictable quadratic variation of $M(t)$ is given by 
\begin{equation}\label{}
\langle M,M\rangle_t= \int_0^t\int_{\R_+} \left|(F(X(s),r(s)+h(X(s),r(s),y)))^{\delta_1}-(F(\BX(s),r(s)))^{\delta_1}\right|^2  dy ds
\end{equation}

We simplify the above equation to 
\begin{equation}\label{qvarfull}
\begin{split}
\langle M,M\rangle_t=& \int_0^t \sum_{j\neq r(s)}q_{r(s)j}(\BX(s)) \left|(F(X(s),j))^{\delta_1}-(F(\BX(s),r(s)))^{\delta_1}\right|^2   ds.
\end{split}
\end{equation}
By calculations identical to those done to get \eqref{e:M2_B2} we will get that if $2\delta_1<1$ then
\begin{equation}\label{e:QVM_B2}
\E_{\bx,k}\left(\langle M,M\rangle_t\right)\leq C_0|\CN|(M_{F,\delta}-1)^2 H_2(F(\bx,k)+t)
\end{equation}
This implies that for fixed $\bx, k$ if $t$ is large enough then
\[
\E_{\bx,k}\left(\langle M,M\rangle_t\right)\leq C_{\bx,k}t 
\]
for some $C_{\bx,k}>0$. We can then use Doob's inequality and Borel-Cantelli as in \eqref{e:doob} and \eqref{e:BC} to get that 
\[
\PP_{\bx,k}\left(\lim_{t\to\infty} \frac{M(t)}{t}=0\right)=1 
\]
\end{proof}

Now from \eqref{B.3} we get 
\begin{equation}
\frac{1}{t}\int_0^t (F(\BX(s),r(s))^{\delta_1}(1+\max_i\{|f_i(\BX(s),r(s))|\}) ds\leq \frac{F(\bx,k)}{t\delta_1\gamma_b} +\frac{M(t)}{t\delta_1\gamma_b}+H^*.      
\end{equation}
By Claim \ref{c:B2} and letting $t\to \infty$ in the above we conclude that for all $(\bx,k)\in\R^n\times\CN$
\[
\PP_{\bx,k}\left(\limsup_{t\to\infty}\frac{1}{t}\int_0^t (F(\BX(s),r(s))^{\delta_1}(1+\max_i\{|f_i(\BX(s),r(s))|\}) ds \leq H^{*}\right)=1
\]
This proves the first part of the Lemma.

Observe that if $\|\bx\|\geq (2c\tilde{K})^{1/\delta_1}$ then clearly $(F(\bx,k))^{\delta_1}\geq 2\tilde{K}$. Therefore, if we let $\tilde{K}>H^*$, then we have for all $(\bx,k)$ such that $(\BX(0),r(0))=(\bx,k)$ that
\begin{equation}
\PP_{\bx,k}\left(\limsup_{t\to\infty}\frac{1}{t}\int_0^t \1_{\|\BX(s)\|> (2c\tilde{K})^{1/\delta_1}} ds \leq \frac{1}{t}\int_0^t \frac{(F(\BX(s),r(s))^{\delta_1}}{2\tilde{K}} ds \leq 1/2\right)=1   
\end{equation}
which is equivalent to 
\begin{equation}
\PP_{\bx,k}\left(\limsup_{t\to\infty}\frac{1}{t}\int_0^t \1_{\|\BX(s)\|\leq (2c\tilde{K})^{1/\delta_1}} ds \leq \frac{1}{t}\int_0^t \frac{(F(\BX(s),r(s))^{\delta_1}}{2\tilde{K}} ds \geq 1/2\right)=1   
\end{equation}
and the proof of the lemma is complete.

\end{proof}

\bibliographystyle{plainnat}
\bibliography{Kolmogorov}
\end{document}